\documentclass{amsart}
\usepackage{amsmath, amssymb, amsfonts, amscd, amsthm, mathrsfs}
\usepackage{tikz}
\usetikzlibrary{matrix,arrows}
\usepackage{comment}

\renewcommand{\hat}{\widehat}

\renewcommand{\hat}{\widehat}

\renewcommand{\bar}{\overline}
\newcommand{\pscal}[1]{\langle #1 \rangle}
\newcommand{\pfloor}[1]{\lfloor #1 \rfloor}

\newcommand{\0}{\mathbf{0}}
\newcommand{\A}{\mathbf{A}}

\newcommand{\C}{\mathbf{C}}
\newcommand{\D}{\mathbf{D}}
\newcommand{\F}{\mathbf{F}}
\newcommand{\Q}{\mathbf{Q}}
\newcommand{\Z}{\mathbf{Z}}
\newcommand{\R}{\mathbf{R}}
\newcommand{\G}{\mathbf{G}}
\newcommand{\M}{\mathbf{M}}
\newcommand{\N}{\mathbf{N}}
\newcommand{\T}{\mathbf{T}}

\newcommand{\cF}{\mathcal{F}}
\newcommand{\cG}{\mathcal{G}}

\newcommand{\cM}{\mathcal{M}}

\newcommand{\cU}{\mathcal{U}}

\newcommand{\cD}{\mathcal{D}}

\newcommand{\cA}{\mathscr{A}}
\newcommand{\ccB}{\mathscr{B}}
\newcommand{\cC}{\mathscr{C}}
\newcommand{\cT}{\mathscr{T}}
\newcommand{\ccN}{\mathscr{N}}

\newcommand{\Tr}{\mathrm{Tr}}
\newcommand{\ord}{\mathrm{ord}}
\newcommand{\diag}{\mathrm{diag}}
\newcommand{\NP }{\mathrm{NP}}
\newcommand{\HP }{\mathrm{HP}}
\newcommand{\GNP }{\mathrm{GNP}}
\newcommand{\sgn}{\mathrm{sgn}}

\newcommand{\lcm}{\mathrm{lcm}}
\newcommand{\Mat}{\mathrm{Mat}}
\newcommand{\Vol}{\mathrm{Vol}}
\newcommand{\Zeta}{\mathrm{Zeta}}
\newcommand{\SL}{\mathrm{SL}}
\newcommand{\disc}{\mathrm{disc}}
\newcommand{\Gal}{\mathrm{Gal}}

\renewcommand{\Vert}{\mathrm{Vert}}

\newcommand{\coeff}{\mathrm{coeff}}

\newcommand{\Norm}{\mathrm{Norm}}
\newcommand{\conv}{\mathrm{conv}}

\theoremstyle{plain}
\newtheorem{theorem}{Theorem}[section]
\newtheorem{proposition}[theorem]{Proposition}
\newtheorem{lemma}[theorem]{Lemma}

\newtheorem{corollary}[theorem]{Corollary}

\theoremstyle{definition}

\newtheorem{remark}[theorem]{Remark}
\newtheorem{examples}[theorem]{Examples}

\newtheorem{notations}[theorem]{Notations}

\begin{document}

\title[$L$-functions of exponential sums]
{Asymptotic variations of $L$-functions of exponential sums}
\subjclass[2000]{11,14}

\author{Hui June Zhu}
\address{
Department of mathematics
SUNY at Buffalo,
Buffalo, NY 14260}
\email{hjzhu@buffalo.edu}
\date{\today}

\begin{abstract}
Let $\Delta$ be an integral convex polytope of dimension $n$ in $\R^n$ that
contains the origin.
For every Laurent polynomial $f$ with coefficients in $\bar\Q$
regular with respect to its Newton polytope $\Delta$,
let $L^*(\bar{f};T)$ be the $L$-function of exponential sums of
the reduction $\bar{f}$ of $f$ at a prime of $\bar\Q$.
Then $L^*(\bar{f};T)^{(-1)^{n-1}}$ is a polynomial, and we
denote by $\NP(\bar{f})$ the normalized Newton polygon of this polynomial.
It has an absolute combinatorial lower convex bound by $\HP(\Delta)$ depending
only on $\Delta$.
Suppose $\Delta$ is simplicial at all origin-less facets and it contains
$J\cup V$, where  $J$ is a subset of $\Delta\cap\Z^n$
not intersects with any origin-less facets of $\Delta$ and
$V$ is a disjoint subset of $\Delta\cap\Z^n$
containing all non-origin vertices of $\Delta$ such
that $J\cup \Vert(\Delta)$ generates the monoid $C(\Delta)\cap\Z^n$ up to finitely many points,
where $C(\Delta)$ is the cone of $\Delta$.
Let $\A_V^J$ be the space of all Laurent polynomials
$f=\sum_{j\in J}a_j x^j +\sum_{j\in V}c_j x^j$
with parameters $(a_j)_{j\in J}$ and with prescribed $(c_j)_{j\in V}$ in $\Q$,
and let $\GNP(\A_V^J,\bar\F_p):=\inf_{\bar{f}}\NP(\bar{f})$
where $\bar{f}$ ranges over all regular polynomials in $\A_V^J(\bar\F_p)$.
Then we prove that there exists a Zariski dense open subset $\cU$
defined over $\Q$ in $\A_V^J$ such that
for every $f\in \cU(\bar\Q)$ and for every prime $p$ large enough
we have
$\NP(\bar{f})=\GNP(\A_V^J,\bar\F_p)$ and
$
\lim_{p\rightarrow \infty} \NP(\bar{f}) = \HP(\Delta).
$
These results have immediate application to the zeta function of
the toric Artin-Schreier varieties defined by $y^p-y = \bar{f}$.

We also prove the following theorem:
Let $\Delta$ be an integral convex polytope of dimension $n$ in $\R^n$ containing 
the origin.
Let $\T_\Delta$ be the space of all affine toric hypersurfaces
$V(f)$ defined by $f=0$ where 
$f=\sum_{j\in \Delta\cap\Z^n}a_j x^j$ parametered by
$(a_j)_{j\in \Delta\cap\Z^n}$'s where $a_j\neq 0$ for vertices $j$.
Let $\HP(\T_\Delta)$ be the Hodge polygon of toric hypersurfaces in $\T_\Delta$.
For any regular $V(f) \in \T_\Delta(\bar\Q)$,
let $V(\bar{f})$ be the reduction of $V(f)$ at a prime of $\bar\Q$ over $p$, and
let $\NP(V(\bar{f}))$ denote the normalized $p$-adic Newton polygon of
the key polynomial component of zeta function of $V(\bar{f})$.
Let $\GNP(\T_\Delta,\bar\F_p):=\inf_{\bar{f}} \NP(V(\bar{f}))$
where $V(\bar{f})$ ranges over $\T_\Delta(\bar\F_p)$. 
If $\Delta\cap\Z^n$ has a unimodular triangulation, then we have
for all $p$ large enough that
$
\GNP(\T_\Delta,\bar\F_p) = \HP(\T_\Delta).
$
This paper proves two conjectures of Wan.
\end{abstract}

\maketitle


\section{Introduction}
\label{S:intro}

In this paper we consider only integral convex polytope $\Delta$ in $\R^n$
for some $n\geq 1$ that contains the origin $\0$.
For any $j:=(j_1,\ldots,j_n)\in \Z^n$ in standard basis of $\Z^n$
we write $x^j := x_1^{j_1}\cdots x_n^{j_n}$.
A Laurent polynomial $f$ with {\em Newton polytope} $\Delta$
is a function $f(x) = \sum_{j\in \Sigma} a_j x^j$ with $a_j\neq 0$
for some finite set $\Sigma$ in $\Z^n$ such that the
convex hull of $\Sigma\cup \0$ in $\R^n$ is
$\Delta$.

The object of our study is the
family of all Laurent polynomials with a prescribed Newton polytope $\Delta$ with coefficients in
a global field (e.g., $\Q$) and
the $L$ function of exponential sums of this family at all special fibres
(e.g.,  at all primes $p$).
Solutions to families of polynomial equations with given Newton polytope
is a central object of study in toric geometry (see survey article \cite{Stu98} and \cite{Stu98}).
There have been two separate lines of developments.
One classical result in toric geometry is that
the Newton polytopes controls valuations of
roots of the system of equations $f_1=\ldots=f_k=0$ over Archimedean fields (e.g., $\R$)
(see \cite{Kou76} for example).
Recent development generalizes this to over non-Archimedean local fields
using method combining toric and rigid geometry.
Dwork's original method does not require our variety to be
regular (which we shall define below); but with certain regularity condition
Adolphson-Sperber \cite{AS89} was able to apply a method of Kouchnirenko \cite{Kou76}
and extended the above technique to $L$ function of expotential sums and hence
zeta functions of hypersurfaces over finite fields.
In summary, certain algebraic invariants of a system of
regular Laurent polynomials can be read from torical combinatorial
invariants. This paper concerns the $p$-adic valuation of
roots of $L$ functions of exponential sums of Laurent polynomials, and its
applications to long-standing questions in algebraic geometry.
See \cite{Maz72} for a beautiful historical account 
from algebraic geometric point of view, 
see \cite{Dw64} and \cite{Kat88} for earlier development.

We first fix notations in order to  recall some classical results in the area.
Let $p$ be a prime and $q$ a $p$-power, let $\F_q$ be the finite field of $q$ elements.
Fix a nontrivial additive character $\psi$ of $\F_p$, and we take
without loss of generality
that $\psi(x)=\zeta_p^{x}$ for a primitive $p$-th root of unity $\zeta_p$ here.
Let $\wp$ be a prime of $\bar\Q$ with residue field $\F_q$ say.
For a Laurent polynomial $f$ in $\bar\Q[x_1,\ldots,x_n]$,
let the $m$-th exponential sum of the Laurent polynomial $\bar{f}=(f\bmod \wp)$
over $\F_q$ be
\begin{eqnarray}
S_m^*(\bar{f})=\sum \psi(\Tr_{\F_{q^m}/\F_p}(\bar{f}(x_1,\ldots,x_n)))
\end{eqnarray}
where the sum is over all $(x_1,\ldots,x_n)$ in the $n$-torus $(\G_m)^n(\F_{q^m})$.
By a well-known theorem of Dwork-Bombieri-Grothendieck, the following
$L$ function is a rational function in $\Z[\zeta_p](T)$
\begin{eqnarray}
L^*(\bar{f},T) =\exp(\sum_{m=1}^{\infty} S_m^*(\bar{f})\frac{T^m}{m} )
=\frac{\prod_{i}(1-\alpha_i T)}{\prod_{j}(1-\beta_jT)}.
\end{eqnarray}
Equivalently
$
S_m^*(\bar{f})  = \beta_1^m+\beta_2^m +\ldots -\alpha_1^m -\alpha_2^m-\ldots,
$
where $\alpha_i, \beta_j$ are algebraic integers (Weil numbers) in $\bar\Q$
according to \cite{Dw64} and \cite{Del80}.

Write $\Vert(\Delta)$ for the set of all (non-origin) vertices on $\Delta$.
Let $\partial(\Delta)$ denote the set of integral points on origin-less faces of $\Delta$.
Let $\A(\Delta)$ denote the coefficient space of all Laurent
polynomials $f=\sum_{j\in (\Delta-\partial\Delta)\cap \Z^n\cup \Vert(\Delta)}a_j x^j$ 
parametrized by all $(a_j)$'s where
$a_j\neq 0$ for $j\in \Vert(\Delta)$.
We denote by $\A(\Delta^o)$ the subspace of $\A(\Delta)$
consisting of all Laurent polynomials $f=\sum_{j\in (\Delta-\partial\Delta)\cap\Z^n}a_j x^ j +\sum_{j\in \Vert(\Delta)}c_j x^j$ parametered by $(a_j)$'s and with prescribed nonzero $c_j$'s.

For every subset $\Sigma$ of $\Delta$ we define the restriction
$f_\Sigma:=\sum_{j\in \Sigma\cap \Z^n} a_j x^j$ of the Laurent polynomial $f$ to
$\Sigma$.
A Laurent polynomial $f$ over a commutative ring $R$ with Newton polytope $\Delta$ is
{\em regular} with respect to $\Delta$
if for every closed origin-less face $\Sigma$ (on the boundary) of the polytope $\Delta$
the system $\frac{\partial{f_\Sigma}}{\partial{x_1}}=\ldots=\frac{\partial{f_\Sigma}}{\partial{x_n}} =0$
has no common zero in  the $n$-torus $(\G_m)^n(R)$. It is known
that regularity of a Laurent polynomial is a generic condition.
By Grothendieck-Katz \cite{Gro64}, there is a scheme $\cM(\Delta)$ defined over $\Z$ whose
reduction mod $p$ for every $p$  is a Zariski dense open subset $\cM(\Delta,\bar\F_p)$
of $\A(\Delta,\bar\F_p)$  defined over
$\F_p$ consisting of regular $\bar{f}$'s.
In fact recent work of \cite{GKZ08} on resultant and discriminant
yields explicit described $\Z$-scheme $\cM(\Delta)$.
The scheme $\cM(\Delta^o)$ is defined analogously.
For every $\bar{f}\in \cM(\Delta,\bar\F_p)$
Adolphson-Sperber showed that $L^*(\bar{f})^{(-1)^{n-1}}$ is a polynomial
and hence we may define $\NP(\bar{f})$ as
the $p$-adic Newton polygon of $L^*(\bar{f})^{(-1)^{n-1}}$ normalized so that
the $p$-adic order of its residue field cardinality $q$ is $1$.
They  defined a combinatorial polygon $\HP(\Delta)$
for all regular Laurent polynomials $f$ with a given polytope $\Delta$
(see (\ref{E:HP}) for definition of $\HP(\Delta)$).
For ease of reference, we summarize these results in the following theorem.
We use $``\geq"$ to denote
the first Newton polygon lies above or equal to the second in $\R^2$
and their endpoints meet.

\begin{theorem}
\label{T:basic}
\begin{enumerate}
\item (Adolphson-Sperber \cite{AS89}; Denef-Loeser \cite{DL91})
Let $\bar{f}$ be regular with Newton polytope $\Delta(\bar{f})=\Delta$
of dimension $n$ in $\R^n$.
Then $L^*(\bar{f},T)^{(-1)^{n-1}}$ is a polynomial.
For general $\Delta$ it is of degree $n!\Vol(\Delta)$
where $\Vol(\Delta)$ is the volume of $\Delta$.
\item (Grothendieck-Katz \cite{Gro64}, Adolphson-Sperber \cite{AS89})
The following exists
\begin{eqnarray*}
\GNP(\Delta;\bar\F_p)
&:=& \inf_{\bar{f}\in\cM(\Delta,\bar\F_p)} \NP(\bar{f}).
\end{eqnarray*}
We have
$\NP(\bar{f})\geq \GNP(\Delta;\bar\F_p)\geq\HP(\Delta)$
and their endpoints meet.
\end{enumerate}
\end{theorem}

The generic Newton polygon $\GNP(\Delta;\bar\F_p)$  at prime $p$ typically depends on
$\Delta$ and $p$, so in general we can not hope that $\GNP(\Delta;\bar\F_p)$
will coincide with the absolute lower bound $\HP(\Delta)$ that is independent of $p$.

\begin{theorem}[Wan \cite{Wan93}]
\label{T:wan}
For every prime $p$ large enough and
$p\equiv 1\bmod D^*(\Delta)$ for some computable positive
integer $D^*(\Delta)$ depending only on $\Delta$ we have
$\GNP(\Delta;\bar\F_p)=\HP(\Delta)$. For every prime $p$ there is a Zariski dense subset
$\cU_p$ in $\cM(\Delta,\bar\F_p)$ such that for every
$\bar{f}\in \cU_p(\bar\F_p)$ we have $\NP(\bar{f})=\GNP(\Delta;\bar\F_p)$.
\end{theorem}

Suppose $\Delta$ is integral convex polytope of dimension $n$ in $\R^n$ containing
the origin and it is simplicial at each origin-less facets.
If $\bar{f}$ is a regular polynomial with Newton polytope $\Delta$ over
a finite field $\F_q$ such that the restriction $\bar{f}_\delta$ to each origin-less
facet $\delta$ is only supported on $\Vert(\delta)$,
then we have $\NP(\bar{f})=\HP(\Delta)$ if $p\equiv 1\bmod D(\Delta)$
where $D(\Delta)=\lcm_{\delta}\ell_\delta$ where $\ell_\delta$ is the maximal
invariant factor of the $n\times n$ matrix of $\Vert(\delta)$ (see Proposition \ref{P:disc}
for more details).
This is a result due to Wan (see \cite[Theorem 3.1, Corollary 3.2]{Wan04} and its proof
in his groundbreaking paper \cite{Wan93}.) 
In fact, Wan has proved in \cite{Wan04} that for any $\Delta$ and 
$n\leq 3$ we have
$\GNP(\Delta,\bar\F_p)=\HP(\Delta)$ if $p\equiv 1 \bmod D(\Delta)$;
for $n\geq 4$ this is false. 

In one variable case, a recent result \cite{BF07} gives a {\em Hasse variety} $H_p$
explicitly whose complement in $\cM(\Delta,\bar\F_p)$ defines $\cU_p$ as
predicted in Theorem \ref{T:wan}.
One of our main results of this paper is to show a global Hasse variety whose
fibre at each $p$ coincides with the local Hasse variety $H_p$.

\begin{theorem}\label{T:main}
Let $\Delta$ be an integral convex polytope of dimension $n$ in $\R^n$ containing the origin
and is simplicial at all origin-less
facets, and let $(\Delta-\partial\Delta)\cap\Z^n \cup \Vert(\Delta)$
generate the monoid $C(\Delta)\cap \Z^n$ (up to finitely many points).
Let $\A(\Delta^o)$ be the space of all Laurent polynomials
$f=\sum_{j\in(\Delta-\partial\Delta)\cap\Z^n}a_j x^j+\sum_{j\in\Vert(\Delta)}c_j x^j$
parametrized by $(a_j)$ with $j$ not on origin-less faces of $\Delta$ and
with prescribed $c_j\neq 0$.
Then there exists a Zariski dense open subset $\cU$ of  $\A(\Delta^o)$
defined over $\Q$
such that for every $f\in\cU(\bar\Q)$
and for any $p$ large enough we have
$$
\NP(\bar{f}) = \GNP(\Delta;\bar\F_p);
$$
Moreover, we have
$$
\lim_{p\rightarrow \infty} \NP(\bar{f}) = \lim_{p\rightarrow \infty} \GNP(\Delta,\bar\F_p)
=\HP(\Delta).
$$
\end{theorem}

Denote the special fibre at $p$ in $\cM(\Delta^o)/_{\Z}$ by
$\iota_p: \cM(\Delta^o,\bar\F_p) \hookrightarrow\cM(\Delta^o)/_{\Z}.$
Theorem \ref{T:main} establishes
the existence of a  {\em global Hasse scheme} $H$
such that the diagram in Fig.\ref{Fig:local-global} commutes for $\cU=\cM(\Delta^o)-H$
(in Theorem \ref{T:main}) and
$\cU_p=\cM(\Delta,\bar\F_p)-H_p$ (in Theorem \ref{T:wan}).
For the first time Theorem \ref{T:main} establishes the existence of 
global generic-ness which reflects generic-ness at all but finitely many 
special fibers.This line of investigation was partially inspired by
\cite{Elk87}.

\begin{figure}[!h]
\begin{tikzpicture}
\matrix (m) [matrix of math nodes, row sep=3em,
column sep=2em, text height=1.5ex, text depth=0.25ex]
{
\hspace{1.4cm}\bar{f}\in \cU_p(\bar\F_p)& \iota_p^{-1}(\cU) \subseteq \cU_p& \cM(\Delta^o,\bar\F_p)_{/\F_p} &\A(\Delta^o,\bar\F_p)_{/\F_p} \\
\hspace{1cm}f\in \cU(\bar\Q)& \cU_{/\Q} & \cM(\Delta^o)_{/\Z} & \A(\Delta^o)_{/\Q}\\
};

\path[|->]
(m-2-1)  edge node [right] {mod $\wp$} (m-1-1);

\path[right hook->]
(m-1-2) edge node [auto] {$ \iota_p$} (m-2-2)
(m-1-3) edge node [auto] {$ \iota_p$} (m-2-3)
(m-1-4) edge node[auto]  {$\iota_p$} (m-2-4)
(m-1-2) edge (m-1-3)
(m-2-2) edge (m-2-3)
(m-1-3) edge (m-1-4)
(m-2-3) edge (m-2-4);

\end{tikzpicture}
\caption{}
\label{Fig:local-global}
\end{figure}

\begin{remark}
\label{R:main}
\begin{enumerate}
\item
The one variable case of Theorem \ref{T:main} was proved in \cite{Zhu03}.
Special one variable cases were also explored in \cite{Hon01}\cite{Hon02}\cite{Yan03}
and \cite{BF07}.  A special case in multivariable case in which the Laurent polynomial is
a sum of one-variable polynomials is addressed  in \cite{Bla11}, one can reduce
this to one-variable case and conclude by the argument of \cite{Zhu03} or using
deformation theory of Wan \cite{Wan04}.

\item We observe that $\cU$ in the theorem is Zariski dense in $\A(\Delta^o)$
is equivalent to that in $\cM(\Delta^o)$ since $\cM(\Delta^o)$ is
Zariski dense open in $\A(\Delta^o)$.

\item
We observe that $\cU$ has to be a proper subset of $\cM(\Delta^o)$ for all but finitely many $\Delta$:
The existence of {\em global permutation polynomials} $f$
(in one variable case it means a polynomial mod $p$ is
a permutation for the set $\Z/p\Z$ for infinitely many prime $p$)
immediately shows that $\NP(\bar{f})\gneq \GNP(\Delta;\bar\F_p)$ for infinitely many $p$,
so one can not expect to have $\NP(\bar{f}) = \GNP(\Delta;\bar\F_p)$ for all $\bar{f}$
in $\A(\Delta^o,\bar\F_p)$.

\item 
The hypothesis of Theorem \ref{T:main} is sufficient but not necessary.
In this example below we shall find that its hypothesis is not satisfied but
its assertion holds:
In Figure \ref{Fig:non-smooth} we give an example of a simplex polytope that does not
satisfy the hypothesis of Theorem \ref{T:main}, namely,  $\Delta\cap\Z^3$ does not
generate the monoid $C(\Delta)\cap\Z^3$ (up to finitely many integral points).
Indeed, we have $\Delta\cap \Z^3 =\{\0, (1,1,0),(1,0,1), (0,1,1)\}$. It is clear that
the lattice points $(m,m,m)$ for odd $m\in\Z_{\geq 0}$ 
in $C(\Delta)$ are not generated.

\begin{figure}[!h]
\begin{tikzpicture}[scale=0.3]
\draw[->] (0,0) -- (5,0);
\draw[->](0,0) --(0,6);
\draw[->](0,0) -- (-3,-1.5);
\draw[gray]
(2,-1)--(-2,-1)--(-2,3) -- (0,4) -- (4,4) -- (4,0) -- (2,-1) --(2,3)--(4,4);
\draw[gray] (2,3)--(-2,3);
\draw[gray] (-2,-1) -- (0,0) -- (4,0) --(0,0) -- (0,4);
\draw[line width = 2pt] (-2,3)--(4,4)--(2,-1)--(-2,3) -- (0,0) -- (4,4) -- (0,0) --(2,-1);

\node[below] at (5,0) {$x_1$};
\node[left] at (-3,-1.5) {$x_2$};
\node[right] at (0,5.9) {$x_3$};
\node[above right] at (4,0) {$1$};
\node[above right] at (0,4) {$1$};
\node[above left] at (-2,-1) {$1$};
\end{tikzpicture}
\caption{}
\label{Fig:non-smooth}
\end{figure}

\end{enumerate}
\end{remark}

\begin{examples}
\label{EX:T-main}
The following are classical examples of $\Delta$ satisfying the hypothesis of
Theorem \ref{T:main}.
\begin{enumerate}
\item[(a)]
In the one variable case $n=1$,
if we consider all polynomial of degree $d$, then
the polytope $\Delta$ is the line segment
$[0,d]$ on the real line $\R$.
We have $\dim \A(\Delta^o) = d-1$.

\item[(b)]
The space of all polynomials in $n$ variables with prescribed max/min degree in each variable;
or max/min total degree. 
Figure \ref{Fig:dim-2} illustrate some families in two variable case $n=2$.
Consider (i) all polynomials $f$ in variables $x_1$ and $x_2$
of total $\deg(f)=d>1$ or
(ii) all polynomials with bounded degree on each variable $\deg_{x_1}(f) =\deg_{x_2}(f)=d$,
or (iii) all Laurent polynomials with bounded degree at each variable
$-d\leq \deg_{x_i}(f)\leq d$ for $i=1,2$.

\begin{figure}[h]
\begin{tikzpicture}[scale=0.5]
\draw[->] (0,0) -- (4,0);
\draw[->](0,0) --(0,4);
\filldraw[fill=gray!50,line width = 2pt] (0,0) -- (3,0) -- (0,3) -- (0,0);

\node[below] at (0,0) {$0$};
\node[below] at (3,0) {$d$};
\node[left] at (0,3) {$d$};
\node[below] at (4,0) {$x_1$};
\node[left] at (0,4) {$x_2$};
\node[below] at (1,0) {$g_1$};
\node[left] at (0,1) {$g_2$};
\filldraw[fill=black] (0,0) circle (3pt);
\filldraw[fill=black] (3,0) circle (3pt);
\filldraw[fill=black] (0,3) circle (3pt);
\filldraw[fill=blue] (1,0) circle (3pt);
\filldraw[fill=blue] (0,1) circle (3pt);
\draw[xshift=6cm,->] (0,0) -- (4,0);
\draw[xshift=6cm,->](0,0) --(0,4);
\filldraw[xshift=6cm,fill=gray!50,line width = 2pt] (0,0) -- (3,0) -- (3,3) -- (0,3) -- (0,0);
\draw[xshift=6cm,dashed] (0,0) -- (3,3);

\node[below] at (1.5,-4) {(i)};
\node[below] at (7.5,-4) {(ii)};
\node[below] at (16,-4) {(iii)};

\node[below] at (6,0) {$0$};
\node[below] at (9,0) {$d$};
\node[left] at (6,3) {$d$};
\node[right] at (3,3) {$$};
\node[below] at (10,0) {$x_1$};
\node[left] at (6,4) {$x_2$};
\node[below] at (7,0) {$g_1$};
\node[left] at (6,1) {$g_2$};
\node[right] at (7,1) {$g_3$};

\filldraw[xshift=6cm,fill=black] (0,0) circle (3pt);
\filldraw[xshift=6cm,fill=black] (3,0) circle (3pt);
\filldraw[xshift=6cm,fill=black] (0,3) circle (3pt);
\filldraw[xshift=6cm,fill=black] (3,3) circle (3pt);
\filldraw[xshift=6cm,fill=blue] (1,0) circle (3pt);
\filldraw[xshift=6cm,fill=blue] (0,1) circle (3pt);
\filldraw[xshift=6cm,fill=blue] (1,1) circle (3pt);


\filldraw[xshift=16cm,fill=gray!50,line width = 2pt] (-3,-3) -- (-3,3) -- (3,3) -- (3,-3) -- (-3,-3);
\draw[xshift=16cm,->] (-4,0) -- (4,0);
\draw[xshift=16cm,->](0,-4) --(0,4);
\draw[xshift=16cm,dashed] (-3,-3) -- (3,3);
\draw[xshift=16cm,dashed] (-3,3) -- (3,-3);

\node[below left] at (16,0) {$0$};
\node[above left] at (13,0) {$-d$};
\node[above left] at (19.9,0) {$d$};
\node[above right] at (16,-4) {$-d$};
\node[above right] at (16,3) {$d$};
\node[below] at (20,0) {$x_1$};
\node[left] at (16,4) {$x_2$};
\node[right] at (17,0) {$g_3$};
\node[right] at (17,-1) {$g_2$};
\node[right] at (17,1) {$g_1$};

\filldraw[xshift=16cm,fill=black] (0,0) circle (3pt);
\filldraw[xshift=16cm,fill=black] (3,3) circle (3pt);
\filldraw[xshift=16cm,fill=black] (-3,-3) circle (3pt);
\filldraw[xshift=16cm,fill=black] (3,-3) circle (3pt);
\filldraw[xshift=16cm,fill=black] (-3,3) circle (3pt);
\filldraw[xshift=16cm,fill=black] (0,0) circle (3pt);
\filldraw[xshift=16cm,fill=black] (3,3) circle (3pt);

\filldraw[xshift=16cm,fill=blue] (1,0) circle (3pt);
\filldraw[xshift=16cm,fill=blue] (1,1) circle (3pt);
\filldraw[xshift=16cm,fill=blue] (1,-1) circle (3pt);
\filldraw[xshift=16cm,fill=blue] (-1,-1) circle (3pt);
\filldraw[xshift=16cm,fill=blue] (0,1) circle (3pt);
\filldraw[xshift=16cm,fill=blue] (-1,1) circle (3pt);
\filldraw[xshift=16cm,fill=blue] (-1,0) circle (3pt);
\filldraw[xshift=16cm,fill=blue] (0,-1) circle (3pt);

\end{tikzpicture}
\caption{}
\label{Fig:dim-2}
\end{figure}
\end{enumerate}
\end{examples}

This paper is partially motivated by conjectures proposed by Wan in
\cite{Wan04}. We shall prove a strengthened version of Conjecture 1.12 of \cite{Wan04}
in Theorem \ref{T:Wan1.12} and  Conjecture 1.11 of \cite{Wan04} in
Theorem \ref{T:Wan1.11}  (proofs of both lie in Section
\ref{S:proof}).
Notice that when $V=\Vert(\Delta)$ and $J=(\Delta-\partial\Delta)\cap\Z^n$,
Theorem \ref{T:main} follows from Theorem \ref{T:Wan1.12} as a special case.

\begin{theorem}
\label{T:Wan1.12}
Let $\Delta$ be an integral convex polytope of dimension  $n$ in $\R^n$ containing
the origin and it is simplicial at all origin-less facets.
Let $J$ be a set of integral points in $\Delta$ not on any origin-less faces
such that $J\cup \Vert(\Delta)$ generates the monoid $C(\Delta)\cap\Z^n$ 
up to finitely many points.
Let $V$ be a set of non-origin integral points in $\Delta$ disjoint from $J$ and
includes $\Vert(\Delta)$. Let
$
\A_V^J
$
be the family of Laurent polynomials
$
f(x)=\sum_{ j \in J} a_j x^j + \sum_{j\in V} c_j x^j
$
parameterized by $a_j$'s
and with prescribed $c_j$'s where $c_j\neq 0$ for all $j\in \Vert(\Delta)$.
Write $\GNP(\A_V^J,\bar\F_p):=\inf_{\bar{f}} \NP(\bar{f})$
where $\bar{f}$ ranges over all regular $\bar{f}\in \A_V^J(\bar\F_p)$.
Then there exists a Zariski dense open subset $\cU$
defined over $\Q$ in $\A_V^J$ such that for any
$f\in \cU(\bar\Q)$
and for $p$ large enough
we have
$\NP(\bar{f})=\GNP(\A_V^J,\bar{\F}_p)$
and
$$
\lim_{p\rightarrow \infty} \NP(\bar{f})
=\HP(\Delta).
$$
\end{theorem}

\begin{remark}
In 1-variable case $n=1$, let $\Delta$ be the line segment between $0$ and $d\in\Z_{\geq 1}$.   
Given a subset $J$ of $\Delta\cap\Z^n$ then $J\cup \{d\}$ generates the monoid $\Z_{\geq 0}$ up to 
finitely many points
if and only if it is a set of coprime integers. 
\end{remark}

Let $\cM_V^J(\bar\F_p):=\cM(\Delta^o,\bar\F_p)\cap \A_V^J(\bar\F_p)$ 
and by \cite{GKZ08} we have scheme $\cM_V^J$ over $\Z$ consisting of 
regular Laurent polynomials in $\A_V^J$. Then we arrive at a refinement 
of the commutative diagram in Figure \ref{Fig:local-global}.
In fact, the diagram in Fig. \ref{Fig:refinement} embeds naturally 
(component-wise) to that in Fig.\ref{Fig:local-global} as sub-objects
in the corresponding categories.

\begin{figure}[!h]
\begin{tikzpicture}
\matrix (m) [matrix of math nodes, row sep=3em,
column sep=2em, text height=1.5ex, text depth=0.25ex]
{
\hspace{1.4cm}\bar{f}\in \cU_p(\bar\F_p)& \iota_p^{-1}(\cU) \subseteq \cU_p& \cM_V^J(\bar\F_p)_{/\F_p} &\A_V^J(\bar\F_p)_{/\F_p} \\
\hspace{1cm}f\in \cU(\bar\Q)& \cU_{/\Q} & {\cM_V^J}_{/\Z} & {\A_V^J}_{/\Q}\\
};

\path[|->]
(m-2-1)  edge node [right] {mod $\wp$} (m-1-1);

\path[right hook->]
(m-1-2) edge node [auto] {$ \iota_p$} (m-2-2)
(m-1-3) edge node [auto] {$ \iota_p$} (m-2-3)
(m-1-4) edge node[auto]  {$\iota_p$} (m-2-4)
(m-1-2) edge (m-1-3)
(m-2-2) edge (m-2-3)
(m-1-3) edge (m-1-4)
(m-2-3) edge (m-2-4);

\end{tikzpicture}
\caption{}
\label{Fig:refinement}
\end{figure}

An immediate application of our main results is to Artin-Schreier varieties.
Given any Laurent polynomial $f$ over $\bar\Q$ we have
an  Artin-Schreier variety
$$
\bar{V_f}:  y^p-y=\bar{f}
$$
over a finite field
for each reduction $\bar{f}$ of $f$ a prime of $\bar\Q$ (with residue field $\F_q$).
Consider the toric affine hypersurfaces given via the embedding
$\bar{V_f}\hookrightarrow (\G_m)^n$,
its zeta function is
$$\Zeta(\bar{V_f};T)={\mathrm{Norm}}_{\Q(\zeta_p)/\Q}(L^*(\bar{f};T))/(1-T)(1-q^nT).$$
If $f$ is regular with respect to $\Delta$ then ${\mathrm{Norm}}_{\Q(\zeta_p)/\Q}(L^*(\bar{f};T))^{(-1)^{n-1}}$ is polynomial by Theorem \ref{T:basic}, and
we denote $\NP(\bar{V_f})$ for the $p$-adic Newton polygon of this polynomial
normalized by shrinking of a magnitude of $p-1$ vertically and horizontally.
So $\NP(\bar{V_f}) = \NP(\bar{f})$ and hence
the following corollary is an immediate application of Theorem \ref{T:main}.

\begin{corollary}\label{C:main}
Let $\Delta$ and $\A_V^J$ be as in Theorem \ref{T:Wan1.12}.
Then there exists a Zariski dense open subset $\cU$ of  $\A_V^J$
defined over $\Q$ such that for every $f$ in $\cU(\bar\Q)$ and
for $p$ large enough we have
$$
\NP(\bar{V_f}) =\GNP(\A_V^J;\bar\F_p),
$$
where $\bar{V_f}: y^p-y=\bar{f}$ is the Artin-Schreier variety defined above.
Moreover, we have
$$
\lim_{p\rightarrow \infty}\NP(\bar{V_f}) =
\lim_{p\rightarrow \infty}\GNP(\A_V^J,\bar\F_p) =
\HP(\Delta).
$$
\end{corollary}

In Theorem \ref{T:Wan1.12} above we considered the space of Laurent polynomials
with given $\Delta$ and with prescribed coefficients in $\Q$
at the vertices of $\Delta$. In the theorem below,
we consider the space of Laurent polynomials with given $\Delta$
that are parameterized by
{\em all} coefficients.

\begin{theorem}
\label{T:allQ}
Let $\Delta$ be an integral convex polytope of dimension $n$
in $\R^n$ containing the origin and it is simplicial at all origin-less facets.
Let $J$ be a subset of $(\Delta-\partial\Delta)\cap\Z^n\cup \Vert(\Delta)$
so that $J$ generates the monoid $C(\Delta)\cap\Z^n$ up to finitely many points.
Let $\A^J$ be the space of all
Laurent polynomials $f= \sum_{j\in J} a_j x^j$ parameterized by
all $(a_j)$'s where $a_j\neq 0$ for $j\in \Vert(\Delta)$.
Then there exists a Zariski dense open subset $\cU$
of $\A^J$ defined over $\Q$ such that
for every $f\in \cU(\Q)$ and for $p$ large enough
we have $\NP(\bar{f}) =\GNP(\A^J,\F_p)$ and
$$
\lim_{p\rightarrow \infty} \NP(\bar{f})
= \lim_{p\rightarrow\infty}\GNP(\A^J,\F_p)
= \HP(\Delta).
$$
\end{theorem}

We have seen in Wan's Theorem \ref{T:wan} that
$\GNP(\Delta,\bar\F_p)$ coincides with $\HP(\Delta)$
for a special congruence class of prime $p$ and for $p$ large enough.
The following proves a strengthened version of 
Conjecture 1.11 in \cite{Wan04}.

\begin{theorem}
\label{T:Wan1.11}
Let $\Delta$ be an integral convex polytope of dimension $n$ in $\R^n$
containing the origin.
We write $\A^J$ for the space of Laurent polynomials
$f= \sum_{j\in J} a_j x^j$ parameterized by
$(a_j)$'s with $\Delta(f)=\Delta$.
For any subset $J$ with $\Vert(\Delta)\subseteq J\subseteq \Delta\cap\Z^n$
that generates the monoid $C(\Delta)\cap\Z^n$ up to finitely many points,
we have
$$
\lim_{p\rightarrow \infty} \GNP(\A^J;\bar\F_p) = \HP(\Delta).
$$
In particular, if $\Delta\cap\Z^n$ generates the monoid $C(\Delta)\cap\Z^n$ up to finitely many
points then we have
$$
\lim_{p\rightarrow \infty} \GNP(\Delta;\bar\F_p) = \HP(\Delta).
$$
\end{theorem}

\begin{examples}[Wan]
\label{EX:counter-example}
If $J:=\Delta\cap \Z^n$ does not generate $C(\Delta)\cap \Z^n$ up to finitely many points
we do not expect Theorem \ref{T:main}
generally hold and we give a counterexample here:
Let $\Delta$ be given by the following vertices
$P_0=(0,0,0,0),
P_1=(1,1,1,0), P_2=(1,1,0,1), P_3=(1,0,1,1), P_4=(0,1,1,1)$.
Then we have 
$\NP(\bar{f})\geq
\GNP(\Delta;\bar\F_p)$ for any $\bar{f}$ with Newton polytope
$\Delta$ have
$\lim_{p\equiv 2\bmod 3, p\rightarrow\infty}\NP(\bar{f})\neq \HP(\Delta)$.
Namely,
$\lim_{p\rightarrow\infty}\NP(\bar{f})$ and $\lim_{p\rightarrow\infty}\GNP(\Delta,\bar\F_p)$
do not exist.
\end{examples}

A set $\cG$ of integral points in $\Delta$ containing $\Vert(\Delta)$
has a {\em unimodular triangulation}
if $\Delta$ has a triangulation of simplices with all vertices in $\cG$
and whose maximal dimension simplices can be translated to 
be pointed at the origin whose $m$ nonzero vertex vectors
generates the full lattice $\Z^m$. If $\Delta$ is of dimension $\leq 2$ then
$\Delta\cap\Z^n$ always has a unimodular triangulation.
However it is false if $\dim\Delta\geq 3$, e.g., see  Fig. \ref{Fig:non-smooth}.
Following Wan's terminology (see \cite{Wan04}), 
we say $\Delta$ is {\em ordinary} at $p$
if $\GNP(\Delta,\bar\F_p) = \HP(\Delta)$.
Theorem \ref{T:Wan1.11} implies that for large enough $p$
a generic $f$ over $\bar\F_p$ defines an ordinary affine toric 
hypersurface, which we shall state below in Corollary \ref{C:Wan1.11}, whose proof
lies in Section \ref{S:proof}.

\begin{corollary}
\label{C:Wan1.11}
Let $\Delta$ be an integral convex polytopes of dimension $n$ in $\R^n$.
Let $\T_\Delta$ be the space of all affine toric hypersurfaces
$V(f)$ where $f=\sum_{j\in \Delta\cap\Z^n}a_j x^j$ parametered by
$a_j$'s where $a_j\neq 0$ for vertices $j$ in $\Delta$.
Let $\HP(\T_\Delta)$ be the Hodge polygon of toric hypersurfaces in $\T_\Delta$
given by Hodge numbers.
For any regular $V(f) \in \T_\Delta(\bar\Q)$,
let $V(\bar{f})$ be the reduction of $V(f)$ at a prime of $\bar\Q$ over $p$, and
let $\NP(V(\bar{f}))$ denote the normalized $p$-adic Newton polygon of
the key polynomial component of zeta function of $V(\bar{f})$.
Let $\GNP(\T_\Delta,\bar\F_p):=\inf_{\bar{f}} \NP(V(\bar{f}))$
where $V(\bar{f})$ ranges over $\T_\Delta(\bar\F_p)$. 
If $\Delta\cap\Z^n$ has a unimodular triangulation, then we have
for $p$ large enough
$$
\GNP(\T_\Delta,\bar\F_p) = \HP(\T_\Delta).
$$
\end{corollary}

For hypersurface in dimension $\leq 3$ Wan proves  in
\cite[Section 3]{Wan04} that
$\GNP(\T_\Delta,\bar\F_p) = \HP(\T_\Delta)$ for all $p$ if $\dim\Delta\leq 3$.
See his Corollaries 3.8 through 3.14 in \cite{Wan04} for more details.
However, for dimension $n\geq 4$ Wan has demonstrated in \cite[Section 2.4]{Wan04}
(see also Example \ref{EX:counter-example}) 
that there are $\Delta$'s with $\GNP(\T_\Delta,\bar\F_p)\neq \HP(T_\Delta)$.
Our Corollary \ref{C:Wan1.11} generalizes \cite[Theorem 10.4]{Wan08}.

By applying Proposition \ref{P:unimodular} 
and Corollary \ref{C:Wan1.11} we have the following:

\begin{corollary}
\label{C:n-cube}
Let $\T_\Delta$ be the space of all $V(f)$ where $f$ is a Laurent polynomial in $n$ variables
with Newton polytope $\Delta$ equal to an $n$-rectangle (i.e., with prescribed maximal and minimal 
degrees in each variable) 
or $n$-diamond (i.e., with prescribed total degree) 
then for $p$ large enough we have
$\GNP(\T_\Delta,\bar\F_p)=\HP(\T_\Delta).$
\end{corollary}

\begin{notations}
\label{N:notations}
In this paper $\Delta$ is an integral convex polytope in $\R^n$ of dimension $n$.
In fact we assume it to be general enough that
the polynomial $L^*(\bar{f})^{(-1)^{n-1}}$ is always of degree $n!\Vol(\Delta)$.
Let $S(\Delta)=C(\Delta)\cap \Z^n$ be the lattice cone of $\Delta$.
Let $D(\Delta)$ be the least positive integer such that
$w_\Delta(v)$ is of the form $\frac{\Z}{D(\Delta)}$ for all $v\in S(\Delta)$.
For any $i\in \Z_{\geq 0}$ let $S(\Delta)_i$ be consisted of
$r\in S(\Delta)$ with $w_\Delta(r)=\frac{i}{D(\Delta)}$.
Write $S(\Delta)_{\leq k}:=\bigcup_{i=0}^{k} S(\Delta)_i$.
Let $W_\Delta(i)=|S(\Delta)_i|$.
Let $k\in\Z_{\geq 0}$ and let $N_k=|S(\Delta)_{\leq k}|=\sum_{i=0}^{k}W_\Delta(i)$.
Let $H_\Delta(k) := \sum_{i=0}^{n}(-1)^i \binom{n}{i} W_\Delta(k-iD(\Delta))$.
Let $V_\Delta=n!\Vol(\Delta)$, and let $k_\Delta\in\Z_{\geq 0}$ be such that
$V_\Delta=|S(\Delta)_{\leq k_\Delta}|$. Let $N_\Delta=nD(\Delta)$.

Given a polynomial $P(T)=1+c_1T+\ldots + c_NT^N$ with coefficients in
a ring with $p$-adic valuation, then its Newton polygon (normalized with respect to a $p$-power $q$)
is the lower convex hull of the points $(i,\ord_q c_i)$ for all $i=0,1,\ldots,N$.
Denote it by $\NP_q(P(T))$. 
For any regular polynomial $\bar{f}\in \F_q[x_1,\ldots,x_n,1/(x_1\cdots x_n)]$
with Newton polytope $\Delta$ 
we write $\NP(\bar{f}):=\NP_q(L^*(\bar{f})^{(-1)^{n-1}})$.
We define a combinatorial Newton polygon below
\begin{eqnarray}
\label{E:HP}
\HP(\Delta)&:=& \NP_q(\prod_{i=0}^{n} (1-Tq^{\frac{i}{D(\Delta)}}))^{H_\Delta(i)}.
\end{eqnarray}

We write  $\lfloor R\rfloor$ for the biggest integer less than a given real number $R$,
and write $\{R\}$ for $R-\lfloor R\rfloor$.
For any positive integer $k$,
for any matrix $M$ we use $M^{[k]}$ to denote the first $k\times k$ submatrix of $M$;
for any polygon $\NP$ we use $\NP^{[k]}$ to denote the first horizontal length $k$ segment
of this polygon.

Fix positive integers $d_1,\ldots,d_n$ for any prime $p\nmid d_1\cdots d_n$
we let $R_{\Delta_\delta}(p):=(p\bmod d_1,\ldots,p\bmod d_n) \in \prod_{i=1}^{n}(\Z/d_i\Z)^*$
denoting a residue class.
\end{notations}

\noindent{\bf Plan of the article.}
This paper is organized as follows.
Our main results in Theorems \ref{T:Wan1.12}, \ref{T:allQ},
and \ref{T:Wan1.11}
are proved in Section \ref{S:proof}.
The technical part of the proof of Theorem \ref{T:Wan1.12}
consists of two crucial steps we described in Sections \ref{S:transform} and
\ref{S:proof}:
first the new facial decomposition
Theorem \ref{T:subdivision} and the transformation from
a linear  Dwork operator computation
(which depends on the prime in $\bar\Q$ lying over $p$)
to semi-linear Dwork operator (which only depends on the prime $p$)
as in Theorem \ref{T:transform}.
Both of these reduction steps are performed at chain-level.
On the other hand, in order to apply the results in Section \ref{S:transform}, one
needs strong estimates of the integral weight
function which we develop in Section \ref{S:convex} 
after providing a toolbox of basic notation in integral convex geometry;
and estimation of integral weight functions in Section \ref{S:estimate}.
Section \ref{S:dwork} recalls $p$-adic Dwork theory and for our $\Delta$
we obtain building blocks of {\em Hasse polynomials}.
Finally we prove Wan's Conjectures 1.11 and 1.12 of \cite{Wan04}
in Section \ref{S:proof}.
Our Theorem \ref{T:allQ} is proved in Section \ref{S:proof}.
\vspace{5mm}

\noindent{\bf Acknowledgment.}
We thank the extraordinary working environment at the MIT Mathematics
Department and Professor Bjorn Poonen for hospitality during our visit.
Part of this work is supported by an NSA grant 1094132-1-57192.

\section{Integral convex geometry: a toolbox}
\label{S:convex}

\subsection{Preliminaries}
\label{SS:prelim}

The interplay between integral convex geometry and
$p$-adic Dwork method will be apparent in Section \ref{S:dwork},
this is part of the dictionary between integral combinatorial convexity and
$p$-adic Dwork theory.
We recall and develop some auxilary results here on integral convex geometry.
We refer the reader to the books \cite{Ewa96}\cite{Ful93}\cite{Bar02}\cite{DRS10} for more information.
The convex hull of a finite set $\Sigma$ of points in $\R^n$ is called a {\em polytope}.
If $\Delta$ is the convex hull of a finite set $\Sigma$ consisting  of only integral points in $\Z^n$,
then it is an {\em integral polytope} (it is also called {\em lattice polytope} or
{\em integer polytope} in the literature).
A  polytope $\Delta$ is called {\em pointed} (or {\em pointed})
if it contains an extreme point, that is, a point that does not lie in any open line segmant joining
two points of $\Delta$. (A cone is pointed if and only if the origin is a vertex.)

Given a set of points $c_1,\ldots,c_m$ in $\R^n$, the set $\Delta$
of all points $x$ in $\R^n$ satisfying the inequalities (via scalar product) $\pscal{c_i, x} \leq 1$ for all $i=1,\ldots,
m$ is called a {\em polyhedron}.  If we can choose all $c_i$ in $\Q^n$  then this polyhedron is called
{\em rational}. A polytope is a bounded polyhedron.
Each hyperplane $\delta_i$ defined by $\pscal{c_i, x} = 1$ as above
is called a {\em supporting hyperplane}
of $\Delta$, and $\delta_i\cap \Delta$  is called a {\em face} of $\Delta$
(this is also called closed face).
An open face of $\Delta$ is a face minus its boundary.
The complement of interior points of $\Delta$ in $\Delta$ is called the {\em boundary} .
A {\em vertex, edge and facet} of $\Delta$ is a dimensional $0$, $1$ and $n-1$
face in $\Delta$, respectively.

All our polytopes in this paper are integral and all polyhedra (in particular cones) are
rational unless declared otherwise. For any non-empty set $\Delta$ in $\R^n$
the {\it cone} of $\Delta$ is defined by
$C(\Delta) := \{\lambda v |\lambda\in\R_{\geq 0}, v\in \Delta\}.$
Two points $v,v'$ in $C(\Delta)$ are called {\em cofacial} if
$v\R_{>0}$ and $v'\R_{>0}$ penetrate the
a same face of $\Delta$.

We define the {\it lattice cone}  of $\Delta$ as
$S(\Delta) := C(\Delta) \cap \Z^n.$
An integral point $v\in\Z^n$ is {\em primitive}
if all its coordinates in the standard basis of
$\Z^n$ are coprime to each other
(i.e., for $v=(v_1,\ldots,v_n)$ we have $\gcd(v_1,\ldots,v_n)=1$).
In fact, this property is independent of the choice of basis of $\Z^n$
since $\gcd$ is invariant under $\SL_n(\Z)$ transformation.
A {\em primitive generating set} $\cG_\Delta$ of a rational cone
$C(\Delta)$ is a minimal (set-inclusion) finite set of primitive integral points
in the cone that generates $C(\Delta)$ over $\R_{\geq 0}$.
It is clear that every rational cone $C(\Delta)$
contains a primitive generating set $\cG_\Delta$.
If the rational cone $C$ is pointed
then the minimal integral vectors of $C$ make up the
unique primitive generating set.

In integral convex geometry,
an (integral)  {\em simplex} $\Sigma$ is a convex hull of $d+1$ integral points
$v_1,\ldots,v_d,v_{d+1}$ in $\Z^n$
that is not a solution to the equation
$x=\sum_{i=1}^{d}c_i x_i$ for any $c_i\in\Q_{\geq 0}$ with
$\sum_{i=1}^{d}c_i =1$ for some $d\geq 1$.
Equivalently, it is an integral convex polytope
$\Sigma$ in $\R^n$  with $\dim\Sigma=|\Vert(\Sigma)|-1\leq n$.
A simplex is automatically pointed.
An integral polytope is {\em simplicial} if each  face is a simplex.
A rational cone $C$ is {\em simplex}  (or {\em simple}) if it is generated by $\dim C$
primitive integral points over $\R_{\geq 0}$;
it is {\em simplicial} if each proper face of $C$ is simplex.
A simplex cone of dimension $n$ has a unique
primitive generating set $\cG_\Delta=\{g_1,\ldots,g_n\}$ where $g_i$'s
are primitive vertex vectors. These lattice vectors are linearly independent over
$\R_{\geq 0}$.
If $\Delta$ is pointed at origin, then $\Delta$ is simplex if and only if $C(\Delta)$ is.
Note that $\Delta$ is simplex one can always translate it so that it is pointed at origin.

Let $S$ be an integral lattice cone in $\R^n$ whose cone is of dimension $n$
and let $\cG_\Delta=\{g_1,\ldots,g_n\}$ be a primitive generating set. Then
$Z^o(S):=\{\sum_{i=1}^{n} c_i g_i| 0\leq c_i<1\}$ is a
{\em fundamental parallelepiped} of $S$.
Define the {\em discriminant} of $S$ to be
$\disc(S):= |Z^o(S)\cap\Z^n|$.
Then it is well defined (independent of choice for $\cG_\Delta$).
Notice $\disc(\Z_{\geq 0}^n)=1$.
A lattice cone $S$ (or a cone $C(\Delta)$) is
{\em unimodular} in $\R^n$
if $\disc(S)=1$ ($\disc(S(\Delta))=1$). 

\begin{proposition}\label{P:disc}
\begin{enumerate}
\item Let $S\subset S'$ be any two lattice cones in $\Z^n$.
Then
$
|S'\Z/S\Z|=|Z^o(S)\cap S'|=\frac{\disc(S)}{\disc(S')}.
$
In particular, for any lattice cone $S$ in $\Z^n$ we have
$\disc(S)=\Vol(Z^o(S))$.

\item
Let the lattice cone $S$ be simplex with primitive generating set
$\cG_\Delta=\{g_1,\ldots,g_n\}$.
Let $\Delta$ is the convex hull of the set $\cG_\Delta$ and the origin.
Then $\disc(S) = |\Z^n/S\Z|= |\det(g_1,\ldots,g_n)|=n!\Vol(\Delta)$.
We have $\Mat(g_1,\ldots,g_n) \sim \diag(\ell_1,\ldots,\ell_n)$ where
$\ell_1|\cdots | \ell_n$ are invariant factors of the matrix of the
lattice cone of $\cG_\Delta$.
For every integral point $r=\sum_{i=1}^{n}r_ig_i$
in $\Z^n$ we have $r_i\in \frac{\Z}{\ell_n}$ for every $i$.

\item
Let $\Delta$ be
simplex then there is $\cG_\Delta=\{g_1,\ldots,g_n\}$ such that
$\Vert(\Delta)=\{d_1g_1,\ldots,d_ng_n\}$ for some $d_1,\ldots,d_n\in\Z_{\geq 1}$;
if $\Delta$ is simplex and pointed at origin, then such $\cG_\Delta$ is unique.
Furthermore, we have
$$V_\Delta=n!\Vol(\Delta)=d_1\cdots d_n |\det(g_1,\ldots,g_n)|.$$
Let $D(\Delta)$ be the least positive integer
such that $w_\Delta(v) \in \frac{\Z}{D(\Delta)}$
for every $v\in S(\Delta)$ where $w_\Delta(\cdot)$ is the weight function of \cite{AS89}, then
$D(\Delta)=\lcm(d_1,\cdots, d_n)\cdot \ell_n$.
\end{enumerate}
\end{proposition}
\begin{proof}
The statement in Part (2) is standard (see \cite{Ewa96}\cite{Bar02}) or
purely by definition.
Part (1) is \cite[Theorem VII 2.5]{Bar02}.
Consider $S'=\Z^n$ we arrive at $\disc(S) = \Vol(Z^o(S))$ immediately.
Part (3). Since $\Delta$ is a simplex (hence pointed) and
therefore a primitive generator $g$
must lie on vertex say $v$, that is $v=cg$ for some $c\in \Q_{>0}$,
and by its very definition one can see immediately that $c\in\Z_{>0}$. The rest of Part (3) is
straightforward.
\end{proof}

\subsection{Hilbert basis}

A {\it Hilbert basis} of $C(\Delta)$
is a minimal (set-inclusion) finite set of lattice vectors
that generate the monoid $S(\Delta)$
over $\Z_{\geq 0}$.
The well known Lemma \ref{L:finite}
(see \cite[Lemma V.3.4]{Ewa96} for a proof)
says that a Hilbert basis exists and we denote it by $\cG_{\Delta,\Z}$.
Notices that any Hilbert basis $\cG_{\Delta,\Z}$ a priori generates $C(\Delta)$ over $\R_{\geq 0}$,
and hence $\cG_{\Delta,\Z}$ always contains a primitive generating set $\cG_\Delta$.
In Fig. \ref{Fig:dim-2} we have (i) $\cG_\Delta=\cG_{\Delta,\Z}=\{g_1,g_2\}$,
(ii) $\cG_{\Delta,\Z}=\{g_1,g_2\}$ and $\cG_{\Delta_\delta}=\{g_1,g_2\}, \cG_{\Delta_{\delta'}}
=\{g_2,g_3\}$. (ii) Hilbert basis is not unique in this case (notice that $\Delta$ is not pointed).

\begin{lemma}[Gordan's Lemma]
\label{L:finite}
Let $\Delta$ be an integral polytope in $\R^n$ of dimension $n$.
The lattice cone $S(\Delta)$ is finitely generated monoid
in the sense that $S(\Delta) = \sum_{i=1}^{t}g_i \Z_{\geq 0}$
for some $g_1,\ldots, g_t$ in $\Z^n$.
\end{lemma}

Part of the proof of the following result can be found in \cite[Lemma V.3.5]{Ewa96}.

\begin{proposition}\label{P:Hilbert}
Let $\Delta$ be a simplex integer polytope in $\R^n$
with $\cG_\Delta=\{g_1,\ldots,g_n\}$ a primitive generating set of $C(\Delta)$
over $\R_{\geq 0}$.
Then one can find $g_{n+1},\ldots,g_t$ in
$Z^o(S):=\{\sum_{i=1}^{n}c_i g_i | 0\leq c_i<1 \}$ of
the lattice cone $S=C(\Delta)\cap\Z^n$
such that
$\cG_{\Delta,\Z}=\{g_1,\ldots,g_n,g_{n+1},\ldots,g_t\}$
is a Hilbert basis of $C(\Delta)$.
If $C(\Delta)$ is pointed then its Hilbert basis is unique.
\end{proposition}

\begin{proof}
Let $s$ be any integral point in $C(\Delta)$,
then we have $s=\sum_{i=1}^{t} c_i g_i$ for some $c_i\in \R_{\geq 0}$.
So $s-\sum \lfloor c_i\rfloor g_i$ lying in $Z^o\cap \Z^n$.
This proves that $Z^o\cap \Z^n$ generates $S(\Delta)$ over $\R_{\geq 0}$.
Let
$\cG_{\Delta,\Z}$ be the set of all nonzero $g\in Z^o\cap \Z^n$
such that  $g$ is not the sum of two other vectors
in $Z^o\cap \Z^n$.
We observe that every point in
$Z^o\cap \Z^n\backslash \cG_{\Delta,\Z}$ is generated by $\cG_\Delta$ over $\Z_{\geq 0}$
and that $\cG_{\Delta,\Z}$ lies in every generating set of the monoid $S(\Delta)$.
So $\cG_{\Delta,\Z}$ is a Hilbert basis.

It remains to show the uniqueness of Hilbert basis $\cG_{\Delta,\Z}$ in the pointed case.
Suppose there is an integral point $s$ there violating this property, and take such
a point minimizing $c^Ts$ where $c$ is
a vector such that $c^Tx >0$ for all nonzero $x\in C(\Delta)$.
The existence of $c$ is guaranteed because $C(\Delta)$
is pointed and one can prove that there exists a vector $c$ such that $c^Tx>0$ for
every nonzero $x\in C(\Delta)$.
Because $s$ is not in $\cG_{\Delta,\Z}$,
we have $s=s_i+s_j$ for some nonzero points $s_i,s_j$ in
$Z^o\cap \Z^n$. Now we have $c^Ts=c^T s_i+c^T s_j$
and all terms are positive. This means  $c^Ts_i < c^T s_j$
and $c^Ts_j < c^T s$.
By the assumption that $c^Ts$ is minimized under the condition
 that $s$ is not in $\cG_{\Delta,\Z}$, both $s_i$ and $s_j$ must belong to $\cG_{\Delta,\Z}$
contradicting $s$ is not a nonnegative integer combination of points in $\cG_{\Delta,\Z}$.
This proves that such points $s$ does not exist. So $\cG_{\Delta,\Z}$ is  unique.
\end{proof}

Let $\Delta$ be an integral convex polytope in $\R^n$ containing the origin
and let $V$ be a subset of $\Delta\cap\Z^n$.
If $V\subseteq \cG\subseteq\Delta\cap\Z^n$ where $\cG$ generates the
monoid $S(\Delta)$ (resp., up to finitely many points),
then we call a minimal subset of $\cG$
that generates $S(\Delta)$ an {\em (resp., asymptotic) Hilbert basis} containing $V$ in $\Delta$.
If $V$ is the empty set, then this is just an (resp., asymptotic) Hilbert basis.
An important example for this paper
is when $V=\Vert(\Delta)$:
if $\Delta$ is the line segment $[0,d]$ for some $d\geq 1$,
a Hilbert basis containing $V=\{d\}$ is
$\{1,d\}$, and an asymptotic
Hilbert basis containing $V=\{d\}$ of smallest cardinality is
$\{m,d\}$ with $1\leq m< d$ such that $\gcd(m,d)=1$.
For arbitrary dimension one observes that for a Hilbert basis $\cG_{\Delta,\Z}$ for
$S(\Delta)$, the set $\cG_{\Delta,\Z}\cup\Vert(\Delta)$ is always a Hilbert basis
containing $\Vert(\Delta)$.

\subsection{Smooth $\Delta$ and asymptotically smooth $\Delta$}
\label{S:smooth}

We recall triangulation from \cite{DRS10}.
All triangulation in this paper have vertices of simplices on integral points only.
Let $\Delta$ be an integral convex polytope containing the origin.
For every origin-less facet $\delta$ of $\Delta$
let $\Delta_\delta$ be the convex hull of $\delta$ and the origin. Notice
that
\begin{eqnarray}
\label{E:division}
\Delta &=& \cup_\delta \Delta_\delta
\end{eqnarray}
where $\delta$ ranges over all origin-less facets,
and it is called a {\em closed facial subdivision} of $\Delta$. These $\Delta_\delta$'s are not
necessarily disjoint.  If all origin-less facets of $\Delta$ are simplex then this is
a facial triangulation of $\Delta$.

Corresponding to the closed facial subdivision of $\Delta$ in (\ref{E:division}),
\begin{eqnarray}\label{E:closed-facial}
C(\Delta) &=& \cup_{\delta}C(\Delta_\delta)
\end{eqnarray}
is called {\em closed facial subdivision} of $C(\Delta)$.
Consider the interior of $C(\Delta_\delta)$'s for all origin-less faces $\delta$ of $\Delta$,
denote them by $C(\Delta_\delta)^o$'s. The set of these relative open subcones $C(\Delta_\delta)^o$
is called the {\em facial decomposition} in \cite{Wan92}.
The corresponding partitions
\begin{eqnarray}\label{E:open-facial}
C(\Delta) = \cup_{\delta} C(\Delta_\delta)^o;
&\qquad &
S(\Delta) =\cup_{\delta} S(\Delta_\delta)^o
\end{eqnarray}
are called the {\em open facial subdivision} of $C(\Delta)$
(and $S(\Delta)$, respectively) where
$\delta$ ranges in the set $\cF^o(\Delta)$
of all origin-less open faces of $\Delta$ plus the origin as a unique element.

We say $\Delta$ is {\em (resp., asymptotic) smooth}  with $\cG$ if there is
a subset $\cG$ with $\Vert(\Delta)\subseteq \cG \subseteq \Delta\cap \Z^n$
that generates the monoid $C(\Delta)\cap\Z^n$ (resp., up to finitely many points),
in other words, $\Delta\cap\Z^n$ contains an (resp., asymptotic) Hilbert basis for $C(\Delta)$
containing $\Vert(\Delta)$.
In 1-dimensional case where $\Delta$ is the line segment $[0,d]$ for some $d\geq 1$,
then $\Delta$ is asymptotic smooth with $\{m,d\}$ in $\Delta$ if and only if $\gcd(m,d)=1$.
If, furthermore, $\Delta$ has only one origin-less facet and is simplex, then
we say $\Delta$ is {\em smooth simplex with $\cG$}.
See Figure \ref{Fig:smooth} for an example of
a 2-dimensional smooth $\Delta$ with Hilbert basis $\cG_{\Delta,\Z}=\{g_1,g_2\}$
(not smooth simplex) with facial triangulation
$\Delta_\delta\cup\Delta_{\delta'}$ as in (\ref{E:division}).

\begin{figure}[h!]

\begin{tikzpicture}[scale=0.6]

\filldraw[fill=gray!30] (0,0) -- (3,0) -- (3,3) -- (0,3) -- (0,0);
\filldraw[fill=gray!50] (0,0) -- (3,0) -- (3,3);
\draw[->] (0,0) -- (4,0);
\draw[->](0,0) --(0,4);
\draw[dashed,thick] (0,0) -- (3,3);

\node[below left] at (0,0) {0};
\node[left] at (0,3) {3};
\node[below] at (3,0) {3} ;
\node[] at (1,2) {$\Delta_\delta$};
\node[] at (2.5,1) {$\Delta_{\delta'}$};
\node[above] at (1.5,3) {$\delta$};
\node[right] at (3,1.5) {$\delta'$};
\node[below] at (4,0) {$x_1$};
\node[left] at (0,4) {$x_2$};
\node[below] at (1,0) {$g_1$};
\node[left] at (0,1) {$g_2$};
\node[] at (1.1,0.6) {$g_3$};

\filldraw[fill=black] (0,0) circle (3pt);
\filldraw[fill=black] (3,0) circle (3pt);
\filldraw[fill=black] (0,3) circle (3pt);
\filldraw[fill=black] (3,3) circle (3pt);
\filldraw[fill=blue] (1,0) circle (3pt);
\filldraw[fill=blue] (0,1) circle (3pt);
\filldraw[fill=blue] (1,1) circle (3pt);

\end{tikzpicture}

\caption{Closed facial subdivision: $\Delta=\Delta_\delta\cup \Delta_{\delta'}$}
\label{Fig:smooth}
\end{figure}
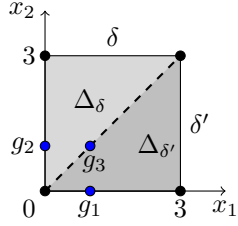

\begin{remark}
\label{R:no-closed-facial}
Suppose $\Delta=\cup_\delta \Delta_\delta$ is a subdivsion of $\Delta$
according to the origin-less facets $\delta's$. Then
$\Delta$ is smooth if all $\Delta_\delta$'s are smooth.
However, the converse is false in general.
In Figure \ref{Fig:hilbert}, we denote by $\delta$ the
face defined by the triangle $P_1P_2P_3$.
Then we have a smooth $\Delta$ (namely
$\Delta\cap\Z^n$ generates the monoid $C(\Delta)\cap \Z^n$),
but $\Delta_\delta$ is not smooth.
For any Hilbert basis $\cG_{\Delta,\Z}$,
the restriction $\cG_{\Delta,\Z}\cap \Delta_\delta$
does not generate the monoid $C(\Delta_\delta)\cap\Z^n$.
This phenomenon shows that a naive asymptotic version
of Wan's closed facial decomposition does not exist.
\end{remark}

\begin{figure}[h]
\begin{tikzpicture}[scale=0.3]
\filldraw[gray!10,line width = 2pt] (-2,3)--(4,4)--(2,-1)--(-2,3);
\draw[->] (0,0) -- (6,0);
\draw[->] (0,0) --(0,7);
\draw[->] (0,0) -- (-3,-1.5);
\draw[gray] (2,-1)--(-2,-1)--(-2,3) -- (0,4) -- (4,4) -- (4,0) -- (2,-1) --(2,3)--(4,4);
\draw[gray] (2,3)--(-2,3);
\draw[gray] (-2,-1) -- (0,0) -- (4,0) --(0,0) -- (0,4);
\draw[line width = 2pt] (4,4)--(0,4)--(-2,3)-- (-2,-1)--(2,-1)--(4,0)--(4,4);
\draw[line width = 2pt] (-2,3)--(4,4)--(2,-1)--(-2,3);

\node[right] at (2,3) {$P$};
\node[left] at (2,2.5) {$\delta$};
\node[below] at (5.5,0) {$x_1$};
\node[left] at (-3,-1.5) {$x_2$};
\node[left] at (0,6.6) {$x_3$};
\node[above right] at (4,0) {$1$};
\node[above right] at (0,4) {$1$};
\node[above left] at (-2,-1) {$1$};
\node[left] at (-2,3) {$P_3$};
\node[right] at (4,4) {$P_2$};
\node[below] at (2,-1)  {$P_1$};

\end{tikzpicture}
\caption{}
\label{Fig:hilbert}
\end{figure}

\subsection{Triangulations}
\label{S:triangulation}

All simplices in this paper are integral simplices, i.e., 
vertices are of integral coordinates in $\R^n$.
This paragraph follows notions in \cite[Chapter 5]{DRS10} closely. 
A {\em point configuration} in $\R^n$ is a finite set of 
(perhaps repeated) points with (non-repeated) labels.
We shall consider a point configuration $\cG$ 
with $\Vert(\Delta)\subseteq \cG\subseteq \Delta\cap\Z^n$,
so the convex hull of $\cG$ is $\Delta$.
In this paper we write $(\Delta,\cG)$ or $\cG$  
for a given integral point configuration
$\cG$.
A {\em triangulation} of $(\Delta,\cG)$  
is a collection $\cT$ of simplices whose union is $\Delta$, 
with vertices in $\cG$,
that satisfies the following properties
\begin{enumerate}
\item 
All faces of simplices of $\cT$ are in $\cT$; (Closure Property.)
\item
The intersection of any two simplices of $\cT$ is a 
(possibly empty) face of both; (Intersection Property)
\end{enumerate}
(Note that these two conditions are the definition of a simpicial complex.)
We remark that it is not necessary that every point in $\cG$ is a vertex 
of a simplex in $\cT$. However, if every point in $\cG$ is a vertex of a simplex in $\cT$ then 
we call $\cT$ a {\em complete triangulation}.
A triangulation $\cT$ of $(\Delta,\cG)$ in $\R^n$ is {\em regular}
if it can be obtained by projecting the lower envelope of a lifting of $\cG$ to $\R^{n+1}$.
More precisely, pick a height function $w: \cG\rightarrow \R$,
the lifting of $\cG$ to $\R^{n+1}$ is 
the set $\cG^w:=(v,w(v))_{v\in \cG}$ in $\R^{n+1}$. 
There is always a regular triangulation for $(\Delta,\cG)$ as proved in Proposition 2.2.4
of \cite{DRS10}. However we remark that there are
$(\Delta,\cG)$ which has no complete regular triangulation.

\begin{remark}
\label{R:comparison}
In many literature regular triangulization is the same 
as convex or  coherent triangulization. Our definition of 
regular triangulation is the same as \cite{DRS10}.
The `convex triangulization'
in \cite{Wan08} is different from ours, 
his `convex triangulization' 
is the same as our complete regular triangulization.  
\end{remark}

\begin{proposition}
\label{P:triangulation}
Let $\Delta$ be an integral convex polytope in $\R^n$ of dimension $n$
containing origin. 
Let $\cG$ be a set with $\Vert(\Delta)\subseteq \cG\subseteq \Delta\cap \Z^n$.
Let $\Delta=\cup_\delta\Delta_\delta$
be the closed facial subdivision in (\ref{E:division})
where $\delta$ ranges over all origin-less facets of $\Delta$.
Then there is a regular triangulation of $(\delta, \cG\cap\delta)$,
that is $\delta=\cup_\ell \delta_\ell$.
Then the convex hull $\Delta_{\delta_\ell}:=\conv(\delta_\ell,\0)$ is simplex, and
we have 
$\Delta = \cup_{\delta}\cup_{\delta_\ell} \Delta_{\delta_\ell}$
where $\delta$ ranges over all origin-less facets of $\Delta$ and 
$\delta_\ell$ ranges over all regular simplices in $\delta$.
\end{proposition}

An integral simplex with all vertices $v_0,v_1,\ldots,v_m$ 
is {\em unimodular} if $v_1-v_0,\ldots$, and $v_m-v_0$ generate
the lattice $\Z^m$ (over $\Z$). 
We remark that a unimodular triangulation is necessarily 
complete but the converse is false, see example in Fig.\ref{Fig:hilbert}.
A triangulation $\cT$ of $(\Delta,\cG)$ is {\em unimodular}
if $\cT$ consists of only unimodular simplices.
Equivalently, $\Delta=\cup_{i}\Delta_i$
where $\Delta_i$ are unimodular simplices with vertices in $\cG$
and their intersections (if nonempty) is also simplex.
It is known that for 
$n=\dim\Delta\leq 2$ then $\Delta\cap\Z^n$  always has a unimodular triangulation.
It is not true in general for dimension $n\geq 3$, see Fig. \ref{Fig:hilbert} for counter-example.
However, we have the following important examples which can be proved 
routinely.

\begin{proposition}
\label{P:unimodular}
Any $n$-rectangle has a unimodular triangulation, namely,  
let $d_i\in \Z_{\geq 1}$ for every $1\leq i\leq n$,  
let $\Delta:=\{j=(j_1,\ldots,j_n))\in\Z^n| -d_i\leq j_i\leq d_i \mbox{ for all $i$}\}$. 

2) Let $d\in\Z_{\geq 1}$, and let 
$\Delta:=\{j=(j_1,\ldots,j_n)\in \Z^n|-d\leq  j_1+\ldots+j_n\leq d\}$.
Then the space of all Laurent polynomials with Newton polytope $\Delta$
has a unimodular triangulation.
\end{proposition}

Notice that the first space in Proposition \ref{P:unimodular} 
is the space of all Laurent polynomials with prescribed 
maximal degree $d_i$ and minimal degree $-d_i$ 
at each variable $x_i$, 
while the second space is that of all Laurent polynomials 
with prescribed maximal total degree $d$ and minimal total degree $-d$.

\subsection{Graded lexicographic order}
\label{S:glex}

Suppose $\Delta$ is (resp., asymptotically) smooth and 
let $\cG_o$ be any subset of $\Delta\cap\Z^n$ that (resp., asymptotically) generates 
the monoid $C(\Delta)\cap\Z^n$.
In this paper we shall consider 
set of all monomials in variables $(A_j)$ with $j$ supported on 
the set $S:=\cG_o-\Vert(\Delta_\delta)$ only. 
Let $t=|\cG_o|$. We fix a total order on $S$.
For any element $\alpha=(\alpha_1,\ldots,\alpha_t)$ in $S$ write 
$|\alpha|=\sum_{i=1}^{t} \alpha_i$.
The {\em graded lexicographic order} with respect to $S$ 
on the set of monomials of the form
$\prod_{j\in S} A_j^{\alpha_j}$ is defined as follows:
for any $\alpha=(\alpha_1,\ldots,\alpha_t), \beta=(\beta_1,\ldots,\beta_t) 
\in \Z_{\geq 0}^{t}$
we have $\alpha>_{glex} \beta$ if $|\alpha|> |\beta|$ or $|\alpha|=|\beta|$ 
and the left-most nonzero entry of $\alpha-\beta$ is $>0$. 
This is a complete order on the set $S$, in fact it is a monomial order (see \cite{CLO}).
It is clear by definition that in any set of such monomials,
a monomial $\prod_{j\in S}A_j^{\alpha_j}$ 
is of the lowest graded lexicographic order if and only if
it is of the lowest (total) degree 
and its exponent vector $(\alpha_1,\ldots,\alpha_t)$ is of the lowest lexicographic order.

\section{Estimates of integral weight function}
\label{S:estimate}

\subsection{Vertex representations and vertex residues}

We retain notations from Section \ref{S:convex}.
For any simplex $\Delta$ pointed at origin,
the (semi-open) {\em zonotope} of $\Delta$ is
\begin{eqnarray}
\label{E:Z^o}
Z^o(\Delta) &:=&\{\sum_{j\in \Vert(\Delta)}\alpha_j j| 0\leq \alpha_j < 1\}.
\end{eqnarray}

\begin{examples}
\label{EX:Hilbert}
The polytope $\Delta$ in Figure \ref{Fig:smooth-triangulation}
is smooth simplex,
 with primitive generating set $\cG_\Delta=\{g_1,g_2\}$ and
Hilbert basis $\cG_{\Delta,\Z}=\{g_1,g_2,g_3\}$. We see $Z^o(\Delta)$ and
$Z^o(S(\Delta))$ illustrated in Figure \ref{Fig:smooth-triangulation}.

\begin{figure}[h!]
\tikzstyle{help lines} =[dashed]
\begin{tikzpicture}[scale=0.5]

\draw[->](-0.3,0) -- (15,0);
\draw[->] (0,-0.3) -- (0,9.4);
\node[below] at (15,0) {$x_1$};
\node[left] at (0,9.4) {$x_2$};

\filldraw[fill=gray!50, draw=black!, line width=1pt] (0,0) -- (2,4) -- (3,0) -- (0,0);
\filldraw[fill=blue] (1,0) circle (4pt);
\filldraw[fill=blue] (1,1) circle (4pt);
\filldraw[fill=blue] (1,2) circle (4pt);
\filldraw[fill=black] (0,0) circle (2pt);
\filldraw[fill=black] (3,0) circle (2pt);
\filldraw[fill=black] (2,4) circle (2pt);

\node[below left] at (0,0) {${\mathbf 0}$};
\node[below] at (1,0)  {$g_1$};
\node[above right] at (1,1) {$g_3$};
\node[above left] at (1,2) {$g_2$};
\node[below] at (3,0) {$3g_1$};
\node[above left] at (2,4) {$2g_2$};

\draw[help lines,step = 1cm, gray!70,ultra thin] (0,0) grid (14,9); 
\draw[help lines,thick] (2,4) -- (4.5,9); \draw[help lines,thick] (4,8)--(14,8);
\draw[help lines,thick] (3,0) -- (7.5,9); \draw[help lines,thick] (2,4)--(14,4);
\draw[help lines,thick] (6,0)--(10.5,9); \draw[help lines, thick] (12,0)--(14,4);
\draw[help lines,thick] (9,0)--(13.5,9);

\draw[thin] (1,2)--(2,2)--(1,0); \draw (0.5,0.5) --(-1,1);\node[left] at (-1,1) {$Z^o(S(\Delta))$};

\node[right] at (2,2){$Z^o(\Delta)$};
\end{tikzpicture}
\caption{}
\label{Fig:smooth-triangulation}
\end{figure}
\end{examples}

Given an integral convex polytope $\Delta$ whose origin-less facets are
simplicial, we study the representations of an integral point in $C(\Delta)$
by non-negative linear combination of integral points in  $\Delta$
if exists. Note that by our definition in Section \ref{S:smooth},
$\Delta$ is (resp., asymptotically) smooth if and only if all (resp., but finitely many)
integral points in $C(\Delta)$ have
such representations.
Let $\Delta=\cup_{\delta} \Delta_\delta$ be its facial triangulation
in (\ref{E:division}) according to its origin-less faces $\delta$.
Then each $\Delta_\delta$ is a simplex pointed at the origin
and $Z^o(\Delta_\delta)$ is its zonotope.

\begin{lemma}\label{L:key}
Let $\Delta$ be an integral convex polytope of dimension $n$ in $\R^n$ containing origin
whose origin-less facets are
simplicial. Let $\Delta$ contain a subset $\cG$ that generates $S(\Delta)$ as a monoid.
Let $\cG_o\subseteq \cG$ generate the monoid $S(\Delta)$. 
For every $v\in S(\Delta)$,
suppose $v\in S(\Delta_\delta)$ for an origin-less
facet $\delta$, then there is a representation
\begin{eqnarray*}
v&=& \sum_{j\in \Vert(\Delta_\delta)} \alpha_j j + R_{\Delta_\delta}(v),
\end{eqnarray*}
for unique
$\alpha_j\in\Z_{\geq 0}$ and
unique $R_{\Delta_\delta}(v)$ in
$
Z^o(\Delta_\delta)\cap\Z^n.
$

Let $\cG_{\Delta_\delta}=\{g_1',\ldots,g_n'\}$ be the primitive
generating set for $C(\Delta_\delta)$. Let $\Vert(\Delta_\delta)=\{d_1g_1',\ldots,d_ng_n'\}$
as in Proposition \ref{P:disc}. Then  $v=\sum_{i=1}^{n}v_i g_i'$
for unique $v_i\in\Q_{\geq 0}$, such that
$$
v=\sum_{i=1}^{n}\lfloor\frac{v_i}{d_i} \rfloor (d_ig_i')+R_{\Delta_\delta}(v).
$$
Fix a complete order of the set $S:=\cG_o-\Vert(\Delta_\delta)$, 
then there is a unique representation
$$
R_{\Delta_\delta}(v) = \sum_{j\in S} \alpha_j j
$$
such that $\alpha:=(\alpha_j)_{j\in S}$ has the lowest graded lexicographic order
with respect to $S$
as defined in Section \ref{S:glex}.
\end{lemma}
\begin{proof}
Since $\Delta_\delta$ is simplex, $\cG_{\Delta_\delta}$ are generators for
the monoid $S(\Delta_\delta)\Q_{\geq 0}$, we have unique $v_i\in\Q_{\geq 0}$
such that $v=\sum_{i=1}^{n} v_i g_i'$.
For every $i$ write
$v_i = \lfloor \frac{v_i}{d_i}\rfloor(d_ig_i') + (\{ \frac{v_i}{d_i}\}d_i)g_i'$.
Then
$$R_{\Delta_\delta}(v)= \sum_{i=1}^{n}(\{ \frac{v_i}{d_i}\}d_i)g_i' \in Z^o(\Delta_\delta)\cap \Z^n.$$
So $v$ uniquely determines these $v_i$'s,
and subsequently these $\alpha_i$ and $R_{\Delta_\delta}(v)$.
The last statement regarding representation of
$R_{\Delta_\delta}(v)$ in terms of $S$ is clear by its very definition and
the order we endowed upon $S$.
\end{proof}

\begin{remark}
\label{R:vertex-rep}
If $\Delta$ in Lemma \ref{L:key} is not simplicial at an origin-less facet $\delta$,
we can still define vertex representation through replacing 
$\delta$ by its triangulation component $\delta_\ell$'s for 
the point configuration $(\Delta,\cG)$
as in Proposition \ref{P:triangulation}
(i.e., for each origin-less facet $\delta$ we may have a
regular triangulation $\delta=\cup_\ell \delta_\ell$).
Suppose $\Vert(\Delta_{\delta_\ell}) =\{d_1g_1',\ldots,d_mg_m'\} 
\subseteq \Vert(\Delta_\delta)$ for
primitive generating set $\cG_{\Delta_{\delta_\ell}}=\{g_1',\ldots,g_m'\}$.
In the statement of Lemma \ref{L:key}, we have a vertex representation
of any $v\in S(\Delta_{\delta_\ell})$ as follows:
$$
v = \sum_{i=1}^{m}\lfloor\frac{v_i}{d_i} \rfloor (d_ig_i') + R_{\Delta_{\delta_\ell}}(v)
$$
with unique $R_{\Delta_{\delta_\ell}}(v)\in Z^o(\Delta_{\delta_\ell})$.
\end{remark}

We call $R_{\Delta_\delta}(v)$ the {\em vertex residue} of $v$,
and the representation in Lemma \ref{L:key} the
{\em vertex representation}.
Notice that 
an arbitrary representation is of the form
$R_{\Delta_\delta}(v)=\sum_{j\in\cG-\partial\Delta} \alpha_j j$ for some
$\alpha_j\in\Z_{\geq 0}$.
If $p$ is a prime number then its vertex residue relative to $\Delta_\delta$
is  $R_{\Delta_\delta}(p):=R_{\Delta_\delta}((p,\ldots,p))$ lying in $\prod_{i=1}^{n} (\Z/d_i\Z)^*$.
Note that $R_{\Delta_\delta}(v)$ only depends on the vertex residue of $v$.

\subsection{Integral weight functions}
\label{S:weight}

For any origin-less facet $\delta$ of $\Delta$,
let $d_\delta\in \Q^n$ define the hyperplane $\pscal{d_\delta,x}=1$
of $\delta$.
Let $v\in S(\Delta_{\delta'})$ for some origin-less facet $\delta'$ of $\Delta$.
For any origin-less face $\delta$ of $\Delta$
we define the {\em weight of $v$ relative to} $\Delta_{\delta}$
by
$$
w_{\Delta_\delta}(v):=\pscal{d_\delta,v}.
$$
If $v\in S(\Delta_\delta)$ or $v\not\in S(\Delta)$
then $w_{\Delta_\delta}(v)=w_\Delta(v)$,
namely it is the same as the classical weight function (see \cite{AS89}  and \cite{Wan92}).
In other words, it is equal to the least $c\in \Q_{\geq 0}$ so that $v\in c\Delta$
if such $c$ exists and is equal to $\infty$ if otherwise.
If $v\in S(\Delta)-S(\Delta_\delta)$, then our weight function is
different. The following lemma is easy to derive hence its proof is omitted.

\begin{lemma}
\label{L:weight}
(1) Let $\Delta$ be simplex pointed at origin, with the primitive generating
set $\cG_\Delta=\{g_1,\ldots,g_n\}$ and $\Vert(\Delta)=\{d_1 g_1,\ldots,d_n g_n\}$
by Proposition \ref{P:disc}.
Write $x=(x_1,\ldots,x_n)$ for a vector of  variables and
$d_\delta:=(1/d_1,\ldots,1/d_n)$ in $\Q^n$. Then (the scalar product)
$\pscal{d_\delta,x}=\sum_{i=1}^{n}\frac{x_i}{d_i}=1$
defines the facet of $\Delta$ not containing the origin.
Then for every $v\in S(\Delta)$ we have
$
w_\Delta(v)=\pscal{d_\delta,v}=\sum_{i=1}^{n}v_i/d_i.
$

(2)
Let $\Delta$ be an integral convex polytope with origin
and let $\Delta=\cup_\delta \Delta_\delta$ be its closed facial subdivision as
in (\ref{E:closed-facial}),
and let $v\in C(\Delta_{\delta'})$ for some origin-less face $\delta'$ of $\Delta$.
Then for any origin-less face $\delta$ we have
$w_{\Delta_\delta}(v)\leq w_\Delta(v)$ where
the equality holds if and only if $v\in C(\delta\cap\delta')$.
\end{lemma}

The convexity of $\Delta$ implies the `convexity' of
the $\Q$-valued weight map $w_\Delta: \Z^n\rightarrow \Q_{\geq 0}\cup \{\infty\}$
as we see in the following lemma due to Adolphson-Sperber (proof omitted).

\begin{lemma}[see \cite{AS89} and \cite{Wan93}]
\label{L:w(v)}
\begin{enumerate}
\item We have
$w_\Delta(cv)=c w_\Delta(v)$ for any $c\in\Q_{\geq 0}$.
\item We have $w_\Delta(v+v')\leq w_{\Delta} (v)+w_{\Delta}(v')$ with equality holds if and only if
$v$ and $v'$ are cofacial, i.e., $v/w_\Delta (v)$ and $v'/w_{\Delta}(v')$ lie on the same closed facet
of $\Delta$.
\item Suppose $\sum_{j\in \Delta\cap\Z^n} v_j j =v$ for $v_j\in \Q_{\geq 0}$.
Let $J'$ be the subset of all $j'$'s in $\Delta\cap\Z^n$ such that $v_{j'}\neq 0$.
Then $w_\Delta(v)\leq\sum_{j\in \Delta\cap\Z^n} v_j$
where the equality holds if and only if the nonzero integral points in $J'$
lie on the same closed origin-less facet in $\Delta$.
\end{enumerate}
\end{lemma}

Let $\Delta$ be an integral convex polytope in $\R^n$
containing the origin. For any subset $\Sigma$ in $\Delta$
let $\Z^+(\Sigma)$ be the lattice cone generated by integral points of $\Sigma$
over $\Z_{\geq 0}$. Let $\cG$ be a subset of $\Delta\cap\Z^n$
and  let $v\in \Z^n$, we  define the
{\em integral weight function supported on} $\cG$
as followings:
\begin{eqnarray}\label{E:w-2}
w_{\cG,\Z}(v) &=&\min (\sum_{j\in \cG} u_{v,j})
\end{eqnarray}
where the minimum was taken over
all solutions $u_{v,j}\in\Z_{\geq 0}$ to representations $v=\sum_{j\in \cG} u_{v,j} j$;
and where we have $w_{\cG,\Z}(v)=\infty$ if no such representation exists.
The following properties are clear and easy to prove.

\begin{lemma}
\label{L:w_Z(v)}
\begin{enumerate}
\item We have $w_{\cG,\Z}(v+v')\leq w_{\cG,\Z}(v)+w_{\cG,\Z}(v')$
for any $v,v'$ in $S(\Delta)$.
\item We have
$w_{\cG,\Z}(c v) \leq c w_{\cG,\Z}(v)$ for any positive integer $c$.
\item We have $w_{\cG,\Z}(v)\geq w_\Delta(v)$ for every $v\in S(\Delta)$.
\end{enumerate}
\end{lemma}

Thus we have defined a function
$w_{\cG,\Z}: \Z^n \rightarrow \Z_{\geq 0}\cup\{\infty\}$.
Notice that  $w_{\cG,\Z}(v) \in \Z_{\geq 0}$ if and only if $v\in \Z^+(\cG)$
and $w_{\cG,\Z}(v)=\infty$ if otherwises.
If $\cG\subset \cG'$ are two subsets of $\Delta\cap\Z^n$
then $w_{\cG',\Z}(v)\leq w_{\cG,\Z}(v)$.

\subsection{Boundedness of integral weight functions}
\label{S:approx}

Let $\Delta$ be an integral convex Newton polytope of dimension $n$
in $\R^n$ containing the origin. 
Let $\cG$ be a subset of $\Delta\cap \Z^n$.
Let $p$ be a prime.
For $r,s,pr-s\in S(\Delta)$, we define
\begin{eqnarray}
\label{E:b_G}
b_{\cG,\Z}(pr-s) &:=& \frac{w_{\cG,\Z}(pr-s)+w_\Delta(s)-w_\Delta(r)}{p-1}.
\end{eqnarray}

Let $\delta$ be an origin-less facet of $\Delta$.
Then we have an open facial subdivision of the cone  
$C(\Delta_\delta) =\cup_{i}\Sigma_i$ as that in (\ref{E:open-facial}).
We arrange them so that $\dim \Sigma_i\leq \dim \Sigma_{i+1}$.
For any $r,s\in S(\Delta_\delta)_{\leq k}$,
we define an order $r\geq _\delta s$ if 
$r\in \Sigma_i$ and $s\in \Sigma_j$ and $i\geq j$.
Then we have
$r\geq_\delta s$ if and only if $pr-s\in S(\Delta_\delta)_{\leq k}$ for $p$ large enough;
or equivalently $pr-s,r$ are cofacial for $p$ large enough.

\begin{theorem}
\label{T:key}
Let $\Delta$ be an integral convex polytope of dimension $n$ in $\R^n$
containing the origin.
Let $N_\Delta:= n D(\Delta)$ where $D(\Delta)$ is the least positive integer
such that $w_\Delta(v)\in\frac{1}{D(\Delta)}\Z$ for all $v\in S(\Delta)$ (as defined in
Proposition \ref{P:disc}).
Let  $\cG$ be a subset of $\Delta\cap\Z^n$ containing  $\Vert(\Delta)$ and
$\cG$ generates the monoid $S(\Delta)$ up to finitely many points.
Then for any $v\in S(\Delta)$ and for $p$ large enough we have
$$
w_\Delta(v)
\leq w_{\cG,\Z}(v)
\leq w_\Delta(v)+ N_\Delta.
$$
If $r\in S(\Delta_\delta)^o$ and $s\in S(\Delta)$; or 
if $r,s\in S(\Delta_\delta)$ with $r\geq_\delta s$,  and if $p$ is large enough, 
then $pr-s\in S(\Delta_\delta)$ and 
$$
w_\Delta(r)\leq b_{\cG,\Z}(pr-s)\leq w_\Delta(r)+\frac{N_\Delta}{p-1}.
$$
If $r,s\in S(\Delta_\delta)$ with $r<_\delta s$, 
then $p$ is large enough then $w_\Delta(r)\lneq b_{\cG,\Z}(pr-s)$.
\end{theorem}

\begin{proof}
For simplicity, we shall prove the theorem
under the hypothesis that $\cG$ generates $S(\Delta)$ since
under the hypothesis $p$ is large enough our argument is not affected.
On the other hand, we may also assume that each
origin-less facet $\delta$ of $\Delta$ is simplex
by our Remark \ref{R:vertex-rep}.

It is easy to derive that $w_{\Delta}(v) \leq w_{\cG,\Z}(v)$.
It remains to prove the second inequality.
By the basic properties of integral weight function
in Lemma \ref{L:w_Z(v)} it reduces to prove that there is a vertex representation
of $v$ that gives rise to the desired upper bound of $w_{\cG,\Z}(v)$.
Let $v\in S(\Delta_\delta)$
for some origin-less face $\delta$ of $\Delta$ where
$\Delta_\delta$ is the closed facial subdivision of $\Delta$ at $\delta$ as in (\ref{E:division}).
By our hypothesis, we have Hilbert basis
$\cG_{\Delta,\Z}\subseteq \cG$ and we fix its order.
Let $\cG_{\Delta_\delta}=\{g_1',\ldots,g_n'\}$
be the primitive generating set for $C(\Delta_\delta)$.
By Proposition \ref{P:disc}, we have
$\Vert(\Delta_\delta) = \{d_1g_1',\ldots,d_ng_n'\}$ for some $d_i\in\Z_{\geq 1}$.
By our hypothesis we know $\Vert(\Delta_\delta)\subseteq \Vert(\Delta)\subseteq \cG$.
For any $v=\sum_{i=1}^{n}v_i g_i'$ for $v_i\in \Q_{\geq 0}$ in $S(\Delta_\delta)$,
there is a unique vertex representation
$v=\sum_{i=1}^{n}\pfloor{\frac{v_i}{d_i}}(d_ig_i')+R_{\Delta_\delta}(v)$
with a unique $R_{\Delta_\delta}(v) \in Z^o(\Delta_\delta)\cap \Z^n$
by Lemma \ref{L:key}.
If we let $S:=\cG-\Vert(\Delta_\delta)$ be fixed with a complete order, then 
there is a unique representation
$R_{\Delta_\delta}(v)= \sum_{j\in S} \alpha_{v,j} j $
for some $\alpha_{v,j}\in \Z_{\geq 0}$.
By Proposition \ref{P:disc} we have
$\frac{1}{D(\Delta)}\leq w_{\Delta_\delta}(j)\leq 1$
if $j\in S(\Delta_\delta)$
and $\frac{1}{D(\Delta)}\leq w_{\Delta_\delta}(j) \leq   w_\Delta(j)\leq 1$
if otherwise due to the convexity of $\Delta$ (see Lemma \ref{L:weight}).
Let $d_\delta$ be the vector that
defines the face $\delta$ as in Lemma \ref{L:weight},
take scalar product with the vector $d_\delta$ on the following equation
$$
R_{\Delta_\delta}(v)=\sum_{j\in S} \alpha_{v,j} j.
$$
Since $R_{\Delta_\delta}(v)$ lies in $S(\Delta_\delta)$, we get
$$
w_{\Delta_\delta}(R_{\Delta_\delta}(v))
=w_\Delta(R_{\Delta_\delta}(v))
=\sum_{j\in S}
 \alpha_{v,j} w_\Delta (j)
\geq \frac{1}{D(\Delta)}\sum_{j\in S}\alpha_{v,j}.
$$
On the other hand, since $R_{\Delta_\delta}(v)\in Z^o(\Delta_\delta)$
we have
$w_\Delta(R_{\Delta_\delta}(v)) < n$; so we have
$\sum_{j\in S} \alpha_{v,j} < nD(\Delta)$.
Therefore, by Lemma \ref{L:w_Z(v)} we have
$$
w_{\cG,\Z}(R_{\Delta_\delta}(v))
\leq \sum_{j\in S}\alpha_{v,j} w_{\cG,\Z}(j)
\leq \sum_{j\in S} \alpha_{v,j}
< n D(\Delta).
$$
Thus by Lemmas \ref{L:weight} and \ref{L:w_Z(v)} again we have
$$
w_{\cG,\Z}(v)
\leq \sum_{i=1}^{n}\lfloor\frac{v_i}{d_i} \rfloor + w_{\cG,\Z}(R_{\Delta_\delta}(v))
\leq w_\Delta(v) + w_{\cG,\Z}(R_{\Delta_\delta}(v))
< w_\Delta(v) + n D(\Delta).
$$

Suppose $w_{\cG,\Z}(pr-s)=\sum_{j\in\cG} u_{pr-s,j}$ for $u_{pr-s,j}\in\Z_{\geq 0}$,
we apply the above result to $v=pr-s$ in $S(\Delta_\delta)$ then
$$
w_\Delta(pr-s)\leq w_{\cG,\Z}(pr-s)=\sum_{j\in \cG}u_{pr-s,j}
<
w_\Delta(pr-s) + N_\Delta.
$$
Write $b_\Delta(pr-s) := (w_\Delta(pr-s)+w_\Delta(s)-w_\Delta(r))/(p-1)$.
By Lemma \ref{L:w(v)} we have
$
w_\Delta (pr-s) +w_\Delta(s) \geq p w_\Delta(r)
$
and hence
$
b_\Delta(pr-s)
\geq w_\Delta(r)
$
where the equality holds if and only if $pr-s$ and $s$ are cofacial.
The last condition is satisfied if $r\in S(\Delta_\delta)^o$ and $s\in S(\Delta)$
or $r,s\in S(\Delta_\delta)$ and $r\geq_\delta s$ for $p$ large enough .
Thus we have
$$
w_\Delta(r) =b_\Delta (pr-s)\leq b_{\cG,\Z}(pr-s)\leq w_\Delta(r)+\frac{N_\Delta}{p-1}.
$$
On the other hand, suppose $r,s\in S(\Delta_\delta)$ with $r<_\delta s$.
Then $s,pr-s$ are not cofacial for any $p$ large enough, and hence
by Lemma \ref{L:w(v)} we have
$w_\Delta(r)\lneq b_\Delta(pr-s)\leq b_{\cG,\Z}(pr-s)$.
\end{proof}

In the following we shall explore the independence of $p$ in vertex representations
defined in Lemma \ref{L:key}. It will be crucial for the construction of
global Hasse polynomials in Section \ref{S:dwork}.

\begin{theorem}
\label{T:approx}
Let our hypothesis and notation be as that in Theorem \ref{T:key},
and let $\Delta$ be simplicial at all origin-less facets.
Let $\partial\Delta$ be the set of integral points on origin-less facets of $\Delta$.
Let $\delta$ be some origin-less facet of $\Delta$, 
$r\in S(\Delta_\delta)$ be bounded. 
Then $pr-s\in S(\Delta_\delta)$ if 
$s\in S(\Delta)$ and $r\in S(\Delta)^o$;
or $r,s\in S(\Delta_\delta)$ and $r\geq _\delta s$ for  $p$ large enough.
Let $pr-s=\sum_{j\in \cG}u_{pr-s,j}j$ for some $u_{pr-s,j}\in\Z_{\geq 0}$
such that $\sum_{j\in\cG}u_{pr-s,j}\leq w_\Delta(pr-s)+N_\Delta$
(e.g., when $w_{\cG,\Z}(pr-s)=\sum_{j\in \cG}u_{pr-s,j}$),
then for every $j\in (\cG-\partial\Delta)\cap\Z^n$ we have
$u_{pr-s,j}\leq n D(\Delta)^2$ is independent of $p$.
\end{theorem}

\begin{proof}
Let $v\in S(\Delta_\delta)$ for an origin-less face $\delta$ of $\Delta$.
If $v=\sum_{j\in \cG} u_{v,j} j$ with $u_{v,j}\in\Z_{\geq 0}$ and
$\sum_{j\in \cG} u_{v,j} \leq w_\Delta(v) + N_\Delta,$
we shall prove that in this paragraph that  for any
$j\in (\cG-\partial\Delta)\cap\Z^n$
we have $u_{v,j} \leq n D(\Delta)^2$.

Since $\cG$ generates $S(\Delta)$ we have
\begin{eqnarray}
\label{E:v}
v&=&\sum_{j\in \cG}u_{v,j} j
\end{eqnarray}
for $u_{v,j}\in\Z_{\geq 0}$.
Let $\pscal{d_\delta,x}=1$ denote the equation of the facet $\delta$ (see Lemma \ref{L:weight}).
Take an obvious partition $\Delta=\Delta_\delta \cup \Delta_\delta^c$.
Then take scalar product with $d_\delta$ on (\ref{E:v}) on both sides,
we have by Lemma \ref{L:weight},
$$
w_\Delta(v) = \sum_{j\in\cG}u_{v,j}w_{\Delta_\delta}(v)=
\sum_{j\in \cG\cap\Delta_\delta} u_{v,j} w_\Delta(j) + \sum_{j\in\cG\cap\Delta_\delta^c}u_{v,j}w_{\Delta_\delta}(j).
$$
By our hypothesis we have some $0\leq M\leq N_\Delta$ such that
\begin{eqnarray}
\label{E:MM}
\sum_{j\in \cG}u_{v,j}- w_\Delta(v) = M,
\end{eqnarray}
hence $\sum_{j\in\cG}u_{v,j}(1-w_{\Delta_\delta}(j))=M$.
Thus we have
$$
\sum_{j\in\cG\cap\delta}u_{v,j}(1-w_{\Delta_\delta}(j))+
\sum_{j\in \cG\cap(\Delta_\delta-\delta)} u_{v,j}(1-w_\Delta(j)) +
\sum_{j\in \cG\cap \Delta_\delta^c} u_{v,j}(1-w_{\Delta_\delta}(j))
= M.$$
Then $w_\Delta(j) \lneq 1$ if $j\in \Delta_\delta-\delta$;
and $w_{\Delta_\delta}(j)\lneq 1$ if $j\in \Delta_\delta^c$
due to convexity of $\Delta$ (see Lemma \ref{L:weight}).
For $j\in\Delta_\delta-\delta$,
this implies $u_{v,j}(1-w_\Delta(j)) \leq M$
and since $1-w_\Delta(j)\geq 1/D(\Delta_\delta)$,
we have $0<u_{v,j}\leq M\cdot D(\Delta_\delta)\leq n \cdot D(\Delta)^2$.
For $j$ lies in $\Delta_\delta^c$,
this implies $u_{v,j}(1-w_{\Delta_\delta}(j))\leq M$.
But $w_{\Delta_\delta}(j)\lneq 1$ implies that $1-w_{\Delta_\delta}(j)\geq 1/D(\Delta)$,
and hence $0<u_{v,j} \leq  n \cdot D(\Delta)^2$.
This proves that for all lattice points
$j\in (\cG-\partial\Delta)\cap\Z^n$
we have  $u_{v,j} \leq  n\cdot D(\Delta)^2$.
Now our statement follows from the above paragraphs by
letting $v=pr-s$.
\end{proof}

\begin{remark}
\label{R:simplicial}
If $\Delta$ is not necessarily simplicial at origin-less facet,
Theorem \ref{T:approx} does not generally hold.
\end{remark}

\subsection{Bound for smooth simplex $\Delta$}

For this subsection we assume $\Delta$ is a simplex pointed at the origin,
namely it has only one origin-less facet which is a simplex,
then we achieve a stronger bounds on our estimates of
its integral weight function.

\begin{lemma}
\label{L:key-simplex}
Let $\Delta$ be simplex pointed at origin with a primitive generating set
$\cG_\Delta=\{g_1,\ldots,g_n\}$ and
Hilbert basis $\cG_{\Delta,\Z}=\{g_1,\ldots,g_n,\ldots,g_t\}$ (by Proposition \ref{P:Hilbert}).
Let  $Z^o(S(\Delta))$ be the fundamental parallelepiped of the lattice cone $S(\Delta)$,
i.e.,
$
Z^o(S(\Delta)):=
\{\sum_{i=1}^{n} \alpha_i g_i | 0\leq \alpha_i <1\}.
$
\begin{enumerate}
\item
Then $R_{\Delta}(v)$ has a unique representation
$$
R_{\Delta}(v) = \sum_{i=1}^{n} c_i g_i + R_{\Delta}^o(v)
$$
for  $c_i=\pfloor{\{\frac{v_i}{d_i}\}d_i}\in\Z$
with
$0\leq c_i\leq d_i-1$ and $R_{\Delta}^o(v)
=\sum_{i=1}^{n}\{\{\frac{v_i}{d_i}\}d_i\}g_i
\in Z^o(S(\Delta))\cap \Z^n$.

\item
Suppose  $\Delta$ is smooth (i.e., $\Delta$ contains $\cG_{\Delta,\Z}$).
Suppose we ordered $g_{n+1},\ldots,g_t$ such that
$w_\Delta(g_{n+1})\leq \cdots \leq w_\Delta(g_t)$,
then we have a unique representation
$R_{\Delta}^o(v)=\sum_{i=n+1}^{t} c_i g_i$
for $c_i\in\Z_{\geq 0}$ such that at most $2n-2$ of
$c_i\neq 0$ and $c_{n+1},\cdots,c_t$ are the minimal possible.
In this case, let $d_i=\pfloor{1/w_\Delta(g_i)}$ for $n+1\leq i\leq t$
and write $c_i=\pfloor{\frac{c_i}{d_i}}d_i+r_i'$ for some $r_i'\in \Z/d_i\Z$
then
$$
R_{\Delta}^o(v)=\sum_{i=n+1}^{t}\pfloor{\frac{c_i}{d_i}}(d_ig_i)
+\sum_{i=n+1}^{t} (r_i'g_i).
$$
\end{enumerate}
\end{lemma}
\begin{proof}
In the last statement the existence of such representation of $R_{\Delta}^o(v)$
with at most $2n-2$ nonzero terms in $\cG_{\Delta,\Z}$ is
due to \cite{Seb90}, the uniqueness is clear by definition.
The rest of this lemma is elementary hence details are left out.
\end{proof}

An important special case of $\cG$ we shall explore in this subsection is
when $\cG=\Delta\cap\Z^n$ in which case we denote it by
\begin{eqnarray}\label{E:w_Z}
w_{\Delta,\Z}(v) &=& w_{\cG,\Z}(v) = \min \sum_{j\in \Delta\cap \Z^n}u_{v,j}
\end{eqnarray}
where the minimum is taken over all possible sums
$v = \sum_{j\in \Delta\cap\Z^n} u_{v,j} j$ with $u_{v,j}\in\Z_{\geq 0}$.
If no such solution exists then $w_{\Delta,\Z}(v)=\infty$.
We remark that
our integral weight function $w_{\Delta,\Z}$ is intimately
related to the height function on Hilbert basis, the bound of
which has been actively pursued, see \cite{Seb90} or \cite{HW97}.

Let $r,s\in S(\Delta)$ be bounded and let $N_\Delta=4n^2-n-2$.
When $pr-s\in S(\Delta)$, let
\begin{equation}
\label{E:b-definition}
\left\{
\begin{array}{ccc}
b_\Delta(pr-s) &:=& \frac{w_{\Delta}(pr-s)+w_\Delta(s)-w_\Delta(r)}{p-1}\\
b_{\Delta,\Z}(pr-s) &:=& \frac{w_{\Delta,\Z}(pr-s)+w_\Delta(s)-w_\Delta(r)}{p-1}.
\end{array}
\right.
\end{equation}
Note that we
have
$b_{\Delta,\Z}(pr-s) = b_{\cG,\Z}(pr-s)$ for $\cG=\Delta\cap\Z^n$ as
defined in (\ref{E:b_G}).

\begin{proposition}
\label{P:key-simplex}
Let $\Delta$ be a smooth simplex with primitive generating set $\cG_{\Delta}$
and the Hilbert basis $\cG_{\Delta,\Z}=\{g_1,\ldots,g_t\}$ (with fixed  order).
Let $N_\Delta:=4n^2-n-2$.
Let $v\in S(\Delta)$ and write $v=\sum_{i=1}^{n}v_i g_i$ for some $v_i\in \Q_{\geq 0}$
(by Proposition \ref{P:Hilbert}).

(1) Then we have
$$w_\Delta(v)\leq w_{\Delta,\Z}(v)
\leq w_\Delta(v)+N_\Delta.$$

(2) We have for $p$ large enough
\begin{eqnarray*}
w_\Delta(r) \leq b_\Delta (pr-s)\leq b_{\Delta,\Z}(pr-s)
\leq b_\Delta(pr-s)+\frac{N_\Delta}{p-1}.
\end{eqnarray*}
The first equality above holds if and only if $r$ and $s$ are cofacial.
If $r,s$ are cofacial then we have
\begin{eqnarray*}
w_\Delta(r)\leq & b_{\Delta,\Z}(pr-s) &\leq w_\Delta(r)+\frac{N_\Delta}{p-1};
\end{eqnarray*}
otherwise, we have $w_\Delta(r)\lneq b_{\Delta,\Z}(pr-s)$.
\end{proposition}

\begin{proof}
(1)
By Proposition \ref{T:key} it reduced to prove the second inequality.
By the representation of $v$ in (\ref{E:v}) we have immediately
$w_{\Delta,\Z}(v)\leq \sum_{i=1}^{n}\lfloor \frac{v_i}{d_i}\rfloor + w_{\Delta,\Z}(R_{\Delta_\delta}(v))$
by definition.
The boundedness of $w_{\Delta,\Z}(R_{\Delta_\delta}(v))\leq N_\Delta$ is proved below.
By the representation of $R_{\Delta_\delta}(v)$  in Lemma \ref{L:key-simplex} (2)
and notice that $c_ig_i\in \Delta$
we have $w_{\Delta,\Z}(R_{\Delta_\delta}(v))\leq n+w_{\Delta,\Z}(R_{\Delta_\delta}^o(v))$.
By Lemma \ref{L:key-simplex}(2) again we have $R_{\Delta_\delta}^o(v)=\sum_{i=n+1}^{t} c'_ig_i$ for some $c'_i\in \Z_{\geq 0}$.
It suffices to give a canonical upper bound for the weight of each $c'_ig_i$.
To ease of notation, let $g$ be any Hilbert basis generator in $\Delta$ and let $c$ be
any positive integer such that $c g \in Z^o(S(\Delta))$, the fundamental parallelepiped
of $S(\Delta)$ defined in Section \ref{S:convex} (see Figure \ref{Fig:smooth-triangulation}
for an illustration).
Let $d$ be the positive integer such that $dg\in \Delta$ and $(d+1) g \not\in \Delta$.
It exists by our hypothesis that $\Delta$ is smooth so all Hilbert basis of the cone $C(\Delta)$
lie in $\Delta$.
Write $g=\sum_{i=1}^{n}\alpha_i g_i$ for some $0\leq \alpha_i <1$,
we have
$
0\leq c\sum_{i=1}^{n} \alpha_i <n
$
and
$
(d+1)\sum_{i=1}^{n}\alpha_i >1
$,
$
d\sum_{i=1}^{n}\alpha_i \leq 1.
$
Thus we have
$
c/d \leq 2 c/(d+1) <  2 n.
$
Since $c g = \pfloor{c/d} (dg) + r g$ with $r = (c\bmod d)$ we have
$w_{\Delta,\Z}(cg) \leq \pfloor{c/d} + 1 \leq 2n$ by the above bound for $c/d$.
This implies that
$w_{\Delta,\Z}(R_{\Delta}^o(v))\leq  (2n-2)(2n+1)$ by \cite{Seb90},
and therefore $w_{\Delta,\Z} (R_{\Delta}(v)) \leq n+w_{\Delta,\Z}(R_{\Delta}^o(v)) \leq 4n^2-n-2$.
The last statement follows immediately.

2)  Use the same argument as that for Theorem \ref{T:key}.
\end{proof}

\section{Integral convex polytopes and Dwork theory}
\label{S:dwork}

Let $E(x)$ be the $p$-adic Artin-Hasse exponential series
$E(x):=\exp(\sum_{i=0}^{\infty} \frac{x^{p^i}}{p^i}) \in (\Z_p\cap \Q)[[x]]$
and $\gamma$ a root of $\log E(x)$ with $\ord_p \gamma=1/(p-1)$.
Write Taylor expansion in variable $x$, we have
$
E(\gamma x) = \sum_{m=0}^{\infty}\lambda_m x^m
$
for some $\lambda_m\in (\Z_p\cap \Q)[\gamma]$ and
$\ord_p \lambda_m \geq \frac{m}{p-1}$ with equality holds.
Here we have $\lambda_m = \frac{\gamma^m}{m!}$ for $0\leq m\leq p-1$, in particular, $\lambda_0=1$.

Let $\cG$ be a subset of integral points in $\Delta$ containing $\Vert(\Delta)$,
and let $A_j$ be a variable with $j\in \cG$.
Let $A:=(A_j)_{j\in \cG}$ be the set of all variables
with subindices in $\cG$ and set $A_{\bf 0}=1$ at origin.
Then $\Q[A]$ is the polynomial ring in variable $A_j$'s
with (subindex) support on $\cG$.
In particular we consider the polynomial ring $(\Z_p\cap \Q)[\gamma^{\Z_{\geq 0}}A]$
with variables $A=(A_j)$ where $j\in \cG\subseteq \Delta\cap \Z^n$ and with Gauss norm.
Every polynomial $P$ here is representated as a sum in the following form
$
P=\sum_{i=0}^{<\infty} \gamma^i h^{(i)}(A)
$
where $h^{(i)}(A)$ is a sum of monomials in $(\Z_p^*\cap \Q)[A]$.
Note that $\ord_p(P)$ is defined as the minimal $p$-adic order of all its coefficients
and hence is equal to the minimal $i/(p-1)\geq 0$ such that $h^{(i)}(A)\neq 0$.

Suppose $\cG$ generates the
lattice cone $S(\Delta)$ over $\Z_{\geq 0}$.
For any $pr-s\in S(\Delta)$ let $F_{pr-s}$ be a {\em Fredholm polynomial
in $(\Z_p\cap \Q)[\gamma^{\Z_{\geq 0}} A]$
supported on} $\cG$
defined by
\begin{eqnarray}
\label{E:F_v}
F_{pr-s}(A) &=& \sum_{Q}(\prod_{j\in \cG} \lambda_{u_{pr-s,j}})
 (\prod_{j\in \cG} A_j^{u_{pr-s,j}})
\end{eqnarray}
where the outer sum is over
the (nonempty) set of all representations
\begin{eqnarray}\label{E:Q}
Q:
&&pr-s=\sum_{j\in \cG} u_{pr-s,j} j, \quad \mbox{ with $u_{pr-s,j}\in\Z_{\geq 0}$.}
\end{eqnarray}
We define normalized Fredholm polynomial as
\begin{eqnarray}
\label{E:M-definition}
M_{pr-s} &:= &\gamma^{w_\Delta(s)-w_\Delta(r)} F_{pr-s}
\end{eqnarray}
lying in $(\Z_p\cap\Q)[\gamma^{\Z_{\geq 0}} A]$.
Let
\begin{eqnarray}
\label{E:M}
\M &:=& (M_{pr-s})_{r,s\in S(\Delta)}.
\end{eqnarray}
Let $S(\Delta)=\cup_{\delta\in\cF^o(\Delta) } S(\Delta_\delta)^o$ be the open facial subdivision
in (\ref{E:open-facial}). 
For any origin-less facet $\delta$ of $\Delta$ for this section we let
\begin{eqnarray}
\label{E:M_delta}
\M_\delta &:=& (M_{pr-s})_{r,s\in S(\Delta_\delta)}.
\end{eqnarray}

\begin{remark}\label{R:degree}
We shall observe that for $p$ large enough we have
\begin{eqnarray*}
M_{pr-s} &=& \sum_{i=0}^{<\infty}\gamma^{(p-1)w_\Delta(r) + i} G_{pr-s}^{(i)}
\end{eqnarray*}
where $G_{pr-s}^{(i)} \in (\Z_p^*\cap\Q)[A]$
is of the following form
\begin{eqnarray*}
G_{pr-s}^{(i)} &=& \sum_{Q} u_{Q,p}\prod_{j\in \cG}A_j^{u_{pr-s,j}}
\end{eqnarray*}
where the sum ranges over all such representations $Q$ in (\ref{E:Q}) above and
for some $u_{Q,p}\in \Z_p^*\cap \Q$.
Hence $G_{pr-s}^{(i)}$ is homogenous with
$
\deg(G_{pr-s}^{(i)})
=
\sum_{j\in\cG}u_{pr-s,j} = w_\Delta(pr-s)+i
$.
For $r,s$ cofacial and $r$ is large enough
the $p$-adic order of coefficient of $G_{pr-s}^{(i)}$
is related to its degree as follows:
$$
\ord_p(\coeff(G_{pr-s}^{(i)}))
=
w_\Delta(r)+\frac{i}{p-1}
=
\frac{\deg(G_{pr-s}^{(i)})+w_\Delta(s)-w_\Delta(r)}{p-1}.
$$
\end{remark}

\subsection{Fredholm polynomials and determinants}

In this section we prove three key ingredients for our proofs
in Section \ref{S:proof}, more precisely, 
we have Theorems \ref{T:mini-2}, \ref{T:Zariski} and \ref{T:nonzero} for 
Theorems \ref{T:Wan1.11-2} (Theorem \ref{T:Wan1.11}), 
\ref{T:Wan-2} (Theorem \ref{T:Wan1.12}), and 
\ref{T:allQ}, respectively.

Suppose  $\Delta$ is an integral polytope of dimension $n$ in $\R^n$ containing the origin.
Let $D(\Delta)$ be the least positive integer such that
$w_\Delta(v)\in \frac{1}{D(\Delta)}\Z$ for all $v\in S(\Delta)$ (as in Proposition \ref{P:disc}).
Let $\Delta=\cup_{\delta}\Delta_\delta$ be closed facial subdivision of $\Delta$
as in (\ref{E:division}).
Let $C(\Delta) =\cup_{\delta} C(\Delta_\delta)$
where $\delta$ ranges over
all origin-less facets of $\Delta$ be the
closed facial subdivision as in (\ref{E:closed-facial}).
Let $\cG\subseteq \Delta\cap \Z^n$ that generates $S(\Delta)$.
For any facet $\delta$ and any $k\in\Z_{\geq 0}$ 
write $N_{\delta,k}:=|S(\Delta_\delta)_{\leq k}|$.

\begin{lemma}
\label{L:D-2}
Let $\Delta$ be an integral convex polytope of dimension 
$n$ in $\R^n$ containing the origin, simplicial at all origin-less facets,
and let $\delta$ be an origin-less facet of $\Delta$.
Let $\cG\subseteq \Delta\cap\Z^n$ generate $S(\Delta)$
and $S=\cG-\Vert(\Delta_\delta)$.
Let $R$ be a vertex residue with respect to $\Delta_\delta$ for a prime,
that is $R\in \prod_{i=1}^{n}(\Z/d_i\Z)^*$ for some $d_i$'s (as in Proposition \ref{P:disc}).
For $r,s\in S(\Delta_\delta)$ and $r>_\delta s$ write 
$R_{\Delta_\delta}(Rr-s):=\sum_{j\in S}\alpha_{Rr-s,j} j$
as in Lemma \ref{L:key}. 
Let $\dot{Q}(Rr-s) = \prod_{j\in S} A_j^{\alpha_{Rr-s,j}}$.
For $0\leq k\leq k_\Delta$, 
let $\D_\delta^{[N_{\delta,k}]}:=(\dot{Q}(Rr-s))_{r,s\in S(\Delta_\delta)_{\leq k}}$ be a monomial matrix.
Then there is a unique monomial in the formal expansion of
$\det(\D_\delta^{[N_{\delta,k}]})$ in $\Z[A]$ with $A=(A_j)_{j\in S}$
of the lowest graded lexicographic order with respect to $\cG_{\Delta,\Z}$ 
(as defined in Section \ref{S:glex}). This monomial has degree $< N_\Delta N_{\delta,k}$.
\end{lemma}

\begin{proof}
Write $t:=|S|$.
First we claim that in each row of $\D_\delta^{[N_{\delta,k}]}$
has $N_{\delta,k}$ distinct $t$-tuple exponent vectors
$(\alpha_{Rr-s,j})_{s\in S(\Delta_\delta)_{\leq k}}$ with fixed $r$,
so is in each column has $N_{\delta,k}$ distinct $t$-tuple exponent vectors
$(\alpha_{Rr-s,j})_{r\in S(\Delta_\delta)_{\leq k}}$ with fixed $s$.
This is clear since $R\in \prod_{i=1}^{n}(\Z/d_i\Z)^*$.
Order the set $\Theta$ of all $t$-tuples $\alpha$ in $\Z_{\geq 0}^t$
with (strictly) increasing graded lexicographic order with respect to $\cG$.
We claim there exists a permutation $\sigma\in S_{N_{\delta,k}}$ such that
its corresponding monomial in the formal expansion of 
$\det(\D_\delta^{[N_{\delta,k}]})$ has the lowest graded lexicographic order 
with respect to $S$.
We shall produce this $\sigma$ explicitly:
Let $\alpha_0\in \Theta$ be of lowest graded lexicographic order, 
let  $\sigma(r)=s$ if we have $r,s\in S(\Delta_\delta)$ such that
$(\alpha_{Rr-s,j})_{j\in S}:= \alpha_0$, and we 
cross off the row $r$ and column $s$ in the matrix $\D_\delta^{[N_{\delta,k}]}$
immediately, let $\ell_0$ be the cardinality of all such pairs $(r,s)$. 
Since we have shown above
that each column and row has distinct exponent vectors 
these pairs $(r,s)$ can never lie in the same row or column
and thus the permutation $\sigma$ is well-defined.
Proceed inductively with $\alpha_i$ in $\Theta$  
with $\alpha_i>_{glex} \alpha_{i-1}$, let $\sigma(r)=s$ 
if $(\alpha_{Rr-s,j})_{j\in S} = \alpha_i$ and $r,s\in S(\Delta_\delta)_{\leq k}$
are not crossed off yet on the matrix $\D_\delta^{[N_{\delta,k}]}$, and let
$\ell_i$ be the number of such pairs. We proceed until 
all rows and columns of the matrix $\D_\delta^{[N_{\delta,k}]}$ are crossed off.
Our argument above shows that this procedure produces a unique
permutation $\sigma\in S_{N_{\delta,k}}$, and it is 
immediately clear that this $\sigma$
corresponds to a monomial that has the lowest graded lexicographic
order with respect to $S$.
The degree bound follows from the argument in Theorem \ref{T:key}.
\end{proof}

For any subset $\cG'$ of a set $\cG$ in $\Delta$,
and for any polynomial $P$ in $\Q[A]$ with $A=(A_j)_{j\in \cG}$,
the {\em specialization of $P$} at $\cG'$ over a field $K$ containing $\Q$
is a map 
$\Q[(A_j)_{j\in \cG}]\rightarrow K[(A_j)_{j\in \cG-\cG'}]$
sending $P$ to $P|_{\cG'}$, which evaluates the polynomial $P$
at variables $A_j= a_j$ for all $j\in \cG'$ and $a_j\in K$.

\begin{theorem}
\label{T:mini-2}
Let $\Delta$ be an integral convex polytope of dimension $n$ in $\R^n$ containing the origin.
Let $N_\Delta=n D(\Delta)$. Let $\delta$ be an origin-less facet of $\Delta$.
Suppose $\Vert(\Delta)\subset \cG \subset \Delta\cup\Z^n$ such that
$\cG$ generates the monoid $S(\Delta)$ up to finitely many points.

\begin{enumerate}
\item 
Let  $r\in S(\Delta_\delta)^o$ and $s\in S(\Delta)$, or  
$r,s\in S(\Delta_\delta)$ and $r\geq_\delta s$ (as defined in Section \ref{S:approx}). 
For $p$ large enough
we have $pr-s\in S(\Delta_\delta)$ (see Theorem \ref{T:key}) 
and
$w_\Delta(r) \leq \ord_p(M_{pr-s}) \leq w_\Delta(r)+\frac{N_\Delta}{p-1}.$
We may write 
\begin{eqnarray*}
M_{pr-s} &=& \sum_{i=0}^{N_\Delta}
\gamma^{(p-1)w_\Delta(r)+i} G_{pr-s}^{(i)} + (\mbox{higher terms})
\end{eqnarray*}
where $G_{pr-s}^{(i)}$ in $(\Z_p^*\cap\Q)[A]$, 
and the minimal $0\leq i_p\leq N_\Delta$ such that 
$G_{pr-s}^{(i_p)}\neq 0$.

For $0\leq i\leq N_\Delta$ and for all $p$ large enough we have 
$$
G_{pr-s}^{(i)} = \sum_{Q} u_{Q,p}\prod_{j\in \cG} A_j^{u_{pr-s,j}}
$$
is of degree $w_\Delta(pr-s)+i$ with some $u_{Q,p}\in\Z_p^*\cap \Q$ 
such that $u_{Q,p}\equiv u_Q\bmod p$ for some $u_Q\in\Q$ indepenent of $p$;
where $Q$ ranges over all solutions $u_{pr-s,j}\in \Z_{\geq 0}$ with 
$pr-s = \sum_{j\in\cG} u_{pr-s,j}j$.

\item
Let $\M_\delta$ be as defined in (\ref{E:M_delta}).
For $1\leq k\leq k_\Delta$ we may write
\begin{eqnarray*}
\det(\M^{[N_{\delta,k}]}_{\delta})  & = &\sum_{m=0}^{N_\Delta N_{\delta,k}} \gamma^{(p-1)\sum_{i=0}^{k}
\frac{W_{\Delta_\delta}(i)i}{D(\Delta)} + m} P_{N_{\delta,k},p}^{(m)} +(\mbox{higher terms})
\end{eqnarray*}
for some homogenous polynomial 
$P_{N_{\delta,k},p}^{(m)}$ in $(\Z_p^*\cap\Q)[A]$
and there is a minimal $0\leq m_{k,p}\leq N_\Delta N_{\delta,k}$ such that
$P_{N_{\delta,k},p}^{(m_{k,p})}\neq 0$.

For $0\leq m\leq N_\Delta N_{\delta,k}$ and for $p$ large enough 
$$
P_{N_{\delta,k},p}^{(m)} =\sum_{Q} v_{Q,p} \prod_{j\in \cG} A_j^{m_j}
$$
is of degree $(p-1)\sum_{i=1}^{k}\frac{W_{\Delta_\delta}(i)i}{D(\Delta)}+m$
for some $v_{Q,p}\in\Z_p^*\cap\Q$ 
such that $v_{Q,p}\equiv v_Q\bmod p$ for some $v_Q\in\Q$ indepenent of $p$.
\end{enumerate}
\end{theorem}

\begin{proof}
(0) For simplicity, we shall prove the result under the
assumption that $\cG$ generates $S(\Delta)$
since the hypothesis that $p$ is large enough our argument is not affected.
We also assume each origin-less facet $\delta$ is simplex, if not we replace it by a 
simplex in the triangulation of $\delta$.
The same reason as in the proof of Theorem \ref{T:key}.

\noindent (1)
Fix a origin-less facet $\delta$, and fix an arbitrary vertex residue $R$.
It suffices to prove our assertion for primes $p$ in the residue class of $R$.

The statement that $pr-s\in S(\Delta_\delta)$ is evident under our hypothesis.
Let $\cG_{\Delta_\delta}=\{g_1',\ldots,g_n'\}$ be a primitive generating set
of $\Delta_\delta$.
Write $S:=\cG-\Vert(\Delta_\delta)$.
Recall from Lemma \ref{L:key} there is a unique representation
$R_{\Delta_\delta}(pr-s)=\sum_{j\in S}\alpha_{Rr-s,j} j$ 
such that $(\alpha_{Rr-s,j})_{j\in S}$ is of lowest graded lexicographic order according to $S$.
By Lemma \ref{L:key} we have the following unique vertex representation
$$
pr-s = \sum_{i=1}^{n}\pfloor{\frac{pr_i-s_i}{d_i}}(d_ig_i')+ R_{\Delta_\delta}(pr-s)
=\sum_{i=1}^{n}\pfloor{\frac{pr_i-s_i}{d_i}}(d_ig_i')+
\sum_{j\in S}  \alpha_{Rr-s,j} j
$$
for some $\alpha_{Rr-s,j}\in\Z_{\geq 0}$.
Corresponding to this above representation,
we have a unique summand $Q(pr-s)_\delta$ of $M_{pr-s}$
$$
Q(pr-s)_\delta
=
\gamma^{w_\Delta(s)-w_\Delta(r)}
\prod_{i=1}^{n} \lambda_{\lfloor\frac{pr_i-s_i}{d_i} \rfloor}\prod_{j\in S}\lambda_{\alpha_{Rr-s,j}}
\prod_{i=1}^{n} A_{d_ig_i'}^{\lfloor\frac{pr_i-s_i}{d_i} \rfloor}\prod_{j\in S}A_j^{\alpha_{Rr-s,j}}.
$$
Its monomial part is
$$
\dot{Q}(pr-s)_\delta:=\prod_{i=1}^{n}A_{d_ig_i'}^{\pfloor{\frac{pr_i-s_i}{d_i}}}
\prod_{j\in S}A_j^{\alpha_{Rr-s,j}}.
$$
By the proof of Theorem \ref{T:key} we have
$$w_\Delta(r)\leq \ord_p(Q(pr-s)_\delta) =b_{\cG,\Z}(pr-s)\leq w_\Delta(r)+\frac{N_\Delta}{p-1}.
$$
This proves that there exists a minimal $0\leq i_p\leq N_\Delta$ 
such that $G_{pr-s}^{(i_p)}\neq 0$. Hence $M_{pr-s}$ is of the given form.

\noindent (2)
Fix a vertex residule class $R$ for 
an origin-less facet $\delta$ of $\Delta$.
Write an open facial decomposition $S(\Delta_\delta)= \cup_i \Sigma_i$ 
as in (\ref{E:open-facial}) such that it is ordered in terms of 
their dimension $\dim\Sigma_i\leq \dim\Sigma_{i+1}$.
Write $\M_{\Sigma_i}$ for the principal sub-matrix of $\M_\delta^{[N_{\delta,k}]}$
consisting of entries $M_{pr-s}$ where $r,s\in (\Sigma_i)_{\leq k}$, 
that is $w_\Delta(r), w_\Delta(s)\leq k/D(\Delta)$.
By Wan's {\em boundary decomposition theorem} in \cite[Section 5, Theorem 5.1]{Wan93}
and our estimates in Theorem \ref{T:key}, we know 
for $p$ large enough 
$$
\ord_p(\det\M_\delta^{[N_{\delta,k}]}) = \ord_p(\prod_{i}\det \M_{\Sigma_i}).
$$
This implies that in the formal Leibniz 
expansion of $\det(\M_{\delta}^{[N_{\delta,k}]})$, that is, 
\begin{eqnarray*}
\det(\M_\delta^{[N_{\delta,k}]})
&=&
\sum_{\sigma\in S_{N_{\delta,k}}} \sgn(\sigma) \prod_{r\in S(\Delta_\delta)_{\leq k}}
M_{pr-\sigma(r)}
\end{eqnarray*}
we can restrict to those $\sigma$'s that 
$r\geq_\delta \sigma(r)$ for our purpose of the paper. 
We write $S_{N_{\delta,k}}^*$ for such permutations $\sigma$ in $S_{N_{\delta,k}}$.
This implies that  the hypothesis of 
Theorem \ref{T:key} is satisfied and we can apply its estimates freely.
We claim that for every prime $p$ with $R_{\Delta_\delta}(p)=R$ such that $p$ is large enough,
there is a unique monomial term 
in this formal expansion of $\det(\M_\delta^{[N_{\delta,k}]})$
that does not cancel with the rest and its coefficient is of 
$p$-adic order small enough to lie in our desired range.
We know that $\det((Q(pr-s)_\delta)_{r,s\in S(\Delta_\delta)_{\leq k}})$
is a summand in $\det(\M_{\delta}^{[N_{\delta,k}]})$.
Thus the monomials in the formal expansion of 
\begin{eqnarray*}
\det(\dot{Q}(pr-s)_{\delta})
&=& \sum_{\sigma\in S^*_{N_{\delta,k}}} \sgn(\sigma)\prod_{r\in S(\Delta_\delta)_{\leq k}}
    \dot{Q}(pr-\sigma(r))_\delta\\
&=& \sum_{\sigma\in S^*_{N_{\delta,k}}}\sgn(\sigma)\prod_{r\in S(\Delta_\delta)_{\leq k}}(\prod_{i=1}^{n}A_{d_ig_i'}^{\lfloor\frac{pr_i-\sigma(r_i)}{d_i} \rfloor}
\prod_{j\in S}A_j^{\alpha_{Rr-\sigma(r),j}})
\end{eqnarray*}
are summands  
in the formal expansion of $\det(\M_\delta^{[N_{\delta,k}]})$ (without coefficients).
By Lemma \ref{L:D-2} above, 
there is a unique monomial given by a 
uniquely determined permutation $\sigma_o\in S^*_{N_{\delta,k}}$ 
$$Z=\prod_{r\in S(\Delta_\delta)_{\leq k}} \prod_{j\in S} A_j^{\alpha_{Rr-\sigma_o(r),j}}$$
in the formal expansion of 
$\det(\D_\delta^{[N_{\delta,k}]}):=
\det(\dot{Q}(pr-s))|_{\Vert(\Delta)}$
with the minimal graded lexicographic order with respect to the set $S$.
This gives rise to a unique monomial in $\det(\dot{Q}(pr-s)_\delta)$
equal to 
$$Y:= Z\prod_{r\in S(\Delta_\delta)_{\leq k}}
\prod_{i=1}^{n} A_{d_ig_i'}^{\lfloor\frac{pr_i-\sigma_o(r_i)}{d_i}\rfloor}.
$$ 
By the very definition of vertex representation 
and its uniqueness, we conclude that this monomial is unique 
among all monomial summands in 
the formal expansion of $\det(\M_\delta^{[N_{\delta,k}]})$.
Notice that by Theorem \ref{T:key} ,
\begin{eqnarray*}
\deg(Y)
&=& \sum_{r\in S(\Delta_\delta)_{\leq k}}
(\sum_{i=1}^{n}\lfloor\frac{pr_i-\sigma(r_i)}{d_i}\rfloor + \sum_{j\in S}\alpha_{Rr-s,j})\\
&< & \sum_{r\in S(\Delta_\delta)_{\leq k}} (\sum_{i=1}^{n} \frac{pr_i-\sigma(r_i)}{d_i}
+N_\Delta)\\
&=&(p-1)\sum_{r\in S(\Delta_\delta)_{\leq k}} w_\Delta(r) + N_\Delta N_{\delta,k}\\
&=& (p-1)\sum_{i=0}^{k}\frac{W_{\Delta_\delta}(i)i}{D(\Delta)} + N_\Delta N_{\delta,k}.
\end{eqnarray*}
On the other hand, it is easy to see 
$\deg(Y)\geq  
 (p-1)\sum_{i=0}^{k}\frac{W_{\Delta_\delta}(i)i}{D(\Delta)}. 
$
By the intimate relation between the degree of a term in a (normalized) Fredholm polynomial
and its $p$-adic order of this term as described in 
Remark \ref{R:degree},  we may write 
\begin{eqnarray}
\det(\M^{[N_{\delta,k}]}_{\delta})&=&\sum_{m=0}^{N_\Delta N_{\delta,k}} \gamma^{(p-1)\sum_{i=0}^{k}
\frac{W_{\Delta_\delta}(i)i}{D(\Delta)} + m} P_{N_{\delta,k},p}^{(m)} +(\mbox{higher terms})\\
\prod_{i}\det(\M_{\Sigma_i})&=&\sum_{m=0}^{N_\Delta N_{\delta,k}} \gamma^{(p-1)\sum_{i=0}^{k}
\frac{W_{\Delta_\delta}(i)i}{D(\Delta)} + m} P_{N_{\delta,k},p}^{(m)} +(\mbox{higher terms})
\label{E:degree}
\end{eqnarray}
(we remark that these two higher terms above are typically not the same).
There is a minimal $0\leq m_k\leq N_\Delta N_{\delta,k}$ 
such that
$P_{N_{\delta,k},p}^{(m_k)}$ is nonzero in $(\Z_p^*\cap\Q)[A]$.
\end{proof}

We shall restrict our parameter space from $\cG$ to a subset 
$J$ contained in $\cG-\partial\Delta$ in the following theorem 
to produce Hasse polynomials independent of $p$.

\begin{theorem}
\label{T:Zariski}
 Suppose $\Delta$ is an integral convex polytope of dimension $n$ in $\R^n$
containing the origin.  Suppose $\Delta$ is simplicial at all origin-less facets.
Let $\cG = J\cup V\subseteq \Delta\cap\Z^n$
where $J$ is a subset not intersections with any origin-less facets
of $\Delta$, and $V$ is a disjoint subset containing $\Vert(\Delta)$,
such that $J\cup\Vert(\Delta)$ generates $S(\Delta)$ up to finitely many points.
For an origin-less facet $\delta$ of $\Delta$,
let $\cG_{\Delta_\delta}=\{g_1',\ldots,g_n'\}$ be the primitive generating set
and $\Vert(\Delta_\delta)=\{d_1g_1',\ldots,d_ng_n'\}$
for some $d_i\in\Z_{\geq 1}$.
Let $R\in \prod_{i=1}^{n}(\Z/d_i\Z)^*$ be a vertex residue of prime.
Let $r,s\in S(\Delta_\delta)$ with $r>_\delta s$.
Let $p$ be a prime large enough with $R_{\Delta_\delta}(p)=R$,
then we have
 nonzero polynomials $G_{Rr-s,V}^{i(r,s)}$ and $P_{N_k,R,V}^{(m_k)}$
in $\Q[A]$ with $A=(A_j)_{j\in J}$
independent of $p$ (depending on the vertex residue class $R$) such that
$$
\left\{
\begin{array}{ccc}
G_{Rr-s,V}^{(i)} &\equiv &  G_{pr-s}^{(i)}|_V\bmod p\\
P_{N_{\delta,k},R,V}^{(m)} &\equiv & P_{N_{\delta,k},p}^{(m)}|_V\bmod p
\end{array}
\right.
$$
where the latter is the specialization to $V$ over $\Q$ 
map of polynomials in $\Q[A]$.
There are minimal such $i$ and $m$ so that the two corresponding polynomials are nonzero respectively, denoted by $i(r,s)$ and $m_k$, we have
$0\leq i(r,s)\leq N_\Delta$ and $0\leq m_k\leq N_\Delta N_{\delta,k}$.
Then we have
\begin{eqnarray*}
\ord_p\det(\M_\delta^{[N_{\delta,k}]}) &=& \sum_{i=0}^{k}\frac{W_{\Delta_\delta}(i)i}{D(\Delta)} + \frac{m_k}{p-1}.
\end{eqnarray*}
\end{theorem}

\begin{proof}
Without loss of generality, we give the proof under the hypothesis that 
$\cG$ generates $S(\Delta)$
as we have argued in the proof of Theorem \ref{T:mini-2}.

\noindent (1) Throughout the proof we assume $p$ is a prime with $R_{\Delta_\delta}(p)=R$.
By Theorem \ref{T:mini-2}, there exists $0\leq i\leq N_\Delta$ such that
$G_{pr-s}^{(i)}\neq 0$, our proof in Theorem \ref{T:mini-2} demonstrated that
the minimal $i=w_{\cG,\Z}(pr-s)-w_\Delta(pr-s)$ which is 
$\leq w_{\cG,\Z}(R_{\Delta_\delta}(pr-s)) =w_{\cG,\Z}(R_{\Delta_\delta}(Rr-s))$.
It does not depend on $p$ (only on its vertex residue $R$ relative to $\delta$).
We write the minimal by $i(r,s)$. 
Then we can write 
\begin{eqnarray*}
M_{pr-s} &=& \sum_{i=i(r,s)}^{N_\Delta}
\gamma^{(p-1)w_\Delta(r)+i}G_{pr-s}^{(i)} + (\mbox{higher terms})
\end{eqnarray*}
where $G_{pr-s}^{(i(r,s))}\neq 0$ in $\Q[A]$. 

Let $\dot{P}$ be a monomial in $G_{pr-s}^{(i(r,s))}$ then 
its specialization at its vertices is of the form
$\dot{P}|_{\partial\Delta} =\prod_{j\in \cG-\partial\Delta} A_j^{\alpha_{pr-s,j}}$.
Then by our assumption and by Theorem \ref{T:key} we have
$\sum_{j\in\cG}\alpha_{pr-s,j} \leq w_\Delta(pr-s) + N_\Delta$.
By Theorem \ref{T:approx} that
these exponents $\alpha_{pr-s,j}$ are all bounded and
independent of $p$, and hence
$\dot{P}|_{\partial\Delta}$ and hence their sum
$\dot{G}_{pr-s}^{(i(r,s))}|_{\partial\Delta}$
is independent of $p$.
By Theorem \ref{T:mini-2} coefficients
of monomials in $G_{pr-s}^{i(r,s)}$  lie in 
$\Z_p^*\cap\Q$ whose mod $p$ reduction is independent of $p$.
Therefore, since $\partial\Delta\cap\cG \subseteq V$, 
the specialization $G_{Rr-s,V}^{(i(r,s))} := (G_{pr-s}^{(i(r,s))}|_V \bmod p)$
is independent of $p$ for all large enough $p$ in the vertex residue of $R$ at $\delta$.
In this case we observe 
$\ord_p M_{pr-s} = w_\Delta(r)+\frac{i(r,s)}{p-1}$.

\noindent (2)
Recall from Theorem \ref{T:mini-2} that
for every  $0\leq k\leq k_\Delta$ we have
$$
\det(\M_\delta^{[N_{\delta,k}]})=
\sum_{m=0}^{N_\Delta N_{\delta,k}}\gamma^{(p-1)(\sum_{i=0}^{k}
\frac{W_{\Delta_\delta}(i)i}{D(\Delta)})+m} P_{N_{\delta,k},p}^{(m)} +
(\mbox{higher terms})
$$
for some homogenous polynomial
$P_{N_{\delta,k},p}^{(m)}$ which is a sum of monomials in $(\Z_p^*\cap\Q)[A]$,
where there is a minimal $0\leq m_k\leq N_\Delta N_{\delta,k}$
such that $P_{N_{\delta,k},p}^{(m_k)}\neq 0$.
Consider all representations
$Q: pr-s=\sum_{j\in \cG}u_{pr-s,j}j$
with $u_{pr-s}\in \Z_{>0}$.
Suppose $pr-s=\sum_{j\in \cG} u_{pr-s,j} j$ yields
$$
w_{\cG,\Z}(pr-s)=\sum_{j\in \cG} u_{pr-s,j}\leq w_{\Delta,\Z}(pr-s) + N_\Delta
$$
by Theorem \ref{T:key}.
Then by Theorem \ref{T:approx},
for $j\in\cG -\partial\Delta$ with $w_\Delta(j)<1$ (i.e., $j$ is not on any
origin-less face $\partial\Delta$
of $\Delta$) we have that
$u_{pr-s,j}\leq N_\Delta$
is independent of $p$ (as it is bounded by a constant depending only on $\Delta$).
This proves for every $i\leq N_\Delta$ and $V$ containing $\Vert(\Delta_\delta)$ 
we can find a polynomial $G_{Rr-s,V}^{(i)}$ in $\Q[A]$ such that 
$G_{Rr-s,V}^{(i)}:=( G_{pr-s}^{(i)}|_V\bmod p)$ is independent of $p$.
By definition the specialization 
$P^{(m_k)}_{N_{\delta,k},R,V}:=(P^{(m_k)}_{N_{\delta,k},p}|_V\bmod p)$
is independent of $p$ as well. 
Let $m_k$ be the smallest such $m$ then we have
$$
\det(\M_\delta^{[N_{\delta,k}]})=
\sum_{m=m_k}^{N_\Delta N_{\delta,k}}\gamma^{(p-1)(\sum_{i=0}^{k}
\frac{W_{\Delta_\delta}(i)i}{D(\Delta)})+m} P_{N_{\delta,k},R,V}^{(m)} +
(\mbox{higher terms})
$$
and hence its $p$-adic order is as desired.
\end{proof}

For the rest of this section let notation be as in Theorem \ref{T:allQ}.
Key techniques in the proof of Theorem \ref{T:allQ} in
Section \ref{S:proof} lie in the following theorem.

\begin{theorem}\label{T:nonzero}
Suppose $\Delta$ is an integral convex polytope of 
dimension $n$ in $\R^n$ containing the origin. 
Suppose $\Delta$ is simplicial at all origin-less facets.
Suppose $\cG=J\subseteq (\Delta-\partial\Delta)\cap\Z^n\cup \Vert(\Delta)$
generates the monoid $S(\Delta)$ up to finitely many points.
Let $\Norm: \Q[(A_j)^{1/D(\Delta)}_{j\in \Vert(\Delta)}]\rightarrow \Q[(A_j)_{j\in\Vert(\Delta)}]$
be the norm map only on variables $(A_j)_{j\in\Vert(\Delta)}$.
Let $r,s\in S(\Delta_\delta)_{\leq k}$ with $r>_\delta s$. Let $R$ be a vertex residue
in $\prod_{i=1}^{n}(\Z/d_i\Z)^*$ with respect to an origin-less facet $\delta$ of $\Delta$.
Then there are nonzero polynomials
$(G^*)_{Rr-s}^{(i(r,s))}$ and $(P^*)_{N_{\delta,k},R}^{(m_k)}$ in $\Q[(A_j)_{j\in J}]$ independent of $p$
for any large enough $p$ with $R_{\Delta_\delta}(p)=R$ such that for all 
$a_j\in \Q$ and $a=(a_j)_{j\in J}$ we have
$$
\left\{
\begin{array}{ccc}
(G^*)_{Rr-s}^{(i(r,s))}(a) &\equiv & \Norm(G_{pr-s}^{(i(r,s))}(a))\bmod p \\
(P^*)_{N_{\delta,k},R}^{(m_k)}(a) &\equiv &\Norm (P_{N_{\delta,k},p}^{(m_k)}(a)) \bmod p.
\end{array}
\right.
$$
\end{theorem}

\begin{proof}
Without loss of generality, we give the proof when $J$ generates $S(\Delta)$
as we have argued in the proof of Theorem \ref{T:mini-2} and Theorem \ref{T:Zariski}.
By Theorem \ref{T:approx}, for any $j\in J$,
and since any $pr-s=\sum_{j\in J}u_{pr-s,j}j$
with $\sum_{j\in J}u_{pr-s,j}\leq w_\Delta(pr-s)+N_\Delta$,
we have that
$u_{pr-s,j}$ is independent of $p$ (bounded by a constant depending only on $\Delta$).
By our hypothesis
we are forced to have for every $1\leq i\leq n$ for $j\in \Vert(\Delta)$ we have
$$
u_{pr-s,j} = \pfloor{\frac{pr_i-s_i}{d_i}} =\frac{pk_{i,1}+k_{i,2}}{D(\Delta)}
$$
for some $k_{i,1}\in\Z_{\geq 0}$ and $k_{i,2}\in \Z$.
Since $r,s\in S(\Delta)_{\leq k}\cap Z^o(\Delta_\delta)$,
the constants $k_{i,1}, k_{i,2}$ are also bounded independent of $p$.
By Theorem \ref{T:mini-2} we have $p$-independent $u_Q\in \Q$ and 
$$
G_{pr-s}^{(i(r,s))}
\equiv \sum_{Q} u_Q \prod_{i=1}^{n} A_{d_ig_i'}^{u_{pr-\sigma(r),d_ig_i'}} I^* 
\equiv \sum_{Q} u_Q \prod_{i=1}^{n} A_{d_ig_i'}^{\frac{pk_{i,1}+k_{i,2}}{D(\Delta)}} I^* \bmod p
$$
where 
$
I^*\in(\Z_p\cap\Q)[A]
$ 
is a factor that is independent of $p$.
Thus for any $a_j\in\Q$ and $p$ large enough
\begin{eqnarray*}
\Norm(G_{pr-s}^{(i(r,s))} (a))
&\equiv &\sum_Q u_Q \prod_{i=1}^{n} a_{d_ig_i'}^{pk_{i,1}+k_{i,2}} \Norm(I^*)\\
&\equiv &\sum_Q u_Q \prod_{i=1}^{n} a_{d_ig_i'}^{k_{i,1}+k_{i,2}} \Norm(I^*)\bmod p.
\end{eqnarray*}
On the other hand, 
\begin{eqnarray*}
(G^*)_{Rr-s}^{(i(r,s))} &:=& (\Norm(G_{pr-s}^{(i(r,s))})\bmod (A_{d_1g'_1}^p-A_{d_1g'_1},\ldots,A_{d_ng'_n}^p-A_{d_ng'_n}))\\
&=&\sum_Q u_Q \prod_{i=1}^{n}A_{d_ig'_i}^{k_{i,1}+k_{i,2}} \Norm(I^*)
\end{eqnarray*}
is clearly independent of $p$ and it lies in $\Q[(A_j)_{j\in J}]$.
Notice that $(G^*)_{Rr-s}^{(i(r,s))}(\hat{a})\equiv \Norm(G_{pr-s}^{(i(r,s))}(\hat{a}))\bmod p$ for all $a\in\Q^{|J|}$ and $\hat{a}$ is the Teichm\"uller lift of $a\bmod p$ component-wise.

Similar argument follows for $(P^*)_{N_{\delta,k},R}^{(m_k)}$ by letting
$$
(P^*)_{N_{\delta,k},R}^{(m_k)}:=(\Norm(P_{N_{\delta,k},p}^{(m_k)})
\bmod (A_{d_1g'_1}^p-A_{d_1g'_1},\ldots,A_{d_ng'_n}^p-A_{d_ng'_n}).
$$
Then
$(P^*)_{N_{\delta,k},R}^{(m_k)}$ and $(G^*)_{Rr-s}^{(i(r,s))}$ both lie in $\Q[(A_j)_{j\in  J}]$ 
are nonzero and are independent of $p$ as desired.
\end{proof}

\subsection{Dwork trace formula for generic families}
\label{S:DworkTraceFormula}

We shall define certain $p$-adic Dwork space (of infinite dimension) below and a compact operator
$\varphi:=\alpha_a$ on them.
Fix an integral convex polytope $\Delta$ of dimension $n$ in $\R^n$.
Let $b>0$ be a real number and let
\begin{eqnarray}\label{E:Dworkspace}
\cD_\Delta(b) &=& \{
\sum_{v\in S(\Delta)}c_vx^v
 \in \bar\Q_p[[x]] | \lim_{|v|\rightarrow \infty} p^{\frac{bw_\Delta(v)}{p-1}}|c_v|_p
=0
\}
\end{eqnarray}
with Gauss norm
$
||\sum_{v\in S(\Delta)} c_v x^v|| = \sup_{v\in S(\Delta)} p^{bw_\Delta(v)/(p-1)} |c_v|_p.
$
Notice that if $b\leq b'$ then we have
$\cD_\Delta(b')\subseteq \cD_{\Delta}(b)$.
One can check that $\cD_\Delta(b)$ is an affinoid algebra,
it is complete with respect to the Gauss norm.

Let $\cG$ be a subset of $\Delta\cap\Z^n$ and let
$f=\sum_{j\in \cG}a_jx^j$ in $\bar\Q[x_1,\ldots,x_n,1/x_1\cdots x_n]$.
Let $\bar{f}=\sum_{j\in \cG} \bar{a_j} x^j$ be its reduction in $\F_q[x_1,\ldots,x_n,1/x_1\cdots x_n]$
where $q=p^a$.
Let $\chi_k: \F_{q^k} \rightarrow \C_p^*$ be an additive character, in particular, we set
$\chi_k(\bar{f}) = \zeta_p^{\Tr_{\F_{q^k}/\F_p} (\bar{f})} = E(\gamma\Tr_{\F_{q^k}/\F_p}(\bar{f}))$
where $\gamma$ is a root of $\log E(x)=0$ of $p$-adic order equal to $1/(p-1)$.
Define the exponential sum
$$
S_k(\bar{f}) = \sum_{x\in \F_{q^k}^n}\chi_k(\bar{f}(x)) =\sum_{x\in \F_{q^k}^n}
E(\gamma \Tr_{\F_{q^k}/\F_p}(\bar{f}(x)))
$$
which lies in $\Z[\zeta_p]$.

Let $\tau\in \Gal(\Q_q/\Q_p)$ that lifts the Frobenius element
in $\Gal(\F_q/\F_p)$ defiend by $\tau(x)=x^p$, and
we extend it uniquely to $\Gal(\Q_p(\gamma)/\Q_p)$
by setting $\tau(\gamma=\gamma$.
Note that on any Teichmuller lifting $\hat{c}$
in $\Q_q$ we have $\tau(\hat{c}) = \hat{c}^q$. Let
$$
F_{[1]}(\bar{f},x) :=\prod_{j\in \cG} E(\gamma\hat{a_j}x^j) = \sum_{v\in \Z^{+}(\cG)} F_v x^v
$$
be the Dwork's splitting function.
For $v\in \Z^+(\cG)$ we have in $(\Q\cap\Z_p)[\gamma]$
$$
F_v = \sum_{(u_{v,j})}(\prod_{j\in \cG}\lambda_{u_{v,j}})(\prod_{j\in \cG}{\hat{a}_j}^{u_{v,j}})
$$
if there exists a representation
$v=\sum_{j\in \cG}u_{v,j} j$ and $u_{v,j}\in\Z_{\geq 0}$; otherwise we have $F_v=0$.
In particular it is important to observe
that if $v$ is of the form $v=pr-s$ then
$$
F_{pr-s}=F_{pr-s}(A)|_{A=\hat{a}}
$$
where $F_{pr-s}(A)$ was defined in (\ref{E:F_v}).

Let
$
F_{[a]}(\bar{f},x):=\prod_{i=0}^{a-1} \tau^{i} F_{[1]}(\bar{f},x^{p^i})
$
where $\tau$ acts trivially on variable $x$.
Notice $F_v$
lies in $(\Z_p\cap \Q)[\gamma][\{\hat{a}\}]$,
and recall from Theorem \ref{T:mini-2}
 that $\ord_pF_v\geq w_\Delta(v)/(p-1)$ and hence
$F_{[1]} (\bar{f},x) \in \cD_\Delta(\frac{1}{p-1})$ and similarly we
conclude that $F_{[a]}(\bar{f},x)\cD_\Delta(\frac{p}{q(p-1)})$.
Let $\psi_p: \cD_\Delta(b)\rightarrow \cD_\Delta(b)$ define by
$\psi_p(\sum c_v x^v) = \sum c_{vp}x^v$, so we have
$\psi(\cD_\Delta(b))\subseteq \cD_\Delta(bp)\subset \cD_\Delta(b)$.
Define Dwork operators as
\begin{eqnarray*}
\alpha_1(\bar{f}) &:=&\psi_p \circ F_{[1]}(\bar{f},x).\\
\alpha_a(\bar{f}) &:=&\underbrace{\alpha_1(\bar{f})\circ \cdots \circ\alpha_1(\bar{f})}_{a}
=\psi_q \circ F_{[a]}(\bar{f},x).
\end{eqnarray*}
Notice that $\alpha_1(\bar{f})$ is $\tau^{-1}$-linear on $\bar\Q_p$ and is
acting trivially on $\Q_p[\gamma]$, while $\alpha_a(\bar{f})$ is linear on $\bar\Q_p$.
Thus we have
$\alpha_a$ gives an endomorphism on $\cD_\Delta(\frac{p}{q(p-1)})$.
From now on we will not specific and only say that
exists $b\in \R_{>0}$ small enough so that
$\alpha_a$ is an endomorphism on $\cD_\Delta(b)$.
The following theorem can be found in \cite{AS89}).

\begin{theorem}[Dwork, Adolphson-Sperber]\label{T:Dwork}
Let $\Delta$ be an integral  polytope of dimension $n$ in $\R^n$ containing the origin.
Let $\bar{f}$ be a regular Laurent polynomial over $\F_q$
with $\Delta(\bar{f})=\Delta$ and supported on a subset $\cG$ of $\Delta\cap\Z^n$.
Then $\alpha_a$ is a  compact operator
acting on the $p$-adic Banach space
$\cD_\Delta(b)$ for some small enough $b>0$. Let $A_a(f)$ be the
nuclear matrix of $\alpha_a$ with respect of any Banach basis $\ccB_\Delta$ of $\cD_\Delta$,
then
the following is a polynomial of degree $V_\Delta$ for general $\Delta$:
\begin{eqnarray}
L^*(\bar{f};T)^{(-1)^{n-1}}
&=&\det(1-T\alpha_a(\bar{f})|\cD_\Delta(b))^{\mu^n}
= \det(1-TA_a(f))^{\mu^n}
\nonumber\\
&=& \prod_{i=0}^{n}\det(1-Tq^iA_a(f))^{(-1)^i\binom{n}{i}}
\end{eqnarray}
where $g(T)^\mu=g(T)/g(qT)$ for any rational function $g(T)$.

Let $\M=(M_{pr-s})_{r,s\in S(\Delta)}$ be defined as in (\ref{E:M}).
 Let $\ccB_\Delta:=\{\gamma^{w_\Delta(v)}x^v\}_{v\in S(\Delta)}$
of $\cD_\Delta$ be our basis then for any $f=\sum_{j\in\Delta} a_jx^j$
we have $A_1(f)=\M(f)=(M_{pr-s}(\hat{a}))_{r,s\in S(\Delta)}$.
\end{theorem}

\section{Fredholm determinant and rigidly nuclear matrices}
\label{S:transform}

This section develops new analytic tools
in studying Fredholm determinants of Dwork operators.
Our major new results here lie in Theorems \ref{T:transform} and \ref{T:blockrigid}.
This lay core foundation for our proofs in Section \ref{S:proof}.
Let $\Sigma$ be a countable set
that is partially ordered by a $\Q_{\geq 0}$-valued weight function $w(-)$.
For any $r,r'\in \Sigma$ we have $r\leq r'$ if $w(r)\leq w(r')$. 
For any $k\in \Q_{\geq 0}$, 
let $N_k:=|\Sigma_{\leq k}|$ be the cardinality of all $r\in \Sigma$ with 
$w(r)\leq k$. For example the set $S(\Delta)$ ordered by $w_\Delta(-)$.
A sequence $(\beta_r)_{r\in \Sigma}$ is strictly increasing if
$\beta_r \lneq \beta_{r'}$ for any $r\lneq r'$.

\subsection{Rigid transformations}

We refer the reader to \cite{Ser62} for Dwork and Serre's theory of
completely continuous maps and Fredholm determinants.
Let $F$ be any local field over $\Q_p$ and let $R$ be a Tate algebra
over $F$ with Gauss norm.
In particular,
let $\Delta$ be an integral convex polytope in $\R^n$ of dimension $n$,
let $\cG$ be a subset of $\Delta\cap\Z^n$ and let
$R=F\pscal{A}$ for variables $A=(A_j)_{j\in \cG}$.
For any strictly convergent power series
$h=\sum_{i}c_i A^i\in R$ with $c_i\in F$ we
have $\ord_p(h) = \inf_{i} \ord_p c_i$ and for any $q=p^a$
we write $\ord_q(h)=\ord_p(h)/a$, the Gauss norm on $R$ is
given by $||h|| = q^{-\ord_q(h)}=\sup_{i}(|c_i|_p)$.
Consider any Banach orthonormizable $R$-modules $E$ and $E'$,
denoted by $\cC(E,E')$ the set of completely continuous $R$-linear map
from $E$ to $E'$. We say that a matrix $M$ over $R$ is {\em nuclear}
if there exist a Banach orthonormizable $R$-module $E$
and a map $u$ in $\cC(E,E)$ such that
$M$ is the matrix of $u$ with respect to an orthonormal basis.
Write here $\M = (m_{r,s})_{r, s\in \Sigma}$ for a matrix over $R$
subindiced by $\Sigma$.
Then $\M$ is nuclear if and only if
$\lim_{w(r)\rightarrow \infty}\inf_{s} \ord_p (m_{r, s}) = \infty$
or equivalently $\lim_{w(r)\rightarrow\infty}\sup_{s}||m_{r,s}||=0$.

\begin{lemma}\label{L:skew}
\begin{enumerate}
\item
If  $\{E_i\}_{i\in \Z/a\Z}$ is a family of orthonormizable Banach $R$-modules.
Set $E = E_0\oplus \cdots \oplus E_{a-1}$,
equipped with the supremum norm, that is for $v=(v_0,\ldots,v_{a-1})$ in $E$
one has $||v||=\max_{i}||v_i||$ where
$||\cdot||$ are the norms on $E$ and $E_i$'s, respectively. Let $u_i\in \cC(E_i,E_{i+1})$
and set $u \in \cC(E,E)$ such that $u|_{E_i} = u_i$.
Then
$$
\det(1-(u_{a-1}\cdots u_1u_0)T^a) = \det(1-uT).
$$
\item
Let $(M_0,M_1,\ldots,M_{a-1})$ be an $a$-tuple of nuclear matrices
over $R$. Set
$$
\vec{M}_{[a]}
=
\left(
\begin{array}{ccccc}
0 & & & \cdots& M_{a-1}\\
M_0&0& & &\vdots \\
& M_1 & & &\\
\vdots&&\ddots& &\\
0&\cdots&& M_{a-2} & 0
\end{array}
\right).
$$
Then
$
\det(1-(M_{a-1}\cdots M_1M_0)T^a)=\det(1-\vec{M}_{[a]} T)$.
\end{enumerate}
\end{lemma}
\begin{proof}
By the definition and remark preceding the lemma,
it suffices to prove Part (1).
By \cite[page 77, Corollaire 3]{Ser62} we have
$\det(1-uT) = \exp(-\sum_{s=1}^{\infty} \Tr(u^s)T^s/s)$.
Notice that for any $s\in \Z_{\geq 1}$ the trace
$\Tr((u_{i+a-1}\cdots u_{i+1}u_i)^s)$ is independent of
$i\in \Z/a\Z$. As $\Tr(u^s)=0$ unless $a|s$, we have
\begin{eqnarray*}
\det(1-uT) &=& \exp(-\sum_{s=1}^{\infty} \Tr(u^{as})T^{as}/(as ))\\
&=& \exp(-\sum_{s=1}^{\infty} \sum_{i\in\Z/a\Z} \Tr((u_{i+a-1}\cdots u_{i+1} u_i)^s)T^{as}/(as))\\
&=& \exp(-\sum_{s=1}^{\infty} \Tr((u_{a-1}\cdots u_1u_0)^s)T^{as}/s )\\
&=&\det(1-(u_{a-1}\cdots u_1u_0) T^a).
\end{eqnarray*}
This finishes the proof.
\end{proof}

We denote by $g:R\rightarrow R$ a map that is an automorphism on the local field
$F$ over $\Q_p$ and has $g(A_j)=A_j^p$ for all $j\in \cG$. 
For a matrix $\M$ we write $\M^g$ for $g$ action on each entry of $\M$.

\begin{proposition}
\label{P:transform}
Let $\M=(m_{r,r'})_{r,r'\in \Sigma}$ be an (infinite) nuclear matrix with coefficients in
$R$. Let $g$ be as defined above and write
$\det(1-T\cdot \M^{g^{a-1}}\cdots \M^g \M) = C_0+C_1T+\cdots$ in $R[[T]]$.
Fix $k\in \Z_{\geq 1}$. 
Denote by $\cA$ the set of all $N_k\times N_k$ submatrices of $\M$
contained in the first $N_k$ rows of $\M$, and denote by $\ccB$
the set of all other $N_k\times N_k$ submatrices of $\M$.
Set $t_\cA = \inf_{W\in\cA} \ord_p\det W$ and
$t_\ccB =\inf_{W\in\ccB}\ord_p\det W$.
Consider the following conditions
\begin{enumerate}
\item[(i)] $\ord_q C_{N_k} < \frac{t_\cA + t_\ccB}{2}$ and $t_\cA < t_\ccB$
\item[(ii)] $\ord_p\det \M^{[N_k]} < \frac{t_\cA + t_\ccB}{2}$
\item[(iii)]  $\ord_q C_{N_k} =\ord_p\det \M^{[N_k]}$
\end{enumerate}
Then $(i)\Leftrightarrow (ii) \Rightarrow (iii)$.
\end{proposition}

\begin{proof}
Notice that for any $f\in R$ we have $||g(f)||=||f||$.
Since $\ord_p \M^{[N_k]} \geq t_\cA$, we notice that
$(ii)$ is equivalent to
\begin{eqnarray}\label{E:(ii)}
\ord_p\det \M^{[N_k]} < \min(\frac{t_\cA+t_\ccB}{2}, t_\ccB).
\end{eqnarray}
It is clear that $(\ref{E:(ii)})\Rightarrow (i)$. It remains to show
$(i)\Rightarrow (\ref{E:(ii)})\Rightarrow (iii)$ below.
Apply Lemma \ref{L:skew} to $M_i:=\M^{g^i}$
we have
$\det(1-T\M^{g^{a-1}} \cdots \M^g \M)
=\det(1-T \vec{\M}_{[a]})
=C_0+ C_1T^a+C_2T^{2a}+\cdots+ C_K T^{a N_k}+\cdots$,
where
$C_K = \sum_{\N} (-1)^{a N_k}\det \N
$
and the sum ranges over all principal $a N_k\times a N_k$ submatrices $\N$ of $\vec{\M}_{[a]}$.
Let $\N$ be such a matrix, and let $\N_s$ the intersection of $\N$ and $\M^{g^s}$
as submatrices of $\vec{\M}_{[a]}$ for all $0\leq s\leq a-1$. It is easy to
see that
$
\det \N = (-1)^{(a-1)k}\prod_{s=0}^{a-1}\det \N_s
$
or $0$ depending on whether every $\N_s$ is a $N_k\times N_k$ matrix or not.
So we may assume that every $\N_s$ is a $N_k\times N_k$ matrix.

Consider $\N_s$ as a submatrix of $\M^{g^s}$ from now on.
Define two disjoint subset $X$ and $Y$ of $\Z/a\Z$ below:
$$
X:=
\{
s\in \Z/a\Z|  (\N_s)^{g^{-s}} \in \cA -\{\M^{[N_k]}\}
\}
$$
and
$$
Y:=
\{
s\in \Z/a\Z| (\N_s)^{g^{-s}} \in \ccB
\}.
$$
Since $\N$ is principal, the set of columns of $\N_s$ as a subset in $\Z_{\geq 1}$ is
exactly the same as the set of the rows of $\N_{s-1}$.
Consequently, if $s\in X$ then $s-1\in Y$.
Let $Y'=\{s\in\Z/a\Z| s+1\in X\}$
 and $Z= \Z/a\Z - X\cup Y$.
Then $\Z/a\Z$ is the disjoint union of
$X\cup Y', Y-Y'$ and $Z$.
If $s\in X$ (then $s-1\in Y'$)
then $\ord_p(\det \N_s \det \N_{s-1} ) \geq t_\cA + t_\ccB$;
If $s\in Y-Y'$ then
$\ord_p(\det \N_s) \geq t_\ccB$;
If $s\in Z$ then $\ord_p(\det\N_s) = \ord_p \det \M^{[N_k]}$.
Therefore,
\begin{eqnarray}\label{E:14}
\ord_q \det \N &\geq & \min(\frac{t_\cA + t_\ccB}{2}, t_\ccB, \ord_p \det \M^{[N_k]})
\end{eqnarray}
and hence
\begin{eqnarray}\label{E:15}
\ord_q C_{N_k} &\geq & \min(\frac{t_\cA+t_\ccB}{2}, t_\ccB, \ord_p \det \M^{[N_k]}).
\end{eqnarray}

There is a unique $N$ with $X=Y=\emptyset$, denote it by
$\ccN$, then we have $\ord_q\det \ccN= \ord_p \det \M^{[N_k]}$.
If $N\neq \ccN$, then $X$ or $Y-Y'$ is nonempty and hence from
(\ref{E:(ii)}) and the derivation of (\ref{E:14})
we see that $\ord_q \det \N > \ord_p \det \M^{[N_k]}$.
This proves that (\ref{E:(ii)})$\Rightarrow$(iii). It is clear that (i) implies (\ref{E:(ii)})
by combining (\ref{E:15}) and (i).
\end{proof}

\begin{theorem}[Rigid transform]
\label{T:transform}
Let $\M=(m_{r,r'})_{r,r'\in \Sigma}$ be a nuclear matrix with coefficients in $R$.
Let $g$ be defined as above, and write
$\det(1-T\cdot \M^{g^{a-1}}\cdots \M^g \M) = C_0+C_1T+\cdots$ in $R[[T]]$.
Let $(\beta_r)_{r\in \Sigma}$
be a strictly increasing sequence  in $\Q_{\geq 0}$ 
(i.e., $\beta_r\lneq \beta_{r'}$ if $r\lneq r'$)
such that
$\lim_{r\rightarrow \infty}\beta_r=\infty$
and
$\beta_r \leq \inf_{s\in \Sigma} \ord_p(m_{r,s})$ be
a lower bound in $p$-adic order for each row of $\M=(m_{r,s})_{r,s\in \Sigma}$.

Suppose for any $r^+ \in \Sigma_{>k}$ and $r^0\in \Sigma_k$ we have
\begin{eqnarray*}
\sum_{r\in \Sigma_{\leq k}} \beta_r \leq
&
\ord_p \det(\M^{[N_k]})
& <
\sum_{r\in \Sigma_{\leq k}} \beta_r + \frac{\beta_{r^+}-\beta_{r^0}}{2},
\end{eqnarray*}
then we have
$
\ord_q C_{N_k} =\ord_p \det \M^{[N_k]}.
$
Furthermore,
\begin{eqnarray*}
\NP_q(\det(1-T\; \M^{g^{a-1}} \cdots \M^g \M))^{[N_k]} &=&\NP_p(\det(1-T \M^{[N_k]})).
\end{eqnarray*}
\end{theorem}
\begin{proof}
Let $t_\cA$ and $t_\ccB$ be as in Proposition \ref{P:transform}.
Then by the very definitions we have
\begin{eqnarray*}
\sum_{r\in \Sigma_{\leq k}} \beta_r + \frac{\beta_{r^+}-\beta_{r^0}}{2} &\leq &
\min(\frac{t_\cA+t_\ccB}{2}, t_\ccB)
\end{eqnarray*}
By our hypothesis we have
$$
\ord_p \det\M^{[N_k]}  <  \sum_{r\in S(\Delta)_{\leq k}}\beta_r + \frac{\beta_{r^+}-\beta_{r^0}}{2}
\leq \min(\frac{t_\cA+t_\ccB}{2}, t_\ccB).
$$
So that we can apply Proposition \ref{P:transform} and have
$\ord_qC_{N_k}=\ord_p \det \M^{[N_k]}$.
The last statement follows immediately.
\end{proof}

We call a nuclear matrix $\M$ over Tate algebra $R$ satisfying the
hypothesis of Theorem \ref{T:transform}
{\em rigidly nuclear}.  Below we shall give an important class of
rigidly nuclear matrices that is ubiquitous in generic setting.

\begin{proposition}[Rigidly nuclear criterion]
\label{P:rigid}
Let $R$ be a $p$-adic Tate algebra and
let $\M=(m_{r,s})_{r,s\in \Sigma}$ be an infinite matrix over $R$,
and let $(\beta_r)_{r\in \Sigma}$ be a strictly increasing sequence
in $\Q_{\geq 0}$ (i.e., $\beta_r\lneq \beta_{r'}$ if $r\lneq r'$);
and $\lim_{r\rightarrow\infty}\beta_r=\infty$,
satisfying the following condition:
For each $r\in S(\Delta)$
we have $\beta_r\leq \ord_p m_{r,s}\leq \beta_r+\varepsilon_{r,p}$
for all $s\in\Sigma$, and $\beta_r$ is independent of $p$
and $\lim_{p\rightarrow\infty}\varepsilon_{r,p}=0^+$.
For any $k\in \Q_{\geq 0}$ we have
$\det \M^{[N_k]} = \sum_{r\in \Sigma_{\leq k}}\beta_r
+ \epsilon_{N_k}$ 
where $\epsilon_{N_k}\geq 0$ and $\epsilon_{N_k}\rightarrow 0^+$
as $p\rightarrow\infty$.
Then $\M$ is rigidly nuclear for $p$ large enough.
\end{proposition}

\begin{proof}
The first two  conditions imply that $M$ is nuclear.
Let notations be as in Theorem \ref{T:transform},
then one notices that for $p$ large enough we always have that
$\epsilon_{N_k}<
\frac{\beta_{r^+}-\beta_{r^0}}{2}$
for any $r^+\in \Sigma_{>k}$ and $r^0\in \Sigma_{k}$,
since the latter is
independ of $p$ and
is positive rational by our hypothesis. Hence it is rigidly nuclear.
\end{proof}

\begin{remark}
For $L$ function of exponential sums of $\bar{f}$ over $\F_{p^a}$,
we compute the characteristic polynomial of the
Dwork operator $\alpha_a(\bar{f}) =
\alpha_1(\bar{f})\circ \cdots \circ \alpha_1(\bar{f})$ (as we see in Section \ref{S:dwork}).
Since the operator $\alpha_1(\bar{f})$ is semi-linear,
the characteristic polynomial of $\alpha_1(\bar{f})$ is
{\em not} directly related to that of $\alpha_a(\bar{f})$.
This is one of the core difficulties here in computing $L$-function of exponential sums of $\bar{f}$,
which reflects exactly the same difficulties when computing the
 zeta function over $\F_{p^a}$.

Let $\M=A_1(f)$ such that
$\NP_p(\det(1-T\;A_1(f))) = \HP_\infty(\Delta)$, then
this matrix is rigidly nuclear so $\NP_p(\det(1-T\;A_1(f)))=\NP_q(\det(1-T\; A_a(f)))$,
 by Proposition \ref{P:rigid}. Hence this recovers \cite[Theorem 2.4]{Wan93}.
\end{remark}

For the case $\Delta$ is smooth or asymptotically smooth simplex then
we shall use the combination of Theorem \ref{T:transform}
and Proposition \ref{P:rigid}
to show that for generic $f$  in $\A(\Delta^o,\bar\Q)$
the Newton polygon of the Fredholm determinant of
$\alpha_1(\bar{f})$ is equal to that of $\alpha_a(\bar{f})$, and
of $\GNP(\Delta,\bar\F_p)$.
For non-simplex $\Delta$, one needs a asymtotic facial decomposition
theorem we shall prepare below in Theorem \ref{T:blockrigid}.

\subsection{Rigid transformations in block form}

When $\Delta$ is smooth or asymptotically smooth but not simplex, one needs to
develop an asymptotic facial decomposition theory, that generalize the
facial decomposition theory innitiated by Wan \cite[Section 5]{Wan93}.
The following theorem lays the foundation in this aspect.

\begin{theorem}[Rigid transform in block form]
\label{T:blockrigid}
Let $\Sigma=\cup_{i=1}^{N}\Sigma_i$ be a partition of 
the countable set $\Sigma$ partially ordered by the weight function $w(-)$.
Let $\Sigma_{\leq k}$ consist of all $r\in \Sigma$ with weight $w(r)\leq k$.
For any $k\in\Q_{\geq 0}$, write  $N_{\Sigma_i,k}:=|(\Sigma_i)_{\leq k}|$.
Let $\M=(m_{r,s})_{r,s\in \Sigma}$
be a nuclear matrix over  $R$, and $\M^{[N_k]}=(m_{r,s})_{r\in\Sigma_{\leq k}, s\in\Sigma_{\leq k}}$.
For any $1\leq i,j \leq N$
let $\M_{\Sigma_i,\Sigma_j}=(m_{r,s})_{r\in \Sigma_i,s\in\Sigma_j}$.
Let $(\beta_r)_{r\in \Sigma}$
be a strictly increasing sequence in $\Q_{\geq 0}$
(i.e., $\beta_r \lneq \beta_{r'}$ if $w(r)\lneq w(r')$)
with $\lim_{r\rightarrow\infty} \beta_r=\infty$
and such that
\begin{enumerate}
\item[(i)] $\M=(\M_{\Sigma_i,\Sigma_j})_{1\leq i,j\leq N}$ is a block form
where each submatrix $\M_{\Sigma_i,\Sigma_i}$ is still nuclear.
\item[(ii)]
For any $r\in \Sigma$ we have $\ord_p m_{r,s}\geq \beta_r$ for all $s\in \Sigma$.
\item[(iii)]
For any $1\leq i\leq N$, and for any $r\in \Sigma_i$, we have
$$
\begin{array}{l}
\ord_p(m_{r,s})
\left\{
\begin{array}{ll}
\gneq  \beta_r & \mbox{if $s\in \Sigma_j$ with $j\neq i$},\\
\geq \beta_r    & \mbox{if $s\in \Sigma_i$},\\
<  \beta_r +\varepsilon_{r,p} &\mbox{if $s\in \Sigma_i$},
\end{array}
\right.
\end{array}
$$
 where $\varepsilon_{r,p}\rightarrow 0^+$ if $p\rightarrow \infty$.
\item[(iv)]
Suppose for every $i$ and $p$ large enough
$$\ord_p(\det \M_{\Sigma_i,\Sigma_i}^{[N_{\Sigma_i,k}]})
= \sum_{r\in \Sigma_{\leq k}\cap \Sigma_i} \beta_r + \epsilon_{N_{\Sigma_i,k}}$$
for some
$
\epsilon_{N_{\Sigma_i,k}}\geq 0$ and $\epsilon_{N_{\Sigma_i,k}}
\rightarrow 0^{+}$ as $p\rightarrow \infty.
$
\end{enumerate}

Then $\M_{\Sigma_i,\Sigma_i}$ and $\M$ are all rigidly nuclear,
and  the followings are all equal
\begin{equation*}
\left\{
\begin{array}{rcl}
\NP_p(\det(1-T\M)^{[N_k]})
&=&\prod_{i=1}^{N}\NP_p(\det(1-T\M_{\Sigma_i,\Sigma_i}^{[N_{\Sigma_i,k}]}))\\
\NP_q(\det(1-T\M^{g^{a-1}}\cdots \M)^{[N_k]} )
&=&\prod_{i=1}^{N} \NP_q (\det(1-T \M_{\Sigma_i,\Sigma_i}^{g^{a-1}}
\cdots \M_{\Sigma_i,\Sigma_i})^{[N_{\Sigma_i,k}]}).
\end{array}
\right.
\end{equation*}
If we write
\begin{equation*}
\left\{
\begin{array}{rcl}
\det(1-T\M^{g^{a-1}}\cdots\M) &=& \sum_{m=0}^{\infty}C_mT^m,\\
\det(1-T\M_{\Sigma_i,\Sigma_i}^{g^{a-1}}\cdots \M_{\Sigma_i,\Sigma_i})
&=&\sum_{m=0}^{\infty}C_{i,m}T^m,
\end{array}
\right.
\end{equation*}
then for each  $k\geq 1$ and for $p$ large enough
\begin{equation*}
\left\{
\begin{array}{lll}
\ord_q(C_{i,N_{\Sigma_i,k}}) &=& \ord_p(\det \M_{\Sigma_i,\Sigma_i}^{[N_{\Sigma_i,k}]})
=\sum_{r\in\Sigma_{\leq k}\cap\Sigma_i} \beta_r + \epsilon_{N_{\Sigma_i,k}}\\
\ord_q (C_{N_k}) &=&\ord_p(\det \M^{[N_k]}) =
\ord_p( \prod_{i=1}^{N}\det \M_{\Sigma_i,\Sigma_i}^{[N_{\Sigma_i,k}]})
=\sum_{r\in \Sigma_{\leq k}} \beta_r + \epsilon_{N_k}
\end{array}
\right.
\end{equation*}
where $\epsilon_{N_k}=\sum_{i=1}^{N} \epsilon_{N_{\Sigma_i,k}}$.
Moreover, we have for every $k\geq 0$
$$
\ord_q(C_{N_k}) = \sum_{i=1}^{N} \ord_q(C_{i,N_{\Sigma_i,k}}).
$$
\end{theorem}

\begin{proof}
(1)
Our hypothesis implies by Proposition \ref{P:rigid} that
each $\M_{\Sigma_i,\Sigma_i}$ is rigidly nuclear.
Thus by Theorem \ref{T:transform} the first finitely many terms in the expression 
below are equal:
$$
\NP_q(\det(1-T\M_{\Sigma_i,\Sigma_i}^{g^{a-1}} \cdots\M_{\Sigma_i,\Sigma_i}))
=\NP_p(\det(1-T\M_{\Sigma_i,\Sigma_i}));
$$
Write 
$$\det(1-T\M_{\Sigma_i,\Sigma_i}^{g^{a-1}}\cdots \M_{\Sigma_i,\Sigma_i})
=\sum_{m=0}^{\infty}C_{i,m}T^m$$
we have for every $k\geq 0$
$$
\ord_q(C_{i,N_{\Sigma_i,k}}) =\ord_p(\det \M_{\Sigma_i,\Sigma_i}^{[N_{\Sigma_i,k}]})
=\sum_{r\in\Sigma_{\leq k}\cap\Sigma_i} \beta_r + \epsilon_{N_{\Sigma_i,k}}.
$$

(2)
By our hypothesis  $\M$ is nuclear we have
$\det(1-T\M) = 1+c_1T+c_2T^2+\ldots$
where $c_n=\sum_{1\leq r_1\leq \ldots \leq r_n}
\sum_{\sigma\in S_n} \sgn(\sigma) \prod_{i=1}^{n} m_{r_i,r_{\sigma(i)}}$
and
$r_i\in \Sigma_{k_i}$ for some
$1\leq k_i\leq N$.
For any $\sigma\in S_n$
such that
$r_{\sigma(i)} \in \Sigma_{k_i}$
we have
$\ord_p \prod_{i=1}^{n} m_{r_i,r_{\sigma(i)}}\rightarrow\sum_{i=1}^{n} \beta_{r_i}^+$
as $p\rightarrow\infty$
by our hypothesis.
On the other hand, suppose $\sigma\in S_n$ such that there exists at least one
$i$ such that $r_{\sigma(i)}\not\in \Sigma_{k_i}$. We may assume
$r_{\sigma(i)}\in \Sigma_{k_j}$ with $k_j\neq k_i$; then
by our hypothesis for $n=N_k$ for some $k\geq 0$
we have $\ord_p \prod_{i=1}^{n} m_{r_i, r_{\sigma(i)}} \gneq
\sum_{i=1}^{n} \beta_{r_i}$.
Hence for $p$ large enough
the terms in the expansion of Fredholm determinant is dominated by
the ones with $\sigma$'s sending $\Sigma_i$ to $\Sigma_i$ for every $1\leq i\leq N$.
Thus for $p$ large enough and for $k\geq 0$
$$
c_{N_k} = \sum\prod_{\ell=1}^{N}
(\sum_{\sigma\in S_{\gamma_\ell}}\sgn(\sigma)\prod_{i=1}^{\gamma_\ell}m_{r_i,\sigma(r_i)})
+(\mbox{higher terms})
$$
where the first sum ranges over all $\gamma_1,\ldots,\gamma_N \in \Z_{\geq 0}$ such that
$\sum_{\ell=1}^{N}\gamma_\ell = N_k$, and where $r_1,\ldots,r_{\gamma_\ell}\in \Sigma_\ell$.
In the next paragraph we will show that for each $1\leq \ell\leq N$
we have $\Sigma_\ell\cap \Sigma_{\leq k} = \{r_1,\ldots,r_{\gamma_\ell}\}$.

Suppose we have an $N$-tuple $(\gamma_1,\ldots,\gamma_N)\neq
(N_{\Sigma_1,k},\ldots,N_{\Sigma_N,k})$.
Since our hypothesis says that $\beta_r \lneq \beta_{r'}$ if $w(r)\lneq w(r')$ and $r,r'\in \Sigma$,
we have that
$$\sum_{\ell=1}^{N}\sum_{r\in \Sigma_{\leq k}\cap\Sigma_\ell}\beta_r
\lneq \sum_{\ell=1}^{N}\sum_{r\in \Sigma_\ell'}\beta_r $$
where $\Sigma_\ell'$ is a subset of $\Sigma_\ell$ consisting of the $\gamma_\ell$ lowest weight
elements and $\sum_{\ell=1}^{N}\gamma_\ell=N_k$.
This proves that for $p$ large enough
\begin{eqnarray*}
\ord_p c_{N_k}
&=&\ord_p \prod_{i=1}^{N} \det(\M_{\Sigma_i,\Sigma_i}^{[N_{\Sigma_i,k}]})
=\sum_{r\in \Sigma_{\leq k}}\beta_r+\epsilon_{N_k}
\end{eqnarray*}
where $\epsilon_{N_k} =\sum_{\ell=1}^{N} \epsilon_{N_{\Sigma_\ell,k}}$.
By our Proposition \ref{P:rigid} for $p$ large enough  $\M$ is rigidly nuclear.
Thus by Theorem \ref{T:transform} we have
$$
\NP_q(\det(1-T\M^{g^{a-1}} \cdots\M)) = \NP_p(\det(1-T\M));
$$
if write
$$
\det(1-T\M^{g^{a-1}}\cdots\M) =\sum_{m=0}^{\infty}C_mT^m,
$$
then by Theorem \ref{T:transform} and the above
\begin{eqnarray*}
\ord_q C_{N_k}
&=&\ord_p \det(\M^{[N_k]}) = \ord_p c_{N_k} 
=\sum_{r\in \Sigma_{\leq k}}\beta_r+\epsilon_{N_k}.
\end{eqnarray*}
The rest of our statement follows immediately.
\end{proof}

\section{Proof of main theorems and conjectures of Wan}
\label{S:proof}

\subsection{Local and global Hasse polynomials}

We define two types of Hasse polynomials below.
Let $\Delta$ be an integral convex polytope of dimension $n$ in
$\R^n$ containing the origin.
For ease of notation we assume $\Delta$ is simplicial at all origin-less facets.
Let $\Delta=\cup_\delta \Delta_\delta$ be a closed facial triangulation of $\Delta$
as in (\ref{E:division}).
Fix an origin-less simplex facet $\delta$ and let $\cG_{\Delta_\delta}=\{g_1',\ldots,g_n'\}$ be its
primitive generating set, then we have by Proposition \ref{P:disc} that
$\Vert(\Delta_\delta):=\{d_1g_1',\ldots,d_ng_n'\}$ for some $d_1,\ldots,d_n\in\Z_{\geq 1}$.
Fix a vertex residue class $R\in \prod_{i=1}^{n}(\Z/d_i\Z)^*$ as defined in Lemma \ref{L:key}.
Let $Z^o(\Delta_\delta)$ be defined as in (\ref{E:Z^o})
and that is $Z^o(\Delta_\delta) = \{\sum_{i=1}^{n}\alpha_i (d_ig_i')|0\leq \alpha_i< 1\}$.

\subsubsection{Definition of local Hasse polynomial at each $p$}
We shall first define {\em local} Hasse polynomials at prime $p$ in this paragraph.
Let $J$ be a subset of $\Delta\cap\Z^n$ containing $\Vert(\Delta)$
and let $\A^J$ be the space of all Laurent polynomials
$f=\sum_{j\in J}a_j x^j$ parameterized by $(a_j)$'s with
$\Delta(f)=\Delta$. Suppose $J$ generates the monoid $S(\Delta)=C(\Delta)\cap\Z^n$
up to finitely many points.
Let $k\in\Z_{\geq 1}$.
Let $r,s\in Z^o(\Delta_\delta)\cap\Z^n$ and $pr-s\in S(\Delta_\delta)$ (for $p$ large enough).
For any prime $p$ large enough and with $R_{\Delta_\delta}(p)=R$,
let $G_{pr-s}^{(i(r,s))}$ and $P_{N_{\delta,k},p}^{(m_{k})}$
be nonzero polynomials in $\Q[A]$ constructed in Theorems \ref{T:mini-2} and \ref{T:Zariski}.
Define local Hasse polynomials $P_{\A^J,p}$ and $G_{\A^J,p}$
in $\Q[A]$ for each $p$ as follows:
\begin{equation}
\label{E:localhasse}
\left\{
\begin{array}{lll}
P_{\A^J,p} &:= &\prod_{\delta} \prod_{k=1}^{k_\Delta} P_{N_{\delta,k},p}^{(m_{k})};
\\
G_{\A^J,p} &:= &\prod_{\delta} \prod_{r,s\in Z^o(\Delta_\delta)\cap \Z^n}G_{pr-s}^{(i(r,s))};
\end{array}
\right.
\end{equation}
where the product ranges over all origin-less facets
$\delta$  of  $\Delta$.

If $\Delta$ has non-simplicial origin-less facet $\delta$,
then we have a regular triangulation $\delta=\cup_\ell \delta_\ell$ of 
the point configuration $(\delta,\cG\cap \delta)$
with simplices $\delta_\ell$ as in Proposition \ref{P:triangulation},
then we shall replace $\delta$ by $\delta_\ell$ in the above 
definition of $P_{\A^J,p}$ and $G_{\A^J,p}$.

\subsubsection{Definition of global Hasse polynomials}

We shall define {\em global} Hasse polynomials in this paragraph.
Let $J\subseteq \Delta-\partial\Delta$ and 
suppose $J\cup\Vert(\Delta)$ generates $S(\Delta)$ up to finitely many points.
Let $V$ be a subset of $\Delta\cap\Z^n$ containing $\Vert(\Delta)$ and disjoint from $J$.
Let $\A_V^J$ be the space of all Laurent polynomials
$f=\sum_{j\in J}a_j x^j + \sum_{j\in V}c_j x^j$ with parameters $(a_j)_{j\in J}$ and with prescribed
$(c_j)_{j\in V}$ in $\Q$.
Let $G_{Rr-s,V}^{(i(r,s))}$ and $P_{N_{\delta,k},R,V}^{(m_k)}$ be the nonzero polynomials
in $\Q[A]$ for the vertex residue class $R$ and facet $\delta$ defined in Theorem \ref{T:Zariski}.
Define global Hasse polynomials $P_{\A_V^J}$ and $G_{\A_V^J}$ in $\Q[A]$ with $A=(A_j)_{j\in J}$:
\begin{equation}
\label{E:globalhasse}
\left\{
\begin{array}{lll}
P_{\A_V^J} &:= & \displaystyle\prod_{\delta}\prod_{k=1}^{k_\Delta} \prod_{R}P_{N_{\delta,k},R,V}^{(m_k)};
\\
G_{\A_V^J} &:= & \displaystyle\prod_{\delta}\prod_{r,s\in Z^o(\Delta_\delta)\cap \Z^n}
\prod_{R}G_{Rr-s,V}^{(i(r,s))}
\end{array}
\right.
\end{equation}
where the outer product ranges over all origin-less facets
$\delta$ of $\Delta$ and
the inner product ranges over all vertex residues $R$
in $\prod_{i=1}^{n}(\Z/d_i\Z)^*$ corresponding to
each origin-less facet $\delta$ of $\Delta$.

\subsection{Asymptotic facial decomposition theorem}

Let $\cG$ be a subset of integral points in $\Delta$ containing $\Vert(\Delta)$
and generates the lattice cone $S(\Delta)$ over $\Z_{\geq 0}$ up to finitely many points.
As in (\ref{E:M-definition}), for $r,s\in S(\Delta)$ let $M_{pr-s}$
be the normalized Fredholm polynomial supported on $\cG$
defined as in (\ref{E:M-definition}) by
\begin{eqnarray*}
M_{pr-s}&=&\gamma^{w_\Delta(s)-w_\Delta(r)}F_{pr-s}=
\gamma^{w_\Delta(s)-w_\Delta(r)}
\sum_{Q}(\prod_{j\in \cG}
\lambda_{u_{pr-s,j}})(\prod_{j\in \cG}A_j^{u_{pr-s,j}})
\end{eqnarray*}
where $Q$ ranges over all representations
$Q: pr-s=\sum_{j\in \cG} u_{pr-s,j} j$ and $u_{pr-s,j}\in\Z_{\geq 0}$.
Notice that $M_{pr-s}$ lies in
$(\Q\cap\Z_p)[\gamma^{\Z_{\geq 0}}A]$ in variable $A=(A_j)_{j\in \cG}$.
Let $\M$ be the nuclear matrix as 
$
\M = (M_{pr-s})_{r,s\in S(\Delta) }
$
as in (\ref{E:M-4}).
Let $C(\Delta) = \cup_{\delta} C(\Delta_\delta)^o$ and
$S(\Delta)=\cup_{\sigma} S(\Delta_\delta)^o$
be the open facial
subdivision as in (\ref{E:open-facial}).
For simplicity, we define
for any origin-less open faces $\delta, \delta'$ of $\Delta$
that
\begin{eqnarray}
\label{E:M-4}
\M_{\delta,\delta'} &:=&
(M_{pr-s})_{r\in S(\Delta_\delta)^o,s\in S(\Delta_{\delta'})^o}
\end{eqnarray}
and $\M_{\delta}:=\M_{\delta,\delta}$.
Hence
\begin{eqnarray}
\label{E:M-blockform}
\M=(\M_{\delta,\delta'})_{\delta,\delta'\in \cF^o(\Delta)}
\end{eqnarray}
is block matrix where
$\delta,\delta'$ range in the set $\cF^o(\Delta)$
of all open faces of $\Delta$ that do not contain the origin.
For any $f=\sum_{j\in J}a_j x^j + \sum_{j\in V}c_j x^j$ let 
$M_{pr-s}(f)$ be the specialization of the Laurent polynomial
$M_{pr-s}$ in $\Q[(A_j)_{j\in J}]$ at $A_j=a_j$ for all $j\in J$,
and let $\M(f)=(M_{pr-s}(f))$ accordingly.

For any $k\geq 1$ let let $N_k=|S(\Delta)_{\leq k}|$ and
$N_{\delta,k}^o=|S(\Delta_\delta)^o_{\leq k}|$.
Below we shall prove that there is an asymptotic
open facial decomposition theorem for generic Fredholm determinant.
Let $k_\Delta$ be such that 
$V_\Delta=n!\Vol(\Delta)=|S(\Delta)_{\leq k_\Delta}|$.

\begin{theorem}
\label{T:subdivision}
(1)
Let $\Delta$ be an integral convex polytope of dimension $n$ in $\R^n$ containing the origin.
Let $J$ be a subset of $\Delta\cap\Z^n$ containing $\Vert(\Delta)$ and
generates the monoid $C(\Delta)\cap\Z^n$ up to finitely many points,
then let $\A^J$ be as in (\ref{E:localhasse}), and let $\cU_p$ be
defined by $P_{\A^J,p}\cdot G_{\A^J,p}\neq 0$ in $\A^J$.
Let $\M=(M_{pr-s})_{r,s \in S(\Delta)}$ 
be the matrix of normalized Fredholm polynomials supported
on $\cG=J$ defined as above (see also (\ref{E:M-definition})).
For any $p$ large enough
and for any $f\in\cU_p(\bar\Q)$ we have for all $1\leq k\leq k_\Delta$
\begin{eqnarray*}
\NP_q(\det(1-T \M(f)^{\sigma^{a-1}}\cdots \M(f))^{[N_k]})
&=& \prod_{\delta\in\cF^o(\Delta)} \NP_p(\det(1-T \M(f)_{\delta})^{[N_{\delta,k}^o]})
\end{eqnarray*}
where $\cF^o(\Delta)$ is the set of origin-less open faces of $\Delta$.

(2) Let $\Delta$ be an integral convex polytope of dimension $n$ in $\R^n$ containing the origin
and it is simplicial at all origin-less facets.
If $J\subseteq (\Delta-\partial \Delta)\cap\Z^n$ such that 
$J\cup \Vert(\Delta)$ generates the monoid $S(\Delta)$ up to finitely many points.
Let $V$ be a subset of $\Delta\cap\Z^n$ containing $\Vert(\Delta)$ and disjoint from $J$.
Let $\A_V^J$ be as above, and let
$\cU$ be defined by $P_{\A_V^J}\cdot G_{\A_V^J} \neq 0$ in $\A_V^J$ as in (\ref{E:globalhasse}).
Let $\M$ be the matrix of normalized Fredholm polynomials supported
on $\cG=J\cup V$ defined as in (\ref{E:M-definition}).
For every regular
$f\in \cU(\bar\Q)$ and $p$ large enough, we have for all $1\leq k\leq k_\Delta$
\begin{eqnarray*}
\NP_q(\det(1-T \M(f)^{\sigma^{a-1}}\cdots \M(f))^{[N_k]})
&=& \prod_{\delta\in\cF^o(\Delta)} \NP_p(\det(1-T \M(f)_{\delta})^{[N_{\delta,k}^o]}).
\end{eqnarray*}

(3) In both of the above two cases:
If we write
$$\det(1-T\M(f)^{\sigma^{a-1}}\cdots \M(f)) = \sum_{m=0}^{\infty}C_mT^m$$
and 
$$
\det(1-T\M(f)_\delta)=\sum_{m=0}^{\infty}C_{\delta,m} T^m
$$
then we have
$$
\ord_q C_{N_k} = \sum_{\delta\in \cF^o(\Delta)} \ord_p C_{\delta,N_{\delta,k}^o}.
$$
\end{theorem}

\begin{proof}
(1) 
We shall verify Theorem \ref{T:blockrigid} applies to $\Sigma=S(\Delta)$.
Let $S(\Delta)=\cup_i \Sigma_i$ be the open facial subdivision as in 
(\ref{E:open-facial}). 
Let $r,s\in S(\Delta_\delta)$ for some $\delta$ and let 
$r>_{\delta} s$ then 
for $p$ large enough we have
$pr-s \in \Sigma_i=S(\Delta_\delta)$. By Theorem \ref{T:mini-2}
$$
w_\Delta(r)\leq  \ord_p(M_{pr-s})
=
b_{\cG,\Z}(pr-s)\leq w_\Delta(r)+\frac{N_\Delta}{p-1}.
$$
We claim that the hypothesis of Theorem \ref{T:blockrigid} is satisfied for the matrix $\M(f)$
for all $f\in \cU_p(\bar\Q)$.
For every $r\in S(\Delta_\delta)$  let $\beta_r:=w_\Delta(r)$.
Let the set $S(\Delta)$ be (partially) ordered by the weight $w_\Delta(\cdot)$.
Let $r,s\in S(\Delta_\delta)$ with $r\geq_\delta s$ or $r\in S(\Delta_\delta)^o$ and $s\in S(\Delta)$.
Write $a=(a_j)_{j\in J}$ for the parameterizing coefficients of $f=\sum_{j\in J}a_j x^j$.
Since the polynomial $G_{\A^J,p}$ has its coefficients all in $\Z_p^*\cap\Q$ and 
$G_{\A^J,p}(a)\neq 0$ in $\bar\Q$, we have that 
for $p$ large enough $G_{\A^J,p}(a)\in \bar\Z_p^*\cap\bar\Q$
and hence by Theorem \ref{T:key}
$$
\ord_p M_{pr-s}(f) = w_\Delta(r) + \frac{i(r,s)}{p-1}.
$$
This  implies
$$
\beta_r\leq \ord_p M_{pr-s}(f) \leq \beta_r + \frac{N_\Delta}{p-1}
$$
where $\frac{N_\Delta}{p-1}\rightarrow 0^+$ as $p\rightarrow \infty$.
On the other hand, since $P_{\A^J,p}\in (\Z_p^*\cap \Q)[A]$ we have
that $P_{\A^J,p}(a)\in\bar\Z_p^*\cap\bar\Q$ for $p$ large enough.
But since  the polynomial $P_{\A^J,p}$ 
has coeffients all in $\Z_p^*\cap\Q$ 
\begin{eqnarray*}
\sum_{i=0}^{k}\frac{W_{\Delta_\delta}(i)i}{D(\Delta)} 
\leq 
\ord_p\det(\M(f)^{[N_{\delta,k}]}_{\delta})
=
\sum_{i=0}^{k}\frac{W_{\Delta_\delta}(i)i}{D(\Delta)} + \frac{m_k}{p-1}
\end{eqnarray*}
where $m_{k}<nD(\Delta)^2$.
For every closed sub-face $\delta'$ of $\delta$ we also have
\begin{eqnarray*}
\sum_{i=0}^{k}\frac{W_{\Delta_{\delta'}}(i)i}{D(\Delta)} 
\leq 
\ord_p\det(\M(f)^{[N_{{\delta'},k}]}_{\delta})
=
\sum_{i=0}^{k}\frac{W_{\Delta_{\delta'}}(i)i}{D(\Delta)} + \frac{m'_k}{p-1}.
\end{eqnarray*}
Now by (\ref{E:degree})  this implies that
$$
\ord_p\det(\M(f)^{[N^o_{\delta,k}]}_{\delta})
= 
\sum_{i=0}^{k}\frac{W^o_{\Delta_\delta}(i)i}{D(\Delta)} + \frac{m^o_k}{p-1}
$$
for some $0\leq m^o_k\leq m_k$, where 
$W^o(\Delta_\delta)(i):=|S(\Delta_\delta)^o_{\leq i}|$.
Finally, when $\delta'\neq \delta$ or $r<_\delta s$ 
we know by Theorem \ref{T:key}
that the above first inequality become
$\beta_r \lneq \ord_p M_{pr-s}(f)$. This proves our claim;
thus Theorem \ref{T:blockrigid} applies and we have
the desired asymptotic open facial decomposition, and
the relation on their coefficients.

(2)
Use Theorem \ref{T:Zariski} to yield an
analogous proof for the case $f\in \cU(\bar\Q)$.

(3) Follows from Theorem \ref{T:blockrigid}.
\end{proof}

\subsection{Relation to two conjectures of Wan}

For any $f\in \A_V^J(\bar\Q)$, let $A_a(f)$ be the matrix defined
as in Theorem \ref{T:Dwork} and let $A_a(f)_\delta$ be the matrix
with entries $M_{pr-s}$ where $r,s$ range in $S(\Delta_\delta)^o$,
see (\ref{E:M-4}).
Define for every open origin-less face $\delta$ of $\Delta$
\begin{eqnarray*}
\GNP_\infty(\A_V^J,\bar\F_p) &:=&\inf_{\bar{f}} \NP_q(\det(1-T A_a(f)))
\end{eqnarray*}
where $\bar{f}$ ranges over all regular $\bar{f}\in \A_V^J(\bar\F_p)$.
We define similarly
\begin{eqnarray*}
\GNP_\infty(\A^J,\bar\F_p) &:=&\inf_{\bar{f}} \NP_q(\det(1-T A_a(f)))
\end{eqnarray*}
where $\bar{f}$ ranges over all regular $\bar{f}\in \A^J(\bar\F_p)$.
Define a chain-level Hodge polygon as follows:
\begin{eqnarray*}
\HP_\infty(\Delta) &:=& \NP_q(\prod_{i=0}^{\infty} (1-T q^{\frac{i}{D(\Delta)}})^{W_\Delta(i)}).
\end{eqnarray*}

\begin{lemma}\label{L:main}
Let notation be as above.\\
(1) Let $\Delta$ be an integral convex polytope of dimension $n$ in 
$\R^n$ containing the origin.
Let $J$ be a subset of $\Delta\cap\Z^n$ containing $\Vert(\Delta)$
that generates the monoid $C(\Delta)\cap\Z^n$ up to finitely many points.
For any prime $p$ 
let $\cU_p$ be defined by $P_{\A^J,p}\cdot G_{\A^J,p} \neq 0$ in $\A^J$.
Then for every $p$ large enough and every $f\in \cU_p(\bar\Q)$
we have for $1\leq k\leq k_\Delta$ that
\begin{eqnarray*}
\NP_q(\det(1-TA_a(f)^{[N_k]})) &=& \GNP_\infty(\A^J,\bar{\F}_p)^{[N_k]}.
\end{eqnarray*}
Furthermore, we have
\begin{eqnarray*}
\lim_{p\rightarrow\infty}\GNP_\infty(\A^J,\bar{\F}_p)^{[N_k]} &=& \HP_\infty(\Delta)^{[N_k]}.
\end{eqnarray*}

(2)
Let $\Delta$ be integral convex polytope of dimension $n$ 
of dimension $n$ in $\Q^n$ containing the origin,
and it is simplicial at all origin-less facets.
Let $J\subseteq (\Delta-\partial\Delta)\cap\Z^n$
such that $J\cup \Vert(\Delta)$ generates the monoid $C(\Delta)\cap\Z^n$ up to finitely
many points.
Let $V$ be a subset of
$\Delta\cap\Z^n$ containing $\Vert(\Delta)$ disjoint from $J$.
Let $\cU$ be defined by $P_{\A_V^J}\cdot G_{\A_V^J}\neq 0$ in $\A_V^J$.
Then for $f\in \cU(\bar{\Q})$ and for all prime $p$ large enough
we have for $1\leq k\leq k_\Delta$ that
\begin{eqnarray*}
\NP_q(\det(1-TA_a(f)^{[N_k]})) &=& \GNP_\infty(\A_V^J,\bar{\F}_p)^{[N_k]},\\
\lim_{p\rightarrow\infty} \NP_q(\det(1-TA_a(f)^{[N_k]}))
&=& \lim_{p\rightarrow\infty}\GNP_\infty(\A_V^J,\bar{\F}_p)^{[N_k]}
= \HP_\infty(\Delta)^{[N_k]}.
\end{eqnarray*}
\end{lemma}

\begin{proof}
We shall give detailed proof for Part (2).
Fix an origin-less (simplex) facet $\delta$ of $\Delta$.
By Theorem \ref{T:Zariski} and Proposition \ref{P:rigid} for each $f\in \cU(\bar\Q)$
and for $p$ large enough the matrix $(A_1(f))_\delta$ 
is rigidly nuclear and we have for all $1\leq k\leq k_\Delta$
$$
\NP_p(\det(1-T A_1(f)_{\delta})^{[N_{\delta,k}]})
=  
\inf_{\bar{f}}\NP_q(\det(1-T A_a(f)))
=:\GNP_\infty(\A_V^J,\bar\F_p)_{\delta}^{[N_{\delta,k}]}.
$$
Write 
$
\det(1-T A_1(f)_\delta) =
\sum_{i=0}^{\infty}C_{\delta,i}T^i
$
then by the proof of Theorem \ref{T:subdivision} we have
$
C_{\delta,N_{\delta,k}} = \sum_{i=0}^{k}\frac{W_{\Delta_\delta}(i) i}{D(\Delta)} + \frac{m_k^o}{p-1}
$
for some $0\leq m_k^o\leq m_k$.
Consequently,
$
\lim_{p\rightarrow \infty}
C_{\delta,N_{\delta,k}} = \ord_p\sum_{i=0}^{k} \frac{W_{\Delta_\delta}(i) i}{D(\Delta)}.
$

On the other hand, let $S(\Delta)=\cup_{\delta\in\cF^o(\Delta)} 
S(\Delta_\delta)^o$ be the open facial subdivision
as that in (\ref{E:open-facial}) where $\cF^o(\Delta)$
is the set of all open origin-less open faces of $\Delta$.
By Theorem \ref{T:Zariski}, for any $f$ in $\cU(\bar\Q)$
and for $p$ large enough the hypothesis of
Theorem \ref{T:blockrigid} is satisfied
for the block matrix $A_1(f) = (A_1(f)_{\delta,\delta'})$
where $\delta$ and $\delta'$ ranges over all open origin-less faces of $\Delta$;
and  for $\beta_r = w_\Delta(r)$.
Hence,  by Theorem \ref{T:subdivision}, 
for any $f\in \cU(\bar\Q)$ and for $p$ large enough that
\begin{eqnarray*}
\NP_q(\det(1-TA_a(f)^{[N_k]}))
&=&\prod_{\delta\in \cF^o(\Delta)}\NP_p(\det(1-TA_1(f)_{\delta}^{[N_{\delta,k}^o]}))\\
&=&\prod_{\delta\in\cF^o(\Delta)} \GNP_\infty(\A_V^J,\bar\F_p)_{\delta}^{[N_{\delta,k}^o]}\\
&=& \GNP_\infty(\A_V^J,\bar\F_p)^{[N_k]}.
\end{eqnarray*}
Write
$
\det(1-TA_a(f))
= \sum_{m=0}^{\infty}C_m T^m,
$
by Theorem \ref{T:subdivision} and Theorem \ref{T:blockrigid}
we have 
$\ord_qC_{N_k} = 
\sum_{\delta\in \cF^o(\Delta)} \ord_p(C_{\delta,N_{\delta,k}^o}).$
By our result in Theorem \ref{T:subdivision}
$$
\ord_q(C_{N_k})
= 
\sum_{\delta\in\cF^o(\Delta)} 
(\sum_{i=0}^{k}\frac{W^o_{\Delta_\delta}(i) i}{D(\Delta)} + \frac{m_k^o}{p-1})
= \sum_{i=0}^{k}\frac{W_\Delta(i) i}{D(\Delta)}+\frac{\sum_{\delta\in\cF^o(\Delta)}m_k^o}{p-1}.
$$
Comparing this with $\HP_\infty(\Delta)$ defined above we have
$$
\lim_{p\rightarrow\infty} \GNP_\infty(\A_V^J,\bar\F_p)^{[N_k]} = \HP_\infty(\Delta)^{[N_k]}.
$$
This finishes the proof of Part (2).

The proof of Part (1) is analogous of the above, replacing $\A_V^J$ by $\A^J$
and then replacing $\cU$ in $\A_V^J$ by $\cU_p$ in $\A^J$ for every $p$.
\end{proof}

\begin{theorem}[Theorem \ref{T:Wan1.12}]
\label{T:Wan-2}
Let $\Delta$ be an integral convex polytope of dimension $n$ in $\R^n$ containing the origin,
that is simplicial at all origin-less faces.
Let $J$ be a set of integral points in $\Delta$ of weight $w_\Delta(-)<1$
such that $J\cup\Vert(\Delta)$ generates the monoid $S(\Delta)$ up to finitely many points.
Let $V$ be the set of nonzero integral points disjoint from $J$
and includes all vertices in $\Delta$.
Let
$
\A_V^J
$
be the family of Laurent polynomials
$
f(x)=\sum_{ j \in J} a_j x^j + \sum_{j\in V} c_j x^j
$
parameterized by $a_j$'s and with prescribed $c_j\in\Q$ and nonzero at
vertices of $\Delta$.
Then there exists a Zariski dense open subset $\cU$
defined over $\Q$ in $\A_V^J$ such that for every
$f\in \cU(\bar\Q)$
and for all prime $p$ large enough we have
$$\NP(\bar{f})=\GNP(\A_V^J,\bar\F_p)$$
and
$$
\lim_{p\rightarrow \infty} \NP(\bar{f}) = \lim_{p\rightarrow\infty}\GNP(\A_V^J,\bar\F_p)
=\HP(\Delta).
$$
\end{theorem}

\begin{proof}
By Remark \ref{R:main}
 we may assume all our $f$ are regular with respect to $\Delta$.
The result of Adolphson-Sperber and Dwork in Theorem \ref{T:Dwork}
shows that
\begin{eqnarray}\label{E:L}
L^*(\bar{f},T)^{(-1)^{n-1}}
&=& \prod_{i=0}^{n}{\det}(1-Tq^i A_a(f))^{(-1)^i\binom{n}{i}}
\end{eqnarray}
is a polynomial of degree $V_\Delta:=n!\Vol(\Delta)$ for regular Laurent
polynomial $f\in \A_V^J(\bar{\Q})$
for general $\Delta$.
The proof for arbitrary $\Delta$ is only a modification of this
one,  so we restrict our proof under the hypothesis that $\Delta$ is general.
Let $\cU$ be defined by $P_{\A_V^J}\cdot G_{\A_V^J}\neq 0$ in $\A_V^J$.
By Lemma \ref{L:main} (2) we have for any $f\in \cU(\bar\Q)$ and for $p$ large enough that
\begin{eqnarray*}
\NP_q(\det(1-TA_a(f)))^{[V_\Delta]} &=& \GNP_\infty(\A_V^J,\bar\F_p)^{[V_\Delta]}
\end{eqnarray*}
and hence by (\ref{E:L}) we have $\NP(\bar{f}) = \GNP(\A_V^J,\bar\F_p)$ and
\begin{eqnarray*}
\lim_{p\rightarrow\infty}\NP(\bar{f}) =
\lim_{p\rightarrow\infty} \GNP(\A_V^J,\bar\F_p)=
\HP(\Delta).
\end{eqnarray*}
This proves our theorem.
\end{proof}

\begin{theorem}[Theorem \ref{T:Wan1.11}]
\label{T:Wan1.11-2}
Let $\Delta$ be an integral contex polytope of dimension $n$ 
in $\R^n$ containing the origin.
Let $\Vert(\Delta)\subseteq J \subseteq \Delta\cap \Z^n$
and $J$ generates the monoid $C(\Delta)\cap\Z^n$ up to finitely many points. 
For every prime $p$ large enough,  
let $\cU_p$ be defined by $P_{\A^J,p}G_{\A^J,p}\neq 0$ in 
$\A^J$. Then for every $f \in \cU_p(\bar\Q)$ we have
$\NP(\bar{f}) = \GNP(\A^J,\bar\F_P)$.
Furthermore, we have
$$
\lim_{p\rightarrow\infty}\GNP(\A^J,\bar\F_p) = \HP(\Delta).
$$
\end{theorem}
\begin{proof}
This proof is analogous to that of Theorem \ref{T:Wan-2}.
One reduces the proof to that over the chain level.
By Lemma \ref{L:main} (1) above, 
for every prime $p$ large enough and $f\in \cU_p(\bar\Q)$
then  
$$
\NP(\bar{f}) = \GNP(\A^J,\bar\F_p).
$$
Since by Lemma \ref{L:main}(1) we have for all $1\leq k\leq k_\Delta$,
$$
\lim_{p\rightarrow\infty}\GNP_\infty(\A^J,\bar\F_p)^{[N_k]}
=\HP_\infty(\Delta)^{[N_k]},
$$
we conclude that 
$\lim_{p\rightarrow \infty} \GNP(\A^J,\bar\F_p) = \HP(\Delta).$
This proves our assertion.
\end{proof}

\subsection{Generic affine toric hypersurfaces are ordinary}

We shall discuss an immediate application of our  
Theorem \ref{T:Wan1.11} in algebraic geometry.
Let $\Delta$ be an integral convex hull of dimension $n$ in $\R^n$ containing the origin.
Let $\T_{\Delta}$ be the space of all affine toric hypersurfaces $V(f)$
with $f=\sum_{j\in\Delta\cap\Z^n} a_j x^j$
parameterized by all $(a_j)$'s and $a_j\neq 0$ when $j$ is a vertex in $\Delta$.
Namely $V(f)$ is the embedding of ${f=0}$ in the $n$-torus $(\G_m)^n$.
For any $V(f)$ in $\T_{\Delta}(\bar\Q)$, namely $f\in\bar\Q[x_1,\ldots,x_n,1/(x_1,\ldots,x_n)]$
we write $V(\bar{f})$ for its reduction at a prime of $\bar\Q$
over $p$. 
We define that a toric hypersurface $V(f)$ {\em regular}
if the Laurent polynomial $x_{n+1} f$ in variables $x_1,\ldots,x_{n+1}$ is regular
as defined in Section \ref{S:intro}.
This is equivalent to that for every face $\Delta'$ (of any dimension) of $\Delta$
the system $f_{\Delta'}=x_1\frac{\partial{f_{\Delta'}}}{\partial{x_1}}
=\cdots = x_n\frac{f_{\Delta'}}{\partial x_n}=0$ has no solution in
$(\G_m)^n$ where $f_{\Delta'}$ is the restriction of $f$ to $\Delta'$.

The Hodge polygon of $V(f)$ for any $f$ with $\Delta(f)=\Delta$
is defined as the Hodge polygon
of the Laurent polynomial $x_{n+1}f$ in $\R^{n+1}$.
Given $\Delta$ of dimension $n$ in $\R^n$ let $(1,\Delta)$ be
its embedding in $\R^{n+1}$. For every $k\in \Z_{\geq 0}$
let $K_\Delta(k) =|(k,k\Delta)\cap\Z^{n+1}|$.
Then the {\em Hodge number} of any toric hypersurface $V(f)$ with $\Delta(f)=\Delta$
is defined by
$$
h_\Delta(k) =\sum_{i=0}^{\infty} (-1)^i \binom{n+i}{i} K_\Delta(k-i),
$$
and let $h_\Delta(k)=0$ if $k\geq n+1$.
Write $V_\Delta:=n!\Vol(\Delta)$ then $V_\Delta=\sum_{k=0}^{n} h_\Delta(k)$.
For any regular $f$ in $\T_\Delta(\F_q)$, it is a main result of
\cite{AS89} and \cite{DL91} that the zeta function of the affine 
toric hypersurface $V(\bar{f})$ is for the following form
\begin{eqnarray}
\label{E:P_f}
\Zeta(V(\bar{f}),T) &=& \prod_{i=0}^{n-1} (1-q^iT)^{(-1)^{n-i}\binom{n}{i+1}}P_f(T)^{(-1)^n}
\end{eqnarray}
whose key polynomial factor $P_f(T)\in 1+T \Z[T]$ is of degree $V_\Delta-1$.
We have
$$
L^*(x_{n+1}f;T) =(P_f(qT)(1-T))^{(-1)^{n}}.
$$
For this reason we define $\NP(V(\bar{f}))$ to be the normalized $p$-adic Newton
polygon of $L^*(x_{n+1}f;T)^{(-1)^n}$, or equivalently if we denote by $\Delta'$ the convex hull
of the origin and $(1,\Delta)$ in $\R^{n+1}$, then
we have
\begin{eqnarray}
\label{E:NP}
\NP(V(\bar{f})) &=& \NP(\Delta',\bar\F_p)
\end{eqnarray}
If one writes
$
P_f(T) = \prod_{i=1}^{V_\Delta-1}(1-\alpha_i(f) T)
$
then its (complex) absolute value $|\alpha_i(f)|\leq q^{\frac{n-1}{2}}$.
The precise Archimedean weights of these reciprocal roots $\alpha_i(f)$
are also determined by \cite{DL91}.
Our Corollary \ref{C:Wan1.11} describes the generic distribution of normalized
$|\alpha_i(f)|_p$.

\begin{proof}[Proof of Corollary \ref{C:Wan1.11}]
Let $\Delta\cap\Z^n=\cup_i \Delta_i$ be its unimodular triangulation
where each $\Delta_i$ is a unimodular simplex (see Section \ref{S:triangulation}).
Embedding $\iota: \Delta\rightarrow (1,\Delta)$ from $\R^n$ to $\R^{n+1}$.
Let $\Delta_i'$ be the convex hull of the origin and
$(1,\Delta_i)$ in $\R^{n+1}$, it is clearly a unimodular simplex by its
definition.
Let $\Delta'$ be the convex hull of the origin and $(1,\Delta)$ in $\R^{n+1}$,
then $\Delta'=\cup_i \Delta_i'$ is a unimodular triangulation. 
This proves that each $C(\Delta_i')$ is unimodular, that is,
every point in $C(\Delta_i')\cap\Z^{n+1}$
is generated by $\Delta_i'=(1,\Delta_i)$.
Thus $C(\Delta')\cap\Z^{n+1}=\cup_i C(\Delta_i)\cap\Z^{n+1}$ is
generated by $\Delta'\cap\Z^{n+1}=(1,\Delta)\cap\Z^{n+1}$.
This proves our hypothesis of Theorem \ref{T:Wan1.11} is satisfied
and hence we can conclude that
$$
\lim_{p\rightarrow\infty} \GNP(\Delta',\bar\F_p) = \HP(\Delta').
$$
By (\ref{E:P_f}) (due to \cite{AS89} and \cite{DL91}) and by (\ref{E:NP}) we have
that 
$\GNP(\Delta',\bar\F_p) = \GNP(\T_\Delta,\bar\F_p)$.
A direct computation shows that $\HP(\Delta')=\HP(\T_\Delta)$ by their definition.
This implies that
$\lim_{p\rightarrow \infty}\GNP(\T_\Delta,\bar\F_p) = \HP(\T_\Delta)$;
and hence 
$\GNP(\T_\Delta,\bar\F_p) = \HP(\T_\Delta)$ for $p$ large enough.
\end{proof}

\subsection{Proof of Theorem \ref{T:allQ}}

Let $\Delta$ be an integral convex polytope of dimension $n$
in $\R^n$ containing the origin, whose origin-less facets are simplex.
Suppose $J\subseteq (\Delta-\partial\Delta)\cap\Z^n \cup\Vert(\Delta)$
generates the monoid $C(\Delta)\cap\Z^n$ up to finitely many points,
and let $\A^J$ be as above.
Let $(P^*)_{N_k,R}^{(m_k)}$ be nonzero polynomials
in $\Q[A]$ defined as in Theorem \ref{T:nonzero}, where $A=(A_j)_{j\in J}$.
We define the following global Hasse polynomials for $\A^J$ in $\Q[A]$ 
where $A=(A_j)_{j\in J}$:
\begin{eqnarray}
\label{E:Hasse}
P^*_{\A^J} &:=& \prod_\delta\prod_{k=1}^{k_\Delta}\prod_{R\in \prod_{i=1}^{n} (\Z/d_i\Z)^*}(P^*)_{N_{\delta,k},R}^{(m_k)}
\end{eqnarray}
where $\delta$ ranges over all origin-less facets of $\Delta$.
Notice that the polynomial $P^*_{\A^J}$ in $\Q[A]$
is nonzero and independent of $p$ or its vertex residue classes
$R=R_{\Delta_\delta}(p)$.

\begin{proposition}
\label{P:simple}
Let notations be as above.
Let $P^*_{\A^J}$ be the nonzero polynomial in $\Q[A]$ defined above in
(\ref{E:Hasse}).
Let $\cU$ be defined by $P^*_{\A^J}\neq 0$ in $\A^J$.
Then for any $f\in \cU(\Q)$ and for any prime $p$ large enough we have
$\ord_p C_{N_k} = \ord_p \det A_1(f)^{[N_k]}$ for all $0\leq N_k\leq V_\Delta$;
and
\begin{eqnarray*}
\NP_p(\det(1-T\; A_1(f)))^{[V_\Delta]}
&=&\GNP_\infty(\A^J,\F_p).
\end{eqnarray*}
Furthermore, for such $f$ we have
\begin{eqnarray*}
\lim_{p\rightarrow \infty}\NP_p( \det(1-T\; A_1(f))^{[V_\Delta]}
&=& \lim_{p\rightarrow\infty}\GNP_\infty(\A^J,\F_p)^{[V_\Delta]}
= \HP_\infty(\Delta)^{[V_\Delta]}.
\end{eqnarray*}
\end{proposition}

\begin{proof}
Let $a=(a_j)_{j\in J}$.
Let $\cU$ be defined by $P^*_{\A^J}\neq 0$.
This defines a Zariski dense open subset defined over $\Q$.
Let $f\in \cU(\Q)$, we consider its matrix $A_1(f)
=(M_{pr-s}(a)_{r,s\in S(\Delta)}$ with support on $\cG=J$.
We will verify below that $A_1(f)$ is rigidly nuclear by the criterion Proposition \ref{P:rigid}
and Theorem \ref{T:nonzero}.
For every $r\in S(\Delta)$, let $\beta_r := w_\Delta(r)$.
Obviously we have $0\leq \beta_r < \beta_s\leq \cdots$ whenever $w_\Delta(r)< w_\Delta(s)$.
For any $s\in S(\Delta)$ we have for $p$ large enough
$\ord_p M_{pr-s} \geq  b_{\cG,\Z}(pr-s) + \frac{i(r,s)}{p-1}$
with the equality holds if and only if $G_{\Delta}(a)\neq 0$.
Hence we have for $f\in \cU(\bar\Q)$ that
$\beta_r\leq \ord_p M_{pr-s}(a) \leq \beta_r+\frac{N_\Delta}{p-1}$ and
$\lim_{p\rightarrow\infty} \frac{N_\Delta}{p-1}=0^+$.
On the other hand, we have $P_{\A^J}(a)\neq 0$.
Hence $\ord_p(\det \M^{[N_k]}(a))=
\sum_{i=0}^{k}\frac{W_\Delta(i)i}{D(\Delta)}+\frac{m_k}{p-1}$
where $\frac{m_k}{p-1}\rightarrow 0$ as $p\rightarrow \infty$.
Thus we may apply Proposition \ref{P:rigid} and
Theorem \ref{T:transform}  to $\M$ and have
$\ord_pC_{N_k} = \ord_p(\det A_1^{[N_k]}(a))=
\sum_{i=0}^{k}\frac{W_\Delta(i)i}{D(\Delta)}+\frac{m_k}{p-1}$.
Furthermore,
$\NP_p(\det(1-T A_1(f)^{[V_\Delta]}))
=\GNP(\A^J;\F_p)$
as desired.
\end{proof}

\begin{proof}[Proof of Theorem \ref{T:allQ}]
This theorem follows from Proposition \ref{P:simple} immediately.
\end{proof}


\begin{thebibliography}{ZZZZZZZ}

\bibitem[AS89]{AS89}
{\sc Alan Adolphson; Steven Sperber:}
Exponential sums and Newton polyhedra:
Cohomology and Estimates.
{\it Annal. of Math.}
{\bf 130},
no. 2 (1989).
367--406.

\bibitem[Bar02]{Bar02}
{\sc Alexander Barvinok:}
{\it A course in convexity.}
Graduate Studies in Mathematics, Vol 54.
American Mathematical Society, 2002.

\bibitem[Ber75]{Ber75}
{\sc D.N. Bernstein:}
The number of roots of a system of equations,
{\it Functional Analysis Appl.}
9 (1975), 1--4.


\bibitem[BF07]{BF07}
{\sc R. Blache; E. Ferard:}
Newton stratification for polynomials: the open stratum.
{\it J. Number Theory} {\bf 123} (2007),
no.2, 456--472.

\bibitem[Bla11]{Bla11}
{\sc Regis Blache:}
Newton polygon for character sums and Poincare series.
{\it Int. J. Number Theory}
{\bf 7} (2011),
no. 6,
1519--1542.

\bibitem[BGR]{BGR}
{\sc S. Bosch; U. Guntzer; R. Remment:} Non-Archimedean analysis,
{\it Grundlehren der Mathematischen Wissenschaften} Vol. {\bf
261}, Springer-Verlag, Berlin, 1984.

\bibitem[BGHMW]{BGHMW}
{\sc Winfried Bruns, Joseph Gubeladze, Martin Henk, Alexander Martin,
Robert Weismantel:}
A counterexample to an integral analogue of
Carath\'eodory's theorem.
{\it Crelle.}

\bibitem[CLO]{CLO}
{\sc David Cox, John Little, Donal O'Shea:}
Ideals, Varieties, and Algorithms.
Undergraduate Texts in Mathematics.
Springer-Verlag.

\bibitem[DRS10]{DRS10}
{\sc Jes\'us De Loera; J\"org Rambau; Francisco Santos:}
{\it Triangulations:
Structures for Algorithms and Applications.}
Algorithms and Computations in Mathematics, vol 25.
Springer. 2010.

\bibitem[Del80]{Del80}
{\sc P. Deligne:}
La conjecture de Weil, II.
{\it Publ. Math. IHES}, {\bf 52} (1980), 137--252.

\bibitem[DL91]{DL91}
{\sc J. Denef; F. Loeser:}
Weights of exponentials sums, intersection cohomology, and Newton polyhedra.
{\it Invent. Math.} {\bf 106} (1991), 275--294.


\bibitem[Dw64]{Dw64}
{\sc Bernard Dwork:}
On the zeta function of a hypersurface, II,
{\it Ann of Math.}
{\bf 80} (1964),
227--299.

\bibitem[Elk87]{Elk87}
{\sc Noam D. Elkies:}
The existence of infinitely many supersingular primes for every elliptic curve over
$\Q$,
{\it Invent. Math.} {\bf 89} (1987),
561-568.

\bibitem[Ewa96]{Ewa96}
{\sc G\"uner Ewald:}
{\it Combinatorial Convexity and Algebraic Geometry.}
Graduate Texts in Mathematics {\bf 168}.
Springer, 1996.

\bibitem[Ful93]{Ful93}
{\sc William Fulton:}
{\it Introduction to toric varieties.}
Annals of Mathematics Studies {\bf 131}.
Princeton University Press, 1993.

\bibitem[GKZ08]{GKZ08}
{\sc I. M. Gelfand; M.M. Kapranov; A.V.Zelevinsky:}
{\it Discriminants, Resultants and multi-dimensional determinants},
Birkhauser, Boston. 2008.

\bibitem[Gro64]{Gro64}
{\sc A. Grothendick:}
Foumule de Lefschetz et rationalit\'e des fonctions $L$.
S\'eminaire Bourbaki,
expos\'e 279,
1964/65.

\bibitem[HW97]{HW97}
{\sc Martin Henk and Robert Weismantel:}
{\it The height of minimal Hilbert bases.}
Results in Mathematics, 32, 1997, 298--303.

\bibitem[Hon01]{Hon01}
{\sc S. Hong:}
Newton polygons of $L$ functions associated with exponential sums of polynomials
of degree four over finite fields.
{\it Finite Fields Appl.}
{\bf 7} (2001), no. 1, 205--237.

\bibitem[Hon02]{Hon02}
{\sc S. Hong:}
Newton polygons for $L$-functions of exponential sums of polynomials of degree six over
finite fields.
{\it J. Number Theory}
{\bf 97} (2002), no.2, 368--396.

\bibitem[Kat88]{Kat88}
{\sc Nicholas M. Katz:}
Gauss sums, Kloosterman sums, and monodromy groups,
{\it Annals of mathematics studies} vol {\bf 116},
Princeton University Press, 1988.

\bibitem[Kou76]{Kou76}
{\sc A. G. Kouchnirenko}:
Polyedres de Newton et nombres de Milnor.
{\it Invent. Math} {\bf  32}
(1976),
1--31.

\bibitem[Maz72]{Maz72}
{\sc Barry Mazur:}
Frobenius and the Hodge filtration.
{\it Bull. A.M.S.}
{\bf 78} (1972),
653--667.

\bibitem[Seb90]{Seb90}
{\sc A. Seb\~o}:
Hilbert bases, Caratheodory's Theorem and combinatorial optimization,
in {\it Proc. of the IPCO conference, Waterloo, Canada},
431--455. (1990).

\bibitem[Ser62]{Ser62}
{\sc Jean-Pierre Serre:}
Endomorphismes compl\`etement continus des espaces de Banach $p$-adiques,
{\it Inst. Hautes \'Etudes Sci. Publ. Math.} {\bf 12} (1962), 69--85.

\bibitem[Stu96]{Stu96}
{\sc Bernd Sturmfels:}
Gr\"obner Bases and convex polytopes.
University Lecture Series, vol 8.
{\it American Mathematical Society.}

\bibitem[Stu98]{Stu98}
{\sc Bernd Sturmfels:}
Polynomial equations and convex polytopes.
{\it Amer. Math. Monthly} 105 (1998), no. 10, 907--922.

\bibitem[Wan92]{Wan92}
{\sc Daqing Wan:}
Newton polygons and congruence decompositions of $L$ functions
over finite fields.
{\it Contemporary Math.} {\bf 138} (1992), 221--241.

\bibitem[Wan93]{Wan93}
{\sc Daqing Wan:}
Newton polygons of zeta functions and $L$-functions
{\it Ann. Math.}
{\bf 137}
(1993), 247--293.

\bibitem[Wan04]{Wan04}
{\sc Daqing Wan:} Variation of Newton polygons for $L$-functions
of exponential sums.
{\it Asian J. Math.} {\bf 8}  (2004), 427--474.

\bibitem[Wan08]{Wan08}
{\sc Daqing Wan:}
Lectures on zeta functions over finite fields (Gottingen Lecture Notes).
{\it Higher dimensional geometry over finite fields,}
eds: D. Kaledin and Y. Tschinkel,
IOS Press, 2008.

\bibitem[Yan03]{Yan03}
{\sc Roger Yang:}
{\it Newton polygons of $L$-functions of polynomials of the form
$X^d+\lambda X$},
Finite Fields Appl. {\bf 9} (2003),
no.1, 59--88.

\bibitem[Zhu03]{Zhu03}
{\sc Hui June Zhu:}
$p$-adic variation of $L$ functions of one
variable exponential sums, I. {\it Amer. J. Math.}
{\bf 125} (2003), 669--690.

\bibitem[Zhu04a]{Zhu04a}
{\sc Hui June Zhu:}
Asymptotic variation of $L$ functions of
one-variable exponential sums.
{\it J. Reine Angew. Math.}
{\bf 572} (2004), 219--233.

\bibitem[Zhu04b]{Zhu04b}
{\sc Hui June Zhu:} $L$ functions of exponential sums over one
dimensional affinoids: Newton over Hodge.
{\it Inter. Math. Res.Notices.}, vol 2004,
no. 30 (2004), 1529--1550.



\end{thebibliography}
\end{document}